\documentclass[11pt]{amsart}
\usepackage{amsfonts}
\usepackage{amsmath,amscd}
\usepackage{amsthm}
\usepackage{amssymb}
\usepackage{latexsym}
\usepackage{ulem}
\usepackage{dsfont}
\usepackage{braket}
\usepackage{tikz}
\usepackage{enumitem}   
\usepackage{float}
\usepackage{caption}
\usetikzlibrary{matrix}
\usepackage[linktocpage=true]{hyperref}

\numberwithin{equation}{section}

\usepackage{soul}
\usepackage{color}

\newtheorem{thm}{Theorem}[section]
\newtheorem{cor}[thm]{Corollary}
\newtheorem{lem}[thm]{Lemma}

\newtheorem{prop}[thm]{Proposition}

\newtheorem{example}[thm]{Example}
\newtheorem{defn}[thm]{Definition}
\newtheorem{conj}[thm]{Conjecture}
\newtheorem{rem}[thm]{Remark}

\definecolor{darkspringgreen}{rgb}{0.09, 0.45, 0.27}

\newcommand\numberthis{\addtocounter{equation}{1}\tag{\theequation}}
\numberwithin{equation}{section}
\begin{document}
	
	\newcommand{\Uq}{U_q sl_2}
	\newcommand{\Uqhat}{U_q \widehat{sl}_2}
	\newcommand{\Loop}{\mathcal{L} U_q sl_2}

	\newcommand{\beqa}{\begin{eqnarray}}
		\newcommand{\eeqa}{\end{eqnarray}}
	\newcommand{\thmref}[1]{Theorem~\ref{#1}}
	\newcommand{\secref}[1]{Sect.~\ref{#1}}
	\newcommand{\lemref}[1]{Lemma~\ref{#1}}
	\newcommand{\propref}[1]{Proposition~\ref{#1}}
	\newcommand{\corref}[1]{Corollary~\ref{#1}}
	\newcommand{\remref}[1]{Remark~\ref{#1}}
	\newcommand{\er}[1]{(\ref{#1})}
	\newcommand{\nc}{\newcommand}
	\newcommand{\rnc}{\renewcommand}

    \newcommand{\CWM}{\mathsf{W}_-^{(N)}}
	\newcommand{\CWP}{\mathsf{W}_+^{(N)}}
	\newcommand{\CG}{\mathsf{G}^{(N)}}
	\newcommand{\CtG}{\widetilde{\mathsf{G}}^{(N)}}
	
	\newcommand{\WMu}{\displaystyle{\sum_{k=0}^{N-1} P_{-k}^{(N)}(u) \tW_{-k}^{(N)}}}
	\newcommand{\WPu}{\displaystyle{\sum_{k=0}^{N-1} P_{-k}^{(N)}(u) \tW_{k+1}^{(N)}}}
	\newcommand{\Gu}{\displaystyle{\sum_{k=0}^{N-1} P_{-k}^{(N)}(u)\tG_{k+1}^{(N)}}}
	\newcommand{\tGu}{\displaystyle{\sum_{k=0}^{N-1} P_{-k}^{(N)}(u) \tilde{\tG}_{k+1}^{(N)}}}
	
	\nc{\cal}{\mathcal}
	\nc{\diag}{\mathrm{diag}}
	\nc{\goth}{\mathfrak}
	\rnc{\bold}{\mathbf}
	\renewcommand{\frak}{\mathfrak}
	\renewcommand{\Bbb}{\mathbb}

	\newcommand{\epsp}{\varepsilon_+}
	\newcommand{\epsm}{\varepsilon_-}
	\newcommand{\bepsp}{\overline{\varepsilon}_+}
	\newcommand{\bepsm}{\overline{\varepsilon}_-}

    \newcommand{\bj}{\boldsymbol{\bar{\jmath}}}

	\newcommand{\bt}{{\bf {\mathsf{T}}}}
	\newcommand{\by}{{\bf {\mathsf{Y}}}}
	\newcommand{\id}{\mathrm{id}}
	\nc{\Cal}{\mathcal}
	\nc{\Xp}[1]{X^+(#1)}
	\nc{\Xm}[1]{X^-(#1)}
	\nc{\on}{\operatorname}
	\nc{\ch}{\mbox{ch}}
	\nc{\Z}{{\bold Z}}
	\nc{\J}{{\mathcal J}}
	\nc{\C}{{\bold C}}
	\nc{\Q}{{\bold Q}}
	\nc{\oC}{{\widetilde{C}}}
	\nc{\oc}{{\tilde{c}}}
	\nc{\ocI}{ \overline{\cal I}}
	\nc{\og}{{\tilde{\gamma}}}
	\nc{\lC}{{\overline{C}}}
	\nc{\lc}{{\overline{c}}}
	\nc{\Rt}{{\tilde{R}}}
	
	\nc{\tW}{\normalfont{{\mathsf{W}}}}
	\nc{\tG}{\normalfont{{\mathsf{G}}}}
    \nc{\tK}{\normalfont{{\mathsf{K}}}}
	\nc{\tZ}{\normalfont{{\mathsf{Z}}}}
	\nc{\tI}{{\mathsf{I}}}

	\nc{\tE}{{\mathsf{E}}}
	\nc{\tF}{{\mathsf{F}}}
	\nc{\tx}{{\mathsf{x}}}
	\nc{\tho}{{\mathsf{h}}}
	\nc{\tk}{{\mathsf{k}}}
	\nc{\tep}{{\bf{\cal E}}}

	\nc{\te}{{\mathsf{e}}}
	\nc{\tf}{{\mathsf{f}}}

	\nc{\odel}{{\overline{\delta}}}
	
	\def\pr#1{\left(#1\right)_\infty}  
	
	\renewcommand{\P}{{\mathcal P}}
	\nc{\N}{{\Bbb N}}
	\nc\beq{\begin{equation}}
		\nc\enq{\end{equation}}
	\nc\lan{\langle}
	\nc\ran{\rangle}
	\nc\bsl{\backslash}
	\nc\mto{\mapsto}
	\nc\lra{\leftrightarrow}
	\nc\hra{\hookrightarrow}
	\nc\sm{\smallmatrix}
	\nc\esm{\endsmallmatrix}
	\nc\sub{\subset}
	\nc\ti{\tilde}
	\nc\nl{\newline}
	\nc\fra{\frac}
	\nc\und{\underline}
	\nc\ov{\overline}
	\nc\ot{\otimes}
	
	\nc\ochi{\overline{\chi}}
	\nc\bbq{\bar{\bq}_l}
	\nc\bcc{\thickfracwithdelims[]\thickness0}
	\nc\ad{\text{\rm ad}}
	\nc\Ad{\text{\rm Ad}}
	\nc\Hom{\text{\rm Hom}}
	\nc\End{\text{\rm End}}
	\nc\Ind{\text{\rm Ind}}
	\nc\Res{\text{\rm Res}}
	\nc\Ker{\text{\rm Ker}}
	\rnc\Im{\text{Im}}
	\nc\sgn{\text{\rm sgn}}
	\nc\tr{\text{\rm tr}}
	\nc\Tr{\text{\rm Tr}}
	\nc\supp{\text{\rm supp}}
	\nc\card{\text{\rm card}}
	\nc\bst{{}^\bigstar\!}
	\nc\he{\heartsuit}
	\nc\clu{\clubsuit}
	\nc\spa{\spadesuit}
	\nc\di{\diamond}
	\nc\cW{\mathsf{W}}
	\nc\cG{\mathsf{G}}
    \nc\cK{\mathsf{K}}
    \nc\cZ{\mathsf{Z}}
	\nc\ocW{\overline{\cal W}}
	\nc\ocZ{\overline{\cal Z}}
	\nc\al{\alpha}
	\nc\bet{\beta}
	\nc\ga{\gamma}
	\nc\de{\delta}
	\nc\ep{\epsilon}
	\nc\io{\iota}
	\nc\om{\omega}
	\nc\si{\sigma}
	\rnc\th{\theta}
	\nc\ka{\kappa}
	\nc\la{\lambda}
	\nc\ze{\zeta}
	
	\nc\vp{\varpi}
	\nc\vt{\vartheta}
	\nc\vr{\varrho}
	
	\nc\odelta{\overline{\delta}}
	\nc\Ga{\Gamma}
	\nc\De{\Delta}
	\nc\Om{\Omega}
	\nc\Si{\Sigma}
	\nc\Th{\Theta}
	\nc\La{\Lambda}
	
	\nc\boa{\bold a}
	\nc\bob{\bold b}
	\nc\boc{\bold c}
	\nc\bod{\bold d}
	\nc\boe{\bold e}
	\nc\bof{\bold f}
	\nc\bog{\bold g}
	\nc\boh{\bold h}
	\nc\boi{\bold i}
	\nc\boj{\bold j}
	\nc\bok{\bold k}
	\nc\bol{\bold l}
	\nc\bom{\bold m}
	\nc\bon{\bold n}
	\nc\boo{\bold o}
	\nc\bop{\bold p}
	\nc\boq{\bold q}
	\nc\bor{\bold r}
	\nc\bos{\bold s}
	\nc\bou{\bold u}
	\nc\bov{\bold v}
	\nc\bow{\bold w}
	\nc\boz{\bold z}
	
	\nc\ba{\bold A}
	\nc\bb{\bold B}
	\nc\bc{\bold C}
	\nc\bd{\bold D}
	\nc\be{\bold E}
	\nc\bg{\bold G}
	\nc\bh{\bar{h}}
	\nc\bi{\bold I}
	\nc\bk{\bold K}
	\nc\bl{\bold L}
	\nc\bm{\bold M}
	\nc\bn{\bold N}
	\nc\bo{\bold O}
	\nc\bp{\bold P}
	\nc\bq{\bold Q}
	\nc\br{\bold R}
	\nc\bs{\bold S}
	\nc\bu{\bold U}
	\nc\bv{\bold V}
	\nc\bw{\bold W}
	\nc\bz{\bold Z}
	\nc\bx{\bold X}

	\nc\ca{\mathcal A}
	\nc\cb{\mathcal B}
	\nc\cc{\mathcal C}
	\nc\cd{\mathcal D}
	\nc\ce{\mathcal E}
	\nc\cf{\mathcal F}
	\nc\cg{\mathcal G}
	\rnc\ch{\mathcal H}
	\nc\ci{\mathcal I}
	\nc\cj{\mathcal J}
	\nc\ck{\mathcal K}
	\nc\cl{\mathcal L}
	\nc\cm{\mathcal M}
	\nc\cn{\mathcal N}
	\nc\co{\mathcal O}
	\nc\cp{\mathcal P}
	\nc\cq{\mathcal Q}
	\nc\car{\mathcal R}
	\nc\cs{\mathcal S}
	\nc\ct{\mathcal T}
	\nc\cu{\mathcal U}
	\nc\cv{\mathcal V}
	\nc\cz{\mathcal Z}
	\nc\cx{\mathcal X}
	\nc\cy{\mathcal Y}

	\nc\e[1]{E_{#1}}
	\nc\ei[1]{E_{\delta - \alpha_{#1}}}
	\nc\esi[1]{E_{s \delta - \alpha_{#1}}}
	\nc\eri[1]{E_{r \delta - \alpha_{#1}}}
	\nc\ed[2][]{E_{#1 \delta,#2}}
	\nc\ekd[1]{E_{k \delta,#1}}
	\nc\emd[1]{E_{m \delta,#1}}
	\nc\erd[1]{E_{r \delta,#1}}
	
	\nc\ef[1]{F_{#1}}
	\nc\efi[1]{F_{\delta - \alpha_{#1}}}
	\nc\efsi[1]{F_{s \delta - \alpha_{#1}}}
	\nc\efri[1]{F_{r \delta - \alpha_{#1}}}
	\nc\efd[2][]{F_{#1 \delta,#2}}
	\nc\efkd[1]{F_{k \delta,#1}}
	\nc\efmd[1]{F_{m \delta,#1}}
	\nc\efrd[1]{F_{r \delta,#1}}

	\nc\fa{\frak a}
	\nc\fb{\frak b}
	\nc\fc{\frak c}
	\nc\fd{\frak d}
	\nc\fe{\frak e}
	\nc\ff{\frak f}
	\nc\fg{\frak g}
	\nc\fh{\frak h}
	\nc\fj{\frak j}
	\nc\fk{\frak k}
	\nc\fl{\frak l}
	\nc\fm{\frak m}
	\nc\fn{\frak n}
	\nc\fo{\frak o}
	\nc\fp{\frak p}
	\nc\fq{\frak q}
	\nc\fr{\frak r}
	\nc\fs{\frak s}
	\nc\ft{\frak t}
	\nc\fv{\frak v}
	\nc\fz{\frak z}
	\nc\fx{\frak x}
	\nc\fy{\frak y}
	
	\nc\fA{\frak A}
	\nc\fB{\frak B}
	\nc\fC{\frak C}
	\nc\fD{\frak D}
	\nc\fE{\frak E}
	\nc\fF{\frak F}
	\nc\fG{\frak G}
	\nc\fH{\frak H}
	\nc\fJ{\frak J}
	\nc\fK{\frak K}
	\nc\fL{\frak L}
	\nc\fM{\frak M}
	\nc\fN{\frak N}
	\nc\fO{\frak O}
	\nc\fP{\frak P}
	\nc\fQ{\frak Q}
	\nc\fR{\frak R}
	\nc\fS{\frak S}
	\nc\fT{\frak T}
	\nc\fU{\frak U}
	\nc\fV{\frak V}
	\nc\fZ{\frak Z}
	\nc\fX{\frak X}
	\nc\fY{\frak Y}
	\nc\tfi{\ti{\Phi}}
	\nc\bF{\bold F}
	\rnc\bol{\bold 1}
	
	\nc\ua{\bold U_\A}
	
	\nc\qinti[1]{[#1]_i}
	\nc\q[1]{[#1]_q}
	\nc\xpm[2]{E_{#2 \delta \pm \alpha_#1}}  
	\nc\xmp[2]{E_{#2 \delta \mp \alpha_#1}}
	\nc\xp[2]{E_{#2 \delta + \alpha_{#1}}}
	\nc\xm[2]{E_{#2 \delta - \alpha_{#1}}}
	\nc\hik{\ed{k}{i}}
	\nc\hjl{\ed{l}{j}}
	\nc\qcoeff[3]{\left[ \begin{smallmatrix} {#1}& \\ {#2}& \end{smallmatrix}
		\negthickspace \right]_{#3}}
	\nc\qi{q}
	\nc\qj{q}
	
	\nc\ufdm{{_\ca\bu}_{\rm fd}^{\le 0}}

	
	\nc\isom{\cong} 
	
	\nc{\pone}{{\Bbb C}{\Bbb P}^1}
	\nc{\pa}{\partial}
	\def\H{\mathcal H}
	\def\L{\mathcal L}
	\nc{\F}{{\mathcal F}}
	\nc{\Sym}{{\goth S}}
	\nc{\A}{{\mathcal A}}
	\nc{\arr}{\rightarrow}
	\nc{\larr}{\longrightarrow}
	
	\nc{\ri}{\rangle}
	\nc{\lef}{\langle}
	\nc{\W}{{\mathcal W}}
	\nc{\uqatwoatone}{{U_{q,1}}(\su)}
	\nc{\uqtwo}{U_q(\goth{sl}_2)}
	\nc{\dij}{\delta_{ij}}
	\nc{\divei}{E_{\alpha_i}^{(n)}}
	\nc{\divfi}{F_{\alpha_i}^{(n)}}
	\nc{\Lzero}{\Lambda_0}
	\nc{\Lone}{\Lambda_1}
	\nc{\ve}{\varepsilon}
	\nc{\bepsilon}{\bar{\epsilon}}
	\nc{\bak}{\bar{k}}
	\nc{\phioneminusi}{\Phi^{(1-i,i)}}
	\nc{\phioneminusistar}{\Phi^{* (1-i,i)}}
	\nc{\phii}{\Phi^{(i,1-i)}}
	\nc{\Li}{\Lambda_i}
	\nc{\Loneminusi}{\Lambda_{1-i}}
	\nc{\vtimesz}{v_\ve \otimes z^m}
	
	\nc{\asltwo}{\widehat{\goth{sl}_2}}
	\nc\ag{\widehat{\goth{g}}}  
	\nc\teb{\tilde E_\boc}
	\nc\tebp{\tilde E_{\boc'}}
	
	\newcommand{\LR}{\bar{R}}
		\newcommand{\eeq}{\end{equation}}
	\newcommand{\ben}{\begin{eqnarray}}
		\newcommand{\een}{\end{eqnarray}}

	\setcounter{MaxMatrixCols}{30}
	\newcommand{\h}{\frac{1}{2}}
	\newcommand{\tha}{\frac{3}{2}}
	
	\newcommand{\bep}{\overline{\epsilon}_+}
	\newcommand{\bem}{\overline{\epsilon}_-}
	\newcommand{\bkp}{\overline{k}_+}
	\newcommand{\bkm}{\overline{k}_-}
	\newcommand{\kp}{k_+}
	\newcommand{\km}{k_-}

    \newcommand{\calW}{\mathcal{W}}
    \newcommand{\calG}{\mathcal{G}}

	\newcommand{\ds}{\mathds}
	
	\newcommand{\fu}{{\langle 12 \rangle} }
	\newcommand{\futt}{{\langle 23 \rangle} }
	\newcommand{\sfu}{{\langle 34 \rangle} }
	
	\newcommand{\CE}{\cal{E} }
	\newcommand{\CF}{\cal{F} }
	\newcommand{\CH}{\cal{H} }
	\newcommand{\CW}{\cal{W}}
	
	\allowdisplaybreaks

	\makeatletter
	\def\@textbottom{\vskip \z@ \@plus 1pt}
	\let\@texttop\relax
	\makeatother
	

	\title[Universal TT- and TQ-relations for ${\cal A}_q$]{Universal TT- and TQ-relations\\ via centrally extended $\boldsymbol{q}$-Onsager algebra} 
	\author{Pascal Baseilhac}
	\author{Azat M. Gainutdinov}
	\author{Guillaume Lemarthe}
	\address{Institut Denis-Poisson CNRS/UMR 7013 - Universit\'e de Tours -
		Parc de Grandmont, 37200 Tours, 
		FRANCE}
	\email{pascal.baseilhac@idpoisson.fr, azat.gainutdinov@cnrs.fr, guillaume.lemarthe@idpoisson.fr}

	\begin{abstract}
    Let ${\cal A}_q$ be the alternating central extension of the $q$-Onsager algebra,
    a comodule algebra over  the quantum loop algebra of $sl_2$.
     We first classify one-dimensional representations of $\cal A_q$, and show that spin-$j$ K-operators constructed in~\cite{LBG} act as K-matrices that match with those previously derived in the literature. 
 	Using these K-operators and  K-matrices, we construct universal spin-$j$ transfer matrices generating commutative subalgebras in $\mathcal{A}_q$. Within a technical conjecture, we  derive their fusion hierarchy, the so-called universal TT-relations.
	 On spin-chain representations of $\mathcal{A}_q$, 
      we show how the universal transfer matrices evaluate to spin-chain transfer matrices, and as a result we get explicit TT-relations for all values of spins for auxiliary and quantum spaces,  any inhomogeneities, and general integrable boundary conditions.   In particular, we derive previously conjectured TT-relations (spin-$j$ open XXZ, alternating open spin chains). 
     Using the TT-relations, we show that $n$th {\it local} conserved quantities of the spin-$j$ chains of length $N$ are polynomials of total degree $4Njn$ in two {\it non-local} operators  of the $q$-Onsager algebra. As a result, we give an algorithm of explicit calculation of all conserved quantities (Hamiltonians and higher logarithmic derivatives of the transfer matrix) in terms of spin operators. Furthermore, using the universal TT-relations we derive exchange relations between spin-$j$ Hamiltonians and the two non-local operators, which shows  existence of non-trivial symmetries for special boundary conditions, in the sense that they commute with all Hamiltonian densities. 
     As a yet another application of our universal TT-relations we propose universal T-system, Y-system and universal TQ-relations for ${\cal A}_q$, and as a result, universal TQ for the $q-$Onsager algebra. In view of application to diagonal boundary conditions, we also obtain universal TT- and TQ-relations for a  degenerate version of~$\cal A_q$ known as centrally extended augmented $q$-Onsager algebra.  We finally discuss implications of our results for generalized Gibbs ensemble construction.
	\end{abstract}
	
	\maketitle
	
	
	%
	
	%
	\tableofcontents
	
	\section{Introduction}

\subsection{Background \& motivations}
In the context of quantum integrable spin chains  and their applications to physics at equilibrium,  explicit knowledge of models' transfer matrices, their fusion hierarchy and related Baxter's TQ-relations  has shown  to be sufficient for solving the model under consideration, namely  Hamiltonian's spectrum, see e.g.~\cite{Bax82,B14}.

While the representation-theoretic meaning of the fusion hierarchy~\cite{KNS,kunib} for closed spin-chains and related Baxter's TQ-relations~\cite{FH13} has been studied in details, in contrast the fusion hierarchy for open spin-chains with generic boundary conditions~\cite{MN91,FNR07} and related TQ-relations are  poorly understood. This lack of understanding is mainly due to the fact that the form of the fusion hierarchy and of TQ-relations heavily depends on the quantum spins and boundary conditions chosen\footnote{For the case of the XXZ spin-$\h$ chain with diagonal boundary conditions, see however~\cite{T20,VW20,VW23}.}, and so their representation-theoretic meaning stays elusive. For instance, for certain class of boundary conditions   the approach using algebraic Bethe ansatz leads to homogeneous TQ-relations~\cite{YNZ06,FNR07,CYSW14} that are sufficient to determine the spectrum of the spin chain under consideration. However, for generic boundary conditions the situation is more complicated. In this latter case, one gets the so-called inhomogeneous TQ-relations~\cite{CYSW13,CYSW14,YZYSW15},
and their representation-theoretic interpretation is an open problem, together with the closely related  fusion hierarchy of transfer matrices known as TT-relations.
This motivates development of a universal transfer-matrix approach aiming to obtain a unified version
of TT-relations, called {\it universal TT-relations}. These are supposed to be relations among generating functions -- universal transfer matrices -- of commuting elements in a certain quantum algebra, and such that they would
reduce to the known examples of TT-relations on various spin-chain representations of this quantum algebra. Following the ideas in~\cite{YNZ06}, one can then propose universal TQ-relations in an infinite-spin limit of the universal TT-relations. 

The universal transfer-matrix approach also proves to be useful in applications  to non-equilibrium physics,  where 
the long time average of local observables after a quantum quench is described by a generalized Gibbs ensemble~\cite{RDO08,PSSV11}. It turns out that
an extensive number of (quasi-)local~con\-ser\-ved quantities are required for the analysis \cite{EF06,RDYO07,P11,Poz13,Pos,FE13,PPSA14,IMP15,INWCEP15}. 
 Here, one needs an access to explicit spin-matrix expressions of higher conserved quantities, and while this problem was solved in some spin-$\h$ closed spin chains~\cite{GM95,NF20}, it remains an open problem in the open boundary case. 
 From this perspective, a representation-independent approach to the explicit construction of local conserved quantities besides the Hamiltonian is desirable in order to avoid the complexity inherent to computationally expensive representation-dependent approaches. 
 
 There are two examples that illustrate the representation-independent approach. 
	In \cite{L20},  local conserved quantities are written in terms of `words' built from images of the  fundamental generators $\{A_0,A_1\}$ of the infinite-dimensional Onsager algebra~\cite{Ons44,GR85,P89,AMPT88}. This has been applied  to non-equilibrium dynamics in a transverse field Ising chain  as well as the superintegrable 3-state Potts model, both for periodic boundary conditions; In \cite{NH21}, for the closed XXZ spin chain of length $N$, with or without a twist, 
    all local conserved quantities were studied in terms of `words' built from images of the affine Temperley-Lieb algebra generators $\{\rho,e_i,i=0,...,N\}$.  As pointed out in~\cite{L20}, the approach based on the Onsager algebra is more straightforward, both conceptually and technically, than working with the usual fermionic representation of the model. Also, the approach taken in~\cite{NH21} for the XXZ spin-$\h$ chain overcomes the problem that, 
    besides  the Hamiltonian, all higher local conserved quantities are increasingly  cumbersome expressions in terms of tensor products of Pauli matrices~\cite{GM95,NF20}.
    Besides this progress in identifying all conserved quantities for {\it closed} integrable spin-$\h$ chains, to our knowledge for integrable  open spin chains with arbitrary spins, even for the spin-$\h$ case, 
    an algorithm for an explicit construction of local conserved quantities in terms of spin operators has not been proposed. 

\subsection{Goal \& approach} \label{intro:goal}
The goal of the present paper is the construction of universal transfer matrices and corresponding fusion hierarchy  such that 
\begin{itemize}
\item[(i)] it specializes to all known examples of integrable open spin chains' transfer matrices of XXZ type and corresponding fusion hierarchies, with a clear understanding within the representation theory of quantum algebras;  
\item[(ii)] it provides an algorithm   determining closed form expressions for  all local conserved quantities in terms  of, first, generators of the quantum algebra, and eventually in terms of spin operators. 
\end{itemize}
As far as quantum integrable models with boundary conditions are concerned, the approach to open chains transfer matrix construction starts in
Sklyanin's framework \cite{Skly88} with identifying reflection algebra and the corresponding solutions of reflection equations. More precisely, let $u$, $v$, $u_1$, $u_2$  denote indeterminates and consider R-matrix solutions $R^{(j_1,j_2)}(u)\in \End(V^{(j_1)}) \otimes \End(V^{(j_2)})$ of 
	 the Yang-Baxter equations for all $j_1,j_2\in\h\mathbb{N}$:\footnote{We  use the standard notations: 
		$$
		R^{(j_1,j_2)}_{12}(u)=R^{(j_1,j_2)}(u) \otimes {\mathbb I}, \quad R_{23}^{(j_2,j_3)}(u)= {\mathbb I}\otimes R^{(j_2,j_3)}(u), \quad  R_{13}^{(j_1,j_3)}(u)= (\cal P^{(j_2,j_1)} \otimes {\mathbb I} ) ( \mathbb{I} \otimes R^{(j_1,j_3)}(u)) ( {\cal P}^{(j_1,j_2)} \otimes {\mathbb I}), $$
		where $R^{(j_1,j_2)}(u) \in \End(V^{(j_1)}\otimes V^{(j_2)})$ and  $\cal P^{(j_1,j_2)}$ flips $V^{(j_1)}\otimes V^{(j_2)}$ to $V^{(j_2)}\otimes V^{(j_1)}$ and it acts as identity on $V^{(j_3)}$.}
	\begin{equation}\label{YBj1j2}
		R_{12}^{(j_1,j_2)}(u_1/u_2) R_{13}^{(j_1,j_3)}(u_1) R_{23}^{(j_2,j_3)}(u_2)=R_{23}^{(j_2,j_3)}(u_2)R_{13}^{(j_1,j_3)}(u_1)R_{12}^{(j_1,j_2)}(u_1/u_2) \ ,
	\end{equation}
	and assume they are all   symmetric like in \cite{Skly88} and where $j_k$'s are $sl_2$ spins, e.g.\ solutions coming from the quantum loop algebra $\Loop$ universal R-matrix.
  {\it The family universal spin-$j$ transfer matrix} $\bt^{(j)}(u)$, a certain generating function of mutually commuting elements  in some  $\Loop$-comodule algebra~$B$,   is built from two key ingredients: a {\it spin-$j$  K-operator} and a {\it spin-$j$ dual K-matrix}. Namely, assume there exists a family of {\it spin-$j$  K-operators} 
	\begin{equation}
		{{\cal K}}^{(j)}(u) \in B((u^{-1})) \otimes\End(V^{(j)})  \label{Kg} \ ,\qquad j\in \h \mathbb{N}\ ,
	\end{equation}
that satisfies the reflection algebra 
\begin{align}
	R^{(j_1,j_2)}(u/v) {\cal K}_1^{(j_1)}(u) R^{(j_1,j_2)}(uv) {\cal K}_2^{(j_2)}(v) &= {\cal K}_2^{(j_2)}(v)R^{(j_1,j_2)}(uv){\cal K}_1^{(j_1)}(u)R^{(j_1,j_2)}(u/v)\ ,\label{REKop}
\end{align}
for all $j_1,j_2\in \h \mathbb{N}$, and
where we use the notations  $\cal K_1= \cal K\otimes {\mathbb I}$, $\cal K_2= {\mathbb I}\otimes \cal K$. 
Here, the entries of the K-operator are formal Laurent series in $u$  with coefficients in the algebra $B$ whose defining relations are  extracted via expanding the reflection equation~\eqref{REKop}.
We also introduce {\it spin-$j$ dual K-matrices}
\begin{equation}
 K^{+{(j)}}(u)\in {\mathbb C}((u^{-1})) \otimes \End(V^{(j)}) \ \label{dualKmatjint}
\end{equation}
 satisfying the so-called {\it dual} reflection equations
\begin{align}
	R^{(j_1,j_2)}(v/u) K_1^{+(j_1)}(u) R^{(j_1,j_2)}(1/uvq^2) K_2^{+(j_2)}(v) &= K_2^{+(j_2)}(v) R^{(j_1,j_2)}(1/uvq^2)  K_1^{+(j_1)}(u) R^{(j_1,j_2)}(v/u)\ .\label{REKdual}
\end{align}
	Then, following \cite{Skly88} we introduce the generating functions
	\begin{equation}\label{tgint}
	\bt^{(j)}(u) = \normalfont{\text{tr}}_{V^{(j)}}\bigl(K^{+{(j)}}(u){\cal K}^{(j)}(u)\bigr) \; \in B((u^{-1}))\ ,
	\end{equation}
	 where the trace is taken over the so-called auxiliary space $V^{(j)}$. 
	Using the reflection algebra relations~\eqref{REKop} and dual reflection equation~\eqref{REKdual}, it can be shown that~\cite{Skly88}:
	\beqa\label{eq:T-com-T}
    [\bt^{(j)}(u),\bt^{(j)}(v)]=0 \ ,
	\eeqa
    and thus each $\bt^{(j)}(u)$ provides a generating function for elements  in a commutative subalgebra of $B$. 

\smallskip
A unification of all known integrable open spin chains associated with the  R-matrix based on the quantum loop algebra $\Loop$ is achieved through the identification $B=\mathcal{A}_q$, where $\mathcal{A}_q$ is a central extension of the $q$-Onsager algebra, a non-abelian associative algebra of infinite dimension.  This algebra was introduced in \cite{BS09,BasBel,Ter21}
 in terms of the so-called `alternating' generators
$$\{{\tW}_{-k},{\tW}_{k+1}, {\tG}_{k+1},\tilde{\tG}_{k+1}\; | \; k\in{\mathbb N}\}$$
satisfying the relations \eqref{qo1}-\eqref{qo11}.
At generic integrable boundary conditions, open spin chains (of XXZ type) of any length and with any choice of quantum spins and inhomogeneity parameters provide finite-dimensional representations of 
$\mathcal{A}_q$~\cite{BK07}. This infinite-dimensional algebra is thus a natural candidate for a universal approach to open spin chains with  $\Loop$ R-matrix.

In~\cite{LBG} in this case of $B=\mathcal{A}_q$, we have constructed a family of fused K-operators of spin-$j$ that solves~\eqref{REKop}. The main result of this paper is the universal fusion hierarchy satisfied by corresponding generating functions $\bt^{(j)}(u)$ from~\eqref{tgint}. It is then applied to 
solve the points (i) and (ii) raised above for all known quantum integrable open spin chains of XXZ type with generic boundary conditions. We note that some of these results are based on the Phd thesis~\cite{lem23}.

\subsection{Main results} 
As a preliminary step before addressing the goals (i) and (ii) settled above, in the first part of the paper we construct the \textit{universal transfer matrix} $\bt^{(j)}(u)$ for $\cal A_q$ and study  its properties.
The first key ingredient in the construction of $\bt^{(j)}(u)$ given by~\eqref{tgint} is the spin-$j$ K-operator   $\cal K^{{(j)}}(u)$ for~$\cal A_q$~\cite{LBG} that finds a natural interpretation within a general framework of universal K-matrices for comodule algebras -- an extended version of the framework introduced
	by Appel-Vlaar in \cite{AV20}.
    The second key ingredient in~\eqref{tg} is the dual  K-matrix $K^{+{(j)}}(u)$ of spin-$j$. It is derived from the K-operator using the classification of one-dimensional representations of $\cal A_q$, given in the present paper. Besides, this classification establishes the precise relationship between spin-$j$ K-operators for $\cal A_q$ and all known spin-$j$ K-matrices, that was not previously considered in the literature.

Having defined $\bt^{(j)}(u)$'s, we establish their relation with the sole commutative subalgebra ${\cal I} \subset {\cal A}_q$ generated by the elements  $\{ \mathsf{I}_{2k+1}|k\in{\mathbb N}\}$ given by:
	%
	\begin{equation}	
	\mathsf{I}_{2k+1}=\overline{\varepsilon}_+{\tW}_{-k} + \overline{\varepsilon}_-{\tW}_{k+1} + \bar{\kappa}_+{\tG}_{k+1}  + \bar{\kappa}_-\tilde{\tG}_{k+1}\ ,\label{Imode}
	\end{equation}
	where the scalars $\overline{\varepsilon}_\pm,\bar{\kappa}_\pm$ form the entries of the dual  K-matrix $K^{+{(\h)}}(u)$.  For $j=\h$, $\bt^{(\h)}(u)$ is the generating function of $\{ \mathsf{I}_{2k+1}|k\in{\mathbb N}\}$.
Namely,
	\begin{equation}
	\bt^{(\frac{1}{2})}(u)=(u^2 q^2-u^{-2} q^{-2}) \left( \mathsf{I}(u)+ \mathsf{I}_0 \right)\ \quad \mbox{with} \quad \mathsf{I}(u)= \sum_{k\in {\mathbb N}}{\mathsf{I}}_{2k+1}U^{-k-1} \ ,\label{t12init}
	\end{equation}
	where we set $U=(qu^2+q^{-1}u^{-2})/(q+q^{-1})$, and $\mathsf{I}_0$ is a scalar.
    The first main result of the present paper is Theorem~\ref{TTrel} where we prove the following {universal} TT-relations in ${\cal I}((u^{-1}))$ for any positive integer or half-integer $j$, provided a technical conjecture \cite[Conj.\,6.1]{LBG} holds:\footnote{It actually requires a  simpler conjecture, see Conjecture~\ref{conj1} in this paper.}
	\begin{equation} \label{normTT}
		\bt^{(j)}(u) = \bt^{(j-\h)}(u q^{-\h}) \bt^{(\h)}(u q^{j-\h}) + \frac{\Gamma (u q^{j-\tha}) \Gamma_+ (uq^{j-\tha})}{ c(u^2 q^{2j}) c(u^2 q^{2j-2}) } \bt^{(j-1)}(u q^{-1})  \ 
	\end{equation}
	with  initial conditions~\eqref{t12init} and
	$\bt^{(0)}(u) =  1$,
    and we set $c(u)=u-u^{-1}$.
	Here, $\Gamma_+(u)$   is the scalar function~\eqref{gammaKP}, $\Gamma(u)$ is the quantum determinant, a central element in $\mathcal{A}_q((u^{-1}))$, given by Proposition~\ref{prop:qdet}. We also give a direct proof of the universal TT-relations~\eqref{normTT} for $j=\h,1,\frac 32$ using a Poincar\'e-Birkhoff-Witt (PBW) basis for ${\cal A}_q$, see Section~\ref{sub:TT-PBW}. 
    
    As a first corollary of~\eqref{normTT}, we obtain the ${\cal A}_q$ variant of T-systems~\cite{kunib}:
	\begin{equation}
	\bt^{(j)}(u q^{-\h}) \bt^{(j)}(u q^{\h}) = \bt^{(j+\h)}(u) \bt^{(j-\h)}(u) + {\bf g}^{(j)}(u) \ ,\label{Tsysint}
	\end{equation}
	where ${\bf g}^{(j)}(u)$ is given by~\eqref{def:gj}. Another immediate consequence of~\eqref{normTT} is that all $\bt^{(j)}(u)$ are  formal Laurent series in $u$ (with highest term $u^{4j}$) whose coefficients are polynomials in  $\{ \mathsf{I}_{2k+1}|k\in{\mathbb N}\}$, and that can be computed explicitly. In particular, we get that 		$\big[\bt^{(j)}(u), \bt^{(k)}(v)\big]=0$ for all $j,k\in\h\mathbb{N}$, generalizing the general property~\eqref{eq:T-com-T}. 
     Furthermore, for properly normalized version of $\bt^{(j)}(u)$, the resulting  TT-relations reflect the product of evaluation representation of $\Loop$ in the Grothendieck ring, see details in Section~\ref{sec:normalTT}, which
    is analogous to the property of the ring homomorphism in~\cite[Thm.\,6.7.1]{AV24} for a different universal transfer matrix construction associated with quantum symmetric pairs.

\smallskip

 Concerning the point (i) raised above in Section~\ref{intro:goal},
  consider the spin-$j$ transfer matrices ${\boldsymbol  t}^{j,\bj}(u)$ for  open spin chains of length $N$  defined in~\eqref{tjs}. These are characterized by an
$N$-tuple of spins   $\bj=(j_1,\ldots,j_N)$ at the quantum spaces and inhomogeneities $v_1,\ldots,v_N\in \mathbb{C}^*$. In particular cases, they are the transfer matrices of  higher XXZ spin chains~\cite{FNR07} and alternating spin chains~\cite{CYSW14}. Combining the results of Sections~\ref{sec3} and~\ref{sec4}, we show in Proposition~\ref{prop:TNtj} that all $\bt^{(j)}(u)$'s specialize to the transfer matrices ${\boldsymbol  t}^{j,\bj}(u)$ on spin-chain representation of ${\cal A}_q$  using the map $\psi^{(N)}_{\bj,\bar{v}}$ introduced in~Definition~\ref{def:psiN}, while the quantum determinant coefficient $\Gamma(u)$ in~\eqref{normTT} becomes a scalar formal Laurent series in $u$ whose coefficients depend on $N$, $j$, $\bj$,  $v_i$'s, and boundary parameters. As a consequence, we get TT-relations on the spin-chains transfer matrices, see Proposition~\ref{prop:TTt}, and show  that they agree with the conjectured TT-relations  
	 in all known particular cases~\cite{FNR07,CYSW14}.
  We thus achieve the point 
 that the universal TT-relations~\eqref{normTT} provide an umbrella for all known open spin-chains TT-relations.

With respect to the point (ii), recall that any local conserved quantity for an integrable open spin chain can be derived from the logarithmic derivatives of its associated transfer matrix~\cite{Skly88}. 
Using the fact that due to the universal TT-relations~\eqref{normTT} our universal transfer matrices are expressed in terms of $\{\mathsf{I}_{2k+1}\}_{k\in \mathbb{N}}$ using~\eqref{t12init} and~\eqref{Imode}, we give in Section~\ref{sec:alg-Hn} an algorithm that allows to compute a closed form expression for any local conserved quantity in terms of spin operators.

 The above mentioned conserved {\it quasi-local} operators\footnote{An operator is said to be quasi-local if its norm at infinite temperature grows linearly with system size, as it does for any local operator, see more precisely in~\cite{IMP15}.}
 are as important for non-equilibrium physics as the local ones.  They
have been constructed  for the XXZ spin-$\h$ chain with closed boundary conditions, and with their generating function~\cite[Eq.\,(7)]{IMP15} \& \cite[Eq.\,(16)]{INWCEP15}  expressed in terms of the transfer matrix and its shifted derivative. We may then ask for the existence of a similar construction for open spin chains.
 It remains an interesting problem to check whether the same expression produces quasi-local charges in the generic open boundary conditions case.
If it does, all the quasi-local conserved quantities can again be expressed as polynomials in the spin-chain images of the $\mathsf{I}_{2k+1}$'s which are non-local operators. A general fundamental question one may ask here: \textit{Is it true that the long time average of local observables after a quantum quench is described by a generalized Gibbs ensemble involving exclusively the spin-chain images of the $\mathsf{I}_{2k+1}$'s?}

\medskip
We further consider applications of our universal TT-relations in studying symmetry properties of integrable spin-$j$ chains. Namely, we first derive  exchange relations between (linear combinations of) the generators of~${\cal A}_q$ and the commuting family $\{\mathsf{I}_{2k+1}\}_{k\in \mathbb{N}}$, with the results in Propositions~\ref{LemmaCondW0W1} and~\ref{prop:lin-comb-Tj}. The exchange relations between the $q$-Onsager operators and the corresponding open spin-chain Hamiltonians in Proposition~\ref{corol1} are quite analogous to the exchange relations~\cite[Eqs.\,(2.62a)]{PS90} between the $\Uq$ generators $S_\pm$ and  twisted periodic XXZ Hamiltonians.
We furthermore show that on spin-chains of arbitrary spins the exchange relations lead to non-trivial symmetries of the spin-$j$ Hamiltonians and transfer matrices in Propositions~\ref{prop:HXXZ-comm-cond} and~\ref{prop:tr-mat-w0-w1}, given by the action of ${\tW}_{0}$, ${\tW}_{1}$ or their linear combination depending on the boundary conditions chosen. 
This is a  generalization of~\cite{Doikou} for $j=\h$. 
 We refer to further details in Section~\ref{sec7:blob}.
We  observe in Section~\ref{sec6:XXX} interesting symmetries at $q=1$, or in the XXX-type Hamiltonians case, for general non-diagonal boundary parameters under the condition that their ratios on both sides are equal.

 Finally in Section~\ref{subsec:univ-TQ-Aq}, following the idea in \cite{YNZ06} we take the limit $j\rightarrow \infty$ of~\eqref{normTT}  with the formal identification
     $\bold{Q}_\pm(u) = \lim_{j\rightarrow \infty}  \bt^{(j)}\bigl(uq^{\pm(j+\h)}\bigr)$
     and obtain the \textit{universal} TQ-relations for $\cal A_q$
     \begin{equation*}
     \bt^{(\h)}(u) \bold{Q}_\pm(u) =  \bold{Q}_\pm(uq^{\mp 1}) -\frac{\Gamma (uq^{(-1 \pm 1)/2}) \Gamma_+ (uq^{(-1\pm 1)/2})}{ c(u^2 q^{\pm 1}) c(u^2 q^{2\pm 1}) }\bold{Q}_\pm(uq^{\pm 1})   \ .
      \end{equation*}
Universal TQ-relations for the $q$-Onsager algebra readily follow using the results of Section~\ref{sub:TT-qOA}, thus generalizing the universal TQ-relations associated with the diagonal case given in~\cite{T20,VW20}.      
On the level of spin chains with generic boundary conditions, 
from the  transfer matrix TT-relations in Proposition~\ref{prop:TTt}
we also derive  the transfer matrix  TQ-relations for ${\boldsymbol  t}^{\h,\bj}(u)$ recovering the TQ-relation~\cite[Eq.\,(3.2)]{FNR07}.
  We equally discuss how universal TT-relations for degenerate versions of $\cal A_q$ can be derived. In particular, universal TT- and TQ-relations for a central extension of the augmented $q-$Onsager algebra are given in Section~\ref{subsec:univ-TQ-Aq-aug}, in view of applications to the case of diagonal boundary conditions.

\smallskip

The paper is organized as follows. 
	In Section~\ref{sec2}, the basic ingredients for the construction of the generating functions $\bt^{(j)}(u)$  in~\eqref{tgint} with $B=\cal A_q$ are recalled, e.g.\ the spin-$j$ K-operators built by the fusion procedure~\cite{LBG} and the quantum determinant of $\cal{A}_q$. One-dimensional representations of $\cal A_q$ are classified, and used to establish the relation between spin-$j$ K-operators, spin-$j$ K-matrices and dual spin-$j$ K-matrices, and evaluate correspondingly  their quantum determinant.   Section \ref{sec3} contains all our results concerning the universal TT-relations for ${\cal A}_q$.  
 As a consequence, in Section~\ref{sub:TT-qOA} we obtain the universal TT-relations for the $q$-Onsager algebra $O_q$ using the surjective algebra map ${\cal A}_q \rightarrow O_q$~\cite{BasBel,Ter21c}.  
   In Section~\ref{sec4}, in order to relate the universal TT-relations to the ones associated with integrable open spin chains  of length $N$, certain quotients of ${\cal A}_q$ denoted $\{{\cal A}^{(N)}_q\}_{N\in{\mathbb N^*}}$  are introduced.  Corresponding `truncated' K-operators $\mathcal{K}^{(j,N)}(u)$ are constructed and they are now Laurent polynomials in $u$ instead of formal Laurent series. Importantly, dressed K-matrices that are typical building ingredients of spin-chains transfer matrices are shown to be images via $\psi^{(N)}_{\bj,\bar{v}}$ of  $\mathcal{K}^{(j,N)}(u)$, see Proposition~\ref{prop:Kjspinchain}. 
    Sections~\ref{sec5} and~\ref{sec6} contain main applications to integrable open spin chains, in particular the TT-relations for spin-chain transfer matrices given by Corollary~\ref{cor:nt}, and an algorithm for local conserved quantities in terms of spin operators in Section~\ref{sec:alg-Hn}. Non-trivial symmetries of spin-$j$ Hamiltonians are then discussed in Section~\ref{sec6}, see Proposition~\ref{prop:HXXZ-comm-cond} and Section~\ref{sec6:XXX}. In  Section~\ref{sec7}, we derive 
    the T-system, Y-system and universal TQ-relations associated with ${\cal A}_q$ and its degenerate versions,
    and briefly sketch  a few perspectives in connection with the algebraic Bethe ansatz and the  algebra of Hamiltonian densities known as the one-boundary Temperley-Lieb or blob algebra.
 In Appendix~\ref{Ap:EF}, we recall some necessary algebraic material from~\cite{LBG}, on products of evaluation representations of $\Loop$.
 In Appendix~\ref{Ap:ex-K-op}, we give the explicit expression for the spin-$1$ K-operator of $\cal A_q$. 
 In Appendix~\ref{Ap:proofT34}, the proof of Theorem~\ref{TTrel} is detailed. 
 In Appendix~\ref{apD}, we recall tensor product representations for $\cal A_q$ that are necessary for the spin-chains specializations of the universal TT-relations. 
 Finally, in Appendix~\ref{ApE} we give the derivation of the first two local conserved quantities of open spin-$\h$ XXZ in terms of images of mutually commuting elements $\textsf{I}_{2k+1}\in \cal A_q$, for $k=0,1,...,N-1$.

	\vspace{2mm}
	
	{\bf Conventions:} 
		We denote  the set of natural numbers by ${\mathbb N} = \{0$, $1$, $2$, $\ldots\}$ and $\mathbb{N}_+= \{1$, $2$, $\ldots \}$, 
		\beqa  \overline{k} = k \mod 2\ ,\qquad  \Big\lfloor\frac{k}{2}\Big\rfloor=\null \frac{k-\overline{k}}{2}\ .
		\eeqa
        The Kronecker symbol is denoted $\delta_{m,n}$.
        
		Let $q\in\mathbb{C}^*$, and we assume  in this paper that $q$ is not a root of unity. 		We use the $q$-numbers 
        $$
        [n]_q= (q^n-q^{-n})/(q-q^{-1})\ .
        $$
 The $q$-commutator  is
		$$
		\big[X,Y\big]_q=qXY-q^{-1}YX
		$$
		and $\big[X,Y\big]=\big[X,Y\big]_1=XY-YX$. 
        We also often use the following notation:
    \begin{equation}\label{eq:cu} 
		c(u)=u-u^{-1} \ ,
	\end{equation} 
    with its inverse in $\mathbb{C}[[u^{-1}]]$ 
        \begin{equation}\label{eq:cu-inv}
         c(u)^{-1} = \sum_{n=0}^{\infty}u^{-2n-1}\ .
        \end{equation}
		
		We  denote by $\mathbb{I}_{2j}$ the $2j\times2j$ identity matrix and the Pauli matrices: 
		\begin{equation} \label{Pauli}
			\sigma_+=\begin{pmatrix}
				0 &1 \\ 
				0& 0
			\end{pmatrix},\qquad \sigma_-=\begin{pmatrix}
				0 &0 \\ 
				1& 0
			\end{pmatrix}, \qquad \sigma_x=\begin{pmatrix}
				0 & 1 \\ 
				1 & 0
			\end{pmatrix}, \qquad \sigma_y=\begin{pmatrix}
				0 & -i \\ 
				i & 0
			\end{pmatrix}, \qquad \sigma_z=\begin{pmatrix}
				1 &0 \\ 
				0& -1
			\end{pmatrix}.
		\end{equation}
		In this paper, all algebras are unital, all scalars are in the complex field ${\mathbb C}$ and $u$, $v$ are indeterminates.

All generating functions considered, e.g.\ the quantum determinant $\Gamma(u)$, are formal Laurent series in $u^{-1}$ with coefficients in the corresponding algebra like $\mathcal{A}_q$. In other words, they are of the form
	$$
	\displaystyle{\sum_{n\in \mathbb{Z}}} f_n u^{-n} \ ,
	$$
	where $f_n \in \mathcal{A}_q$ and all but finitely many $f_n$ for $n<0$ vanish. In this case, we use the notation $\mathcal{A}_q((u^{-1}))$, and the convention that every rational function of the form $1/p(u)$, where $p(u)$ is a Laurent polynomial, is expanded in $u^{-1}$.
	
	      In what follows, we use  conventions on the coproduct of the quantum loop algebra $\Loop$ and definition of its evaluation representations 
from~\cite[Sec.\,2.4]{LBG}.
	
	\section{Fused  K-operators and K-matrices}\label{sec2}
	In this section, the main ingredients for the construction of the generating function (\ref{tgint}) in ${\cal A}_q$ are recalled. In Section~\ref{sec:defAq}, we recall two  presentations of the algebra $\mathcal{A}_q$. 
	In Section~\ref{sec:fusRK}, we review the fused expressions for the spin-$j$ R-matrix and the spin-$j$ K-operators, see Definitions~\ref{def:fusR} and~\ref{def:fusedK}. In Section~\ref{sec:Kop-to-K}, introducing an infinite family of algebra maps $\epsilon: \mathcal{A}_q \to \mathbb{C}$ we derive spin-$j$  K-matrices from spin-$j$ K-operators.
	   The spin-$\h$ and spin-$1$ R- and K-matrices are given explicitly and compared with  other constructions of fused R- and K-matrices in the literature. In Section~\ref{sec:prop-K}, we describe properties of K-matrices and their duals: the intertwining relations they satisfy, the polynomiality of the normalized K-matrices and that all spin-$j$ K-matrices are symmetric under the change $k_\pm \to k_\mp$.
	   Finally, in Section~\ref{sec:qdet}, we recall the quantum determinant for $\mathcal{A}_q$ and compute its image under the maps $\epsilon$.

	  \subsection{Definition of \texorpdfstring{${\cal A}_q$}{Aq}} \label{sec:defAq} 
We now review the reflection algebra presentation of  $\mathcal{A}_q$. It involves  a K-operator which entries are generating functions in the generators of $\mathcal{A}_q$~\cite{BS09}.
			
			First, recall the \textit{fundamental} R-matrix associated with the tensor product of two-dimensional evaluation representations of $\Loop$. It is given by:
	\begin{equation} \label{Rhh}
		R^{(\h,\h)}(u)=\begin{pmatrix}
			uq-u^{-1}q^{-1} & 0 &  0 &0  \\ 
			0  &  u-u^{-1}&  q-q^{-1}&0  \\ 
			0 &  q-q^{-1}&u-u^{-1}  & 0 \\ 
			0 & 0  & 0  &uq-u^{-1}q^{-1} 
		\end{pmatrix}
	\end{equation}
	and satisfies the Yang-Baxter equation (\ref{YBj1j2})  for $j_1=j_2=j_3=\h$. 
    
	\begin{defn}[\cite{BS09}]\label{def:Aq0} ${\cal A}_q$ is an associative algebra over $\mathbb{C}$ with alternating generators 
    $$
    \{\ {\tW}_{-k},{\tW}_{k+1},{\tG}_{k+1},{\tilde{\tG}}_{k+1}|\ k\in {\mathbb N}\ \}\
    $$ with the corresponding generating functions, considered as power series in $u^{-2}$:
		\begin{align}
			{\cW}_+(u)=\sum_{k\in {\mathbb N}}{\normalfont \tW}_{-k}U^{-k-1} \ , \quad {\cW}_-(u)=\sum_{k\in  {\mathbb N}}{\normalfont \tW}_{k+1}U^{-k-1} \ ,\label{c1}\\
			\quad {\cG}_+(u)=\sum_{k\in {\mathbb N}}{\normalfont \tG}_{k+1}U^{-k-1} \ , \; \quad {\cG}_-(u)=\sum_{k\in {\mathbb N}}{\normalfont \tilde{{\tG}}_{k+1}}U^{-k-1} \ ,\label{c2} \; 
		\end{align}
		where we use the  shorthand notation 
        \begin{equation}\label{eq:U-def}
       U=\frac{qu^2+q^{-1}u^{-2}}{q+q^{-1}}\ .
        \end{equation}

 The defining relations are then given by: 
		\begin{align} 
			{R}^{(\frac{1}{2},\frac{1}{2})}(u/v)\ {\cal K}_1^{(\frac{1}{2})}(u)\ { R}^{(\frac{1}{2},\frac{1}{2})}(uv)\ {\cal K}_2^{(\frac{1}{2})}(v)\
			= \ {\cal K}_2^{(\frac{1}{2})}(v)\  { R}^{(\frac{1}{2},\frac{1}{2})}(uv)\ {\cal K}_1^{(\frac{1}{2})}(u)\ { R}^{(\frac{1}{2},\frac{1}{2})}(u/v)\ 
			\label{RE} 
		\end{align}
		with the R-matrix~(\ref{Rhh})   and the fundamental K-operator
		\begin{equation}
			{\cal K}^{(\frac{1}{2})}(u)=\begin{pmatrix} 
				uq \cW_+(u)-u^{-1}q^{-1}\cW_-(u) &\frac{1}{k_-(q+q^{-1})}\cG_+(u)+\frac{k_+(q+q^{-1})}{q-q^{-1}} \\ 
				\frac{1}{k_+(q+q^{-1})}\cG_-(u) +\frac{k_-(q+q^{-1})}{q-q^{-1}}& uq \cW_-(u) -u^{-1}q^{-1}\cW_+(u) 
			\end{pmatrix}  \label{K-Aq} \ ,
		\end{equation}
		where $k_\pm \in \mathbb{C}^*$.
	\end{defn}
 Let us note that $U^{-1}$ can be written as a power series in $u^{-2}$:
\begin{equation}\label{eq:UPowerSeries}
	U^{-1} = (1+q^{-2}) u^{-2} \displaystyle{\sum_{\ell =0}^{\infty} (-u^{-4} q^{-2})^{\ell} } \ .
\end{equation}
Thus, the generating functions $\tW_\pm(u),\tG_\pm(u)$ given by~\eqref{c1},~\eqref{c2} start with $u^{-2}$. Consequently, the leading term of the diagonal entries of the K-operator in~\eqref{K-Aq} is at $u^{-1}$, while the leading term of the off-diagonal entries is at $u^0$. Therefore,  
\beqa
\mathcal{K}^{(\h)}(u) \in \mathcal{A}_q[[u^{-1}]]\otimes \End(\mathbb{C}^2)\ .\label{K-Aqexp}
\eeqa
We recall that $\mathcal{K}^{(\h)}(u)$ is invertible as a power series~\cite[Eq.\,(5.24)]{LBG}.

For convenience, introduce the parametrization:
\begin{equation}
	\rho=k_+k_-(q+q^{-1})^2 \; \in  \mathbb{C}^* \ .\label{rho}
\end{equation}
Inserting the R-matrix and the K-operator, given respectively in~\eqref{Rhh} and~\eqref{K-Aq}, into the reflection equation~\eqref{RE}, one extracts defining relations between the generating functions  $\tW_\pm(u)$, $\tG_\pm(u)$, see~\cite[Def.\,2.2]{BS09}. Then,  the relations in terms of the modes  are given by:
	\begin{align}
		\big[{\tW}_0,{\tW}_{k+1}\big]=\big[{\tW}_{-k},{\tW}_{1}\big]=\frac{1}{(q+q^{-1})}\big({\tilde{\tG}_{k+1} } - {{\tG}_{k+1}}\big)\ ,\label{qo1}\\
		\big[{\tW}_0,{\tG}_{k+1}\big]_q=\big[{\tilde{\tG}}_{k+1},{\tW}_{0}\big]_q=\rho{\tW}_{-k-1}-\rho{\tW}_{k+1}\ ,\label{qo2}\\
		\big[{\tG}_{k+1},{\tW}_{1}\big]_q=\big[{\tW}_{1},{\tilde{\tG}}_{k+1}\big]_q=\rho{\tW}_{k+2}-\rho{\tW}_{-k}\ ,\label{qo3}\\
		\big[{\tW}_{-k},{\tW}_{-l}\big]=0\ ,\quad 
		\big[{\tW}_{k+1},{\tW}_{l+1}\big]=0\ ,\label{qo4}\quad \\
		\big[{\tW}_{-k},{\tW}_{l+1}\big]
		+\big[{{\tW}}_{k+1},{\tW}_{-l}\big]=0\ ,\label{qo5}\\
		\big[{\tW}_{-k},{\tG}_{l+1}\big]
		+\big[{{\tG}}_{k+1},{\tW}_{-l}\big]=0\ ,\label{qo6}\\
		\big[{\tW}_{-k},{\tilde{\tG}}_{l+1}\big]
		+\big[{\tilde{\tG}}_{k+1},{\tW}_{-l}\big]=0\ ,\label{qo7}\\
		\big[{\tW}_{k+1},{\tG}_{l+1}\big]
		+\big[{{\tG}}_{k+1},{\tW}_{l+1}\big]=0\ ,\label{qo8}\\
		\big[{\tW}_{k+1},{\tilde{\tG}}_{l+1}\big]
		+\big[{\tilde{\tG}}_{k+1},{\tW}_{l+1}\big]=0\ ,\label{qo9}\\
		\big[{\tG}_{k+1},{\tG}_{l+1}\big]=0\ ,\quad   \big[{\tilde{\tG}}_{k+1},\tilde{{\tG}}_{l+1}\big]=0\ ,\label{qo10}\\
		\big[{\tilde{\tG}}_{k+1},{\tG}_{l+1}\big]
		+\big[{{\tG}}_{k+1},\tilde{{\tG}}_{l+1}\big]=0\ .\label{qo11}
	\end{align}
The relation of ${\cal A}_q$ to the more familiar $q$-Onsager algebra is described in Section~\ref{sub:TT-qOA}.

		\subsection{Fused R-matrices and K-operators for \texorpdfstring{${\cal A}_q$}{Aq}}  \label{sec:fusRK}
		The presentation of $\mathcal{A}_q$ in Definition~\ref{def:Aq0} is based on the spin-$\h$ K-operator. We now recall the fusion procedure of~\cite[Sec.\,3]{LBG} allowing to construct spin-$j$ K-operators.
 First, we recall important ingredients: for each value of $j$, the operator 
			\begin{equation}
			\mathcal{E}^{(j)}\colon  \mathbb{C}_u^{2j+1} \rightarrow \mathbb{C}_{u_1}^{2} 
			\otimes \mathbb{C}_{u_2}^{2j} \ ,  \quad  u_1 = u q^{-j+\h} , \quad u_2=u q^\h 		\ ,
			\end{equation}
		intertwining the action of $\Loop$ on formal evaluation representations \cite[Section 3.5]{LBG}, and its pseudo-inverse
	\begin{equation}
		\mathcal{F}^{(j)} \colon  \mathbb{C}^{2}_{u_1} \otimes \mathbb{C}_{u_2}^{2j} \rightarrow \mathbb{C}_u^{2j+1} \ .
	\end{equation}
Their matrix  expressions are recalled in Appendix~\ref{Ap:EF}. \smallskip

Fused R-matrices that solve the Yang-Baxter equation~\eqref{YBj1j2} are constructed from the fundamental R-matrix~\eqref{Rhh} as in \cite[Prop.\,4.7]{LBG}, see also~\cite{RSV16}. Below, we use the notation $\langle 12 \rangle$, $\langle 23 \rangle$, to specify which spaces are fused, that is where $\mathcal{E}^{(j)}$ and $\mathcal{F}^{(j)}$ act.
\begin{defn}[\cite{LBG}] \label{def:fusR} For any $j_1,j_2 \in \h \mathbb{N}_+$,
the fused R-matrices are given by the recursions
\begin{equation} \label{v2Rj1j2}
	R^{(j_1,j_2)}(u)=\mathcal{F}^{(j_1)}_{\fu} R_{13}^{(\h,j_2)}(u q^{-j_1+\h}) R_{23}^{(j_1-\h,j_2)}(u q^{\h}) \mathcal{E}^{(j_1)}_{\fu} \ ,
\end{equation}
with
\begin{equation} \label{fused-R-uq}
	R^{(\h,j)}(u)=\mathcal{F}^{(j)}_{\langle 23 \rangle} R_{13}^{(\h,j-\h)}(u q^{-\h}) R_{12}^{(\h,\h)}(u q^{j-\h})  \mathcal{E}^{(j)}_{\langle 23 \rangle} \ , \;\; 
\end{equation}
and where $R^{(\h,0)}(u)=R^{(0,\h)}(u)={\mathbb I }_2$ and $R^{(\h,\h)}(u)$ is given in~\eqref{Rhh}.
\end{defn}

\begin{example}
The spin-$(\h,j)$ R-matrix is given by
~\cite{KR83}:
\begin{align}\nonumber
	R^{(\h,j)}(u) &= \displaystyle{ \prod_{k=0}^{2j-2} }c(u q^{j-\h-k}) \times \\ \label{R-Rqg}
	& \left (  (q-q^{-1})\left ( \sigma_+ \otimes S_- +  \sigma_- \otimes S_+ \right ) + uq^{\frac{1}{2}( {\mathbb I_{4j+2}}+\sigma_z\otimes S_3)} - u^{-1}q^{-\frac{1}{2}( {\mathbb I}_{4j+2} +\sigma_z\otimes S_3)}\right ) ,
\end{align}
where $c(u)$ is given in~\eqref{eq:cu}, the Pauli matrices are recalled in~\eqref{Pauli} and the matrices
\begin{equation}\label{Bdef}
\begin{split}
	(S_+)_{mn}&=B_{j,j+1-m}\delta_{m,n-1}\ ,\quad (S_-)_{mn}=B_{j,j+1-n}\delta_{m-1,n}\quad \mbox{with} \quad B_{j,j'}=\sqrt{ \left[ j+j'\right]_q \left[j-j'+1\right]_q}\ ,\\
	(S_3)_{mn} &= 2(j+1-n) \delta_{m,n} \ , \quad m,n=1,2,...,2j+1 \ ,
\end{split}
\end{equation}
form a $(2j + 1)$-dimensional representation of the $\Uq$ algebra, denoted in what follows by $\pi^j$.
The proof that~\eqref{fused-R-uq} agrees with~\eqref{R-Rqg} is given in~\cite[Lem.\,4.14]{LBG} and relies on the fact that these two expressions can be understood from the framework of the universal R-matrix $\mathfrak{R}$. 
\end{example}
 
\begin{example}\label{ex:R-spin1}
	The spin-$1$ R-matrix is computed using (\ref{v2Rj1j2}) for $j_1=j_2=1$ with $R^{(\h,1)}(u)$ in~\eqref{R-Rqg}. It is given  by:
\begin{equation*}
	R^{(1,1)}(u)=c(u) c(uq)^2\begin{pmatrix}
		a_1(u)& \cdot  & \cdot  & \cdot  & \cdot   & \cdot  & \cdot  & \cdot  & \cdot \\
		\cdot  & a_2(u)&  \cdot & a_3(u)  & \cdot  & \cdot  & \cdot  & \cdot  & \cdot \\
		\cdot  & \cdot  & a_4(u)& \cdot  & a_5(u)  & \cdot  & a_6(u)  & \cdot  & \cdot \\
		\cdot  &  a_3(u)& \cdot  & a_2(u)  & \cdot  & \cdot  & \cdot  & \cdot  & \cdot \\
		\cdot  & \cdot  &  a_5(u)&  \cdot  &  a_7(u)& \cdot  & a_5(u)  & \cdot  & \cdot \\
		\cdot  & \cdot  & \cdot  & \cdot  & \cdot  & a_2(u)& \cdot  & a_3(u)   & \cdot \\
		\cdot  & \cdot  &  a_6(u)& \cdot  &  a_5(u)& \cdot  & a_4(u)  & \cdot  & \cdot \\
		\cdot  & \cdot  & \cdot  & \cdot  & \cdot  &a_3(u)  & \cdot & a_2(u)  & \cdot \\
		\cdot  & \cdot  & \cdot  & \cdot  & \cdot  & \cdot  & \cdot  & \cdot  &a_1(u)
	\end{pmatrix} 
\end{equation*}
where $c(u)$ is defined in~\eqref{eq:cu} and with
\begin{align*}
	a_1(u) =  c(uq^2) \ , \qquad
	a_2(u) =  c(u) \ , \qquad
	a_3(u)  = c(q^2) \ , \qquad 	a_4(u) =  \frac{c(u)c(u q^{-1})}{c(uq)}  \ , \\
	a_5(u) = \frac{ c(q^2)c(u)}{ c(uq) } \ , \qquad
	a_6(u)  =  \frac{c(q) c(q^2)}{c(u q)} \ , \qquad
	a_7(u)  = a_6(u)+a_2(u) \ .
\end{align*}
The spin-$1$ R-matrix given above corresponds, up to an overall normalization factor, to the R-matrix of the nineteen vertex model under the identification~\cite[eq.\,(2.2)]{Inami1996}
\begin{equation*}
	u \rightarrow i \ln(u) \ , \qquad \eta \rightarrow i \ln (q) \ .
\end{equation*}
It was originally obtained by solving directly the Yang-Baxter equation~\cite{ZF80} and  it was later constructed using another fusion procedure developed in~\cite{KRS81}.
\end{example}

	Higher spin-$j$ K-operators solving the reflection algebra relations (\ref{REKop}) with fused R-matrix  (\ref{v2Rj1j2}) are constructed from the initial K-operator solution (\ref{K-Aq}).
\begin{defn}\textit{\cite[Def.\,5.6]{LBG}}\label{def:fusedK} For any $j \in \h \mathbb{N}_+$, the fused K-operators are  defined by the following recursion: 
	\begin{equation} \label{fused-K-op}
		\mathcal{K}^{(j)}(u) =  \mathcal{F}^{(j)} \mathcal{K}_1^{(\h)}(u q^{-j+\h}) R^{(\h,j-\h)}(u^2 q^{-j+1}) \mathcal{K}_2^{(j-\h)}(u q^{\h}) \mathcal{E}^{(j)}  \ ,
	\end{equation}
	with $\mathcal{K}^{(\h)}(u)$ defined in~\eqref{K-Aq} and  $\mathcal{K}^{(0)}(u)=1$, and recall   the expressions for the operators $\mathcal{E}^{(j)}$ and $\mathcal{F}^{(j)}$ given by~\eqref{exprE} and~\eqref{exprF}.
\end{defn}
	\begin{example}
	The spin-$1$ K-operator $\mathcal{K}^{(1)}(u)$ is computed \cite[Sec.\,5.4.1]{LBG}  using~\eqref{fused-K-op} with the fundamental K-operator~\eqref{K-Aq},
	and its expression is reported in Appendix~\ref{Ap:ex-K-op}.
	\end{example}

	\subsection{From K-operators to K-matrices} \label{sec:Kop-to-K}
	We now derive spin-$j$ K-matrices $K^{(j)}(u)$ from spin-$j$ K-operators $\mathcal{K}^{(j)}(u)$ given in Definition~\ref{def:fusedK}. The latter belong to $\mathcal{A}_q((u^{-1})) \otimes \End(\mathbb{C}^{2j+1})$, so it is sufficient to apply one-dimensional representations on the first component to get the spin-$j$ K-matrices. 
	 All one-dimensional representations of $\mathcal{A}_q$ are classified in the following proposition. 
	\begin{prop} \label{prop:eps}
		We have the algebra  map  $\epsilon \colon \mathcal{A}_q \to \mathbb{C}$
			 \begin{align} \label{eq:epsgk}
			\epsilon( \tG_{k+1}) &= \epsilon(\tilde{\tG}_{k+1})= g_{k+1} \ , \\ \label{eq:epswk}
			\epsilon(\tW_{k+1})&=g_0^{-1} \left ( \displaystyle{\sum_{m=0}^{k}} (\overline{m} \epsp + \overline{m+1} \epsm) g_{k-m}  \right) \ ,  \qquad k \in \mathbb{N} \ ,  \\ \label{eq:epswmk}
			\epsilon(\tW_{-k}) &= \epsilon(\tW_{k+1})\rvert_{\varepsilon_\pm \rightarrow \varepsilon_\mp} \ ,
		\end{align}
		with $g_0=\rho/(q-q^{-1})$, and for any scalars $g_{k+1}$, $\varepsilon_\pm \in \mathbb{C}$ for  $k \in \mathbb{N}$, and where we used the notation $\overline{m}= m \mod 2$.
	\end{prop}
	\begin{proof}
	Recall the presentation of $\mathcal{A}_q$ with defining relations~\eqref{qo1}-\eqref{qo11}. In every one-dimensional representation of $\mathcal{A}_q$, the  generators are evaluated to scalars, and so they commute. Therefore, the only non-trivial defining relations in such representations are~\eqref{qo1}-\eqref{qo3}. From~\eqref{qo1}, we conclude that $\epsilon( \tG_{k+1}) = \epsilon(\tilde{\tG}_{k+1})$.
	  Moreover, from~\eqref{qo2} and~\eqref{qo3}, the images $\epsilon(\tW_{\ell})$ for $\ell \in \mathbb{Z}\backslash \{0,1\} $ can be expressed in terms of $\epsilon(\tW_0)$, $\epsilon(\tW_1)$ and $\epsilon(\tG_{k+1})$. Indeed, setting $\epsilon(\tW_0) = \epsp$ and $\epsilon(\tW_1) = \epsm$, the recursion relations~\eqref{qo2},~\eqref{qo3} have a unique solution given by~\eqref{eq:epswk}-\eqref{eq:epswmk}. Finally, the scalars $\varepsilon_\pm$ and $\epsilon(\tG_{k+1}) = g_{k+1}$ have no more constraints, therefore they parameterize one-dimensional representations.
\end{proof}
\begin{example}\label{mapQSP}  Consider the algebra map   $\epsilon \colon \mathcal{A}_q \to \mathbb{C}$ introduced in Proposition~\ref{prop:eps}  for the choice $g_k=0$ for $k\in{\mathbb N}^*$. Denote it by $\epsilon_0$, then
		\beqa
			\epsilon_0( \tG_{k+1}) &=& \epsilon_0(\tilde{\tG}_{k+1})= 0 \ , \notag\\  
			\epsilon_0(\tW_{k+1})&=&     \overline{k} \epsp + \overline{k+1} \epsm\ ,  \qquad k \in \mathbb{N} \ , \label{eq:eps-0}\\
			\epsilon_0(\tW_{-k}) &=&     \overline{k+1} \epsp + \overline{k} \epsm \ .\notag
		\eeqa  
\end{example}
	
	We compute the image of the fundamental K-operator $\mathcal{K}^{(\h)}(u)$ given in~\eqref{K-Aq} under the general map $\epsilon \otimes \id$ defined in Proposition~\ref{prop:eps}. 
	\begin{prop} \label{prop:Kaq-Km}
    For the one-dimensional representations $\epsilon$ of ${\cal A}_q$ defined in Proposition~\ref{prop:eps} we have
		\begin{equation} \label{eq:Kaq-KM}
		(\epsilon\otimes \id ) ( \mathcal{K}^{(\h)}(u)) = \varsigma^{(\h)}(u) K^{(\h)}(u) \ , 
	\end{equation}
	with
	\begin{equation}\label{KM-spin1/2}
		K^{(\h)}(u)=\begin{pmatrix}
			u \epsp + u^{-1} \epsm & \kp (u^2-u^{-2})/(q-q^{-1}) \\ 
			\km (u^2-u^{-2})/(q-q^{-1}) & u\epsm + u^{-1} \epsp
		\end{pmatrix}\ ,
	\end{equation}
	and with the invertible power series
	\begin{equation} \label{eq:varsigh}
		\varsigma^{(\h)}(u) =   \frac{c(q^2)}{\rho c(u^{2})}  \left ( \displaystyle{ \sum_{k=0}^{\infty} U^{-k} g_k } \right) \ .
	\end{equation}
\end{prop}
\begin{proof}
	We compute the l.h.s.\ of the first relation in~\eqref{eq:Kaq-KM} using Proposition~\ref{prop:eps} and the expressions~\eqref{c1}-\eqref{c2} of the generating functions $\cG_\pm(u)$, $\cW_\pm(u)$.
	Firstly, from~\eqref{c2} and~\eqref{eq:epsgk} we have
	\begin{equation}\label{eq:epsgp}
			\epsilon(\cG_\pm(u)) = \displaystyle{\sum_{k=0}^\infty} U^{-k} g_k - g_0 \ ,
	\end{equation}
    and recall that $g_0=\rho/(q-q^{-1})$.
	Secondly, using~\eqref{c1} and~\eqref{eq:epswk} we find that the image of $\cW_-(u)$ under the map $\epsilon$ is
	\begin{align*}
		\epsilon(\cW_-(u)) &= \rho^{-1} (q-q^{-1}) \displaystyle{\sum_{k=0}^{\infty} \sum_{m=0}^k } U^{-k-1}  g_{k-m} ( \overline{m} \varepsilon_++ \overline{m+1}\varepsilon_-) \\
		&= \rho^{-1} (q-q^{-1})	\displaystyle{\sum_{m=0}^{\infty} \sum_{k=m}^\infty } U^{-k-1}  g_{k-m} ( \overline{m} \varepsilon_++ \overline{m+1}\varepsilon_-) \\
		&= \rho^{-1} (q-q^{-1})	\displaystyle{\sum_{m=0}^{\infty} U^{-m-1} (\overline{m} \varepsilon_++ \overline{m+1}\varepsilon_-) \sum_{k=0}^\infty }  U^{-k} g_{k} \ .
	\end{align*}
	Then, we get
	\begin{align} 
		\epsilon(\cW_-(u)) &=  \frac{q-q^{-1}}{\rho} \left (  \frac{\varepsilon_- U^{-1} + \varepsilon_+ U^{-2}}{1-U^{-2}} \right ) \left ( \displaystyle{ \sum_{k=0}^{\infty} U^{-k} g_k } \right) \ , \\ \label{eq:epswm}
		\epsilon(\cW_+(u))&=  \epsilon(\cW_-(u)) \rvert_{\varepsilon_\pm \rightarrow \varepsilon_\mp} \ ,
	\end{align}
where the second relation follows from~\eqref{eq:epswmk}.
	Finally, applying $\epsilon \otimes \id$ to~\eqref{K-Aq} and using~\eqref{eq:epsgp}-\eqref{eq:epswm}, one gets the r.h.s.\ of~\eqref{eq:Kaq-KM}, and  invertibility of the coefficient $\varsigma^{(\h)}(u)$ follows from the fact that this power series has a non-zero constant term.
\end{proof}
Note that the evaluation of the fundamental K-operator in Proposition~\ref{prop:Kaq-Km}  equals (up to an overall scalar function) the well-known K-matrix solution  of~\cite{dVG,GZ}. 
 For what follows, it is useful to notice its specialization at $u=1$:
\beqa
K^{(\h)}(u)|_{u\to 1} = (\varepsilon_+ + \varepsilon_-) {\mathbb I}_2 \ .
\label{eq:Kh}
\eeqa

\smallskip
	We now define spin-$j$  K-matrices via the image  of the fused K-operators $\mathcal{K}^{(j)}(u)$  from Definition~\ref{def:fusedK} under  the general one-dimensional representations $\epsilon$ of $\cal A_q$ defined in Proposition~\ref{prop:eps}.

	\begin{defn}\label{def:fused-K-mat}
For any $j\in \h\mathbb{N}_+$, the fused K-matrices $K^{(j)}(u)$ of spin-$j$ are defined by the relation
		\begin{equation}\label{eq:Kaqj-eps}
			(\epsilon \otimes \id ) (\mathcal{K}^{(j)}(u)) = \varsigma^{(j)}(u) K^{(j)}(u) \ ,
		\end{equation}  
		with invertible
		\begin{equation}\label{eq:varsigj}
			\varsigma^{(j)}(u) = \displaystyle{  \prod_{k=0}^{2j-1} \varsigma^{(\h)}(u q^{j-\h-k})   } \ ,
		\end{equation}
	and where $\varsigma^{(\h)}(u)$ is given in~\eqref{eq:varsigh}.
\end{defn}

Equivalently, using  the definition of fused K-operators in~\eqref{fused-K-op}   and the evaluation of the fundamental K-operator in Proposition~\ref{prop:Kaq-Km}, one can see by induction that the fused K-matrices are defined  for any $j\in \h\mathbb{N}_+$ by the recursion
	\begin{align} \label{fused-K}
		K^{(j)}(u) &=  \mathcal{F}^{(j)} K_1^{(\h)}(u q^{-j+\h}) R^{(\h,j-\h)}(u^2 q^{-j+1}) K_2^{(j-\h)}(u q^{\h}) \mathcal{E}^{(j)} \ ,
	\end{align}
with $K^{(0)}(u)=1$ and $K^{(\h)}(u)$ is defined in~\eqref{KM-spin1/2}. 
Therefore, by construction in \eqref{eq:Kaqj-eps}, the  fused K-matrices satisfy the reflection equation~\eqref{REKop} with the substitution $\mathcal{K}^{(j)}(u) \rightarrow K^{(j)}(u)$. 

	\begin{example}\label{ex:K-spin1}
				The spin-$1$ K-matrix is computed using~\eqref{fused-K} for $j=1$ with the fundamental K-matrix~\eqref{KM-spin1/2}, the R-matrix~\eqref{Rhh} and $\mathcal{E}^{(1)}$, $\mathcal{F}^{(1)}$ given in~\eqref{eq:EHF1}. It reads:
		\begin{equation}\label{KM-spin1}
			K^{(1)}(u)= \frac{c(u^2 q)}{q-q^{-1}}\begin{pmatrix}
				x_1(u) & y_1(u)  & z(u)\\ 
				\tilde{y}_1(u) & x_2(u) &y_2(u) \\ 
				\tilde{z}(u) & \tilde{y}_2(u)  & x_3(u)
			\end{pmatrix}, 
		\end{equation}
		\begin{align*}
			x_1(u)&=(q-q^{-1}) \big ( \varepsilon_+^2 u^2 + \varepsilon_-^2 u^{-2} + \varepsilon_+ \varepsilon_- (q+q^{-1}) \big ) + k_+ k_-c(u^2q^{-1}), \\
			x_2(u)&=(q-q^{-1}) \big ( \varepsilon_+^2 + \varepsilon_-^2 + \varepsilon_+ \varepsilon_-  (u^2q^{-1}+u^{-2}q) \big ) + k_+ k_- \frac{(u^4+u^{-4}-q^{2}-q^{-2})}{q-q^{-1}} , \\
			y_1(u) &= k_+c(u^2)\sqrt{q+q^{-1}}\left( u q^{-1/2} \varepsilon_+ + u^{-1} q^{1/2} \varepsilon_- \right ) , \\
			z(u)& =k_+^2 c(u^2) c(u^2 q^{-1})/(q-q^{-1}), \\
			& \qquad	y_2(u)  =y_1(u)\bigr|_{\varepsilon_\pm\rightarrow \varepsilon_\mp}  , \qquad x_3(u)=x_1(u)\bigr|_{\varepsilon_\pm\rightarrow \varepsilon_\mp}, \qquad  \\
			\tilde{y}_1(u) &= y_1(u)\big\rvert_{\kp\rightarrow \km} , \quad \tilde{y}_2(u)= y_2(u)\big\rvert_{\kp\rightarrow \km}, \quad \tilde{z}(u) = z(u)\big\rvert_{\kp\rightarrow \km} .
		\end{align*}
Using  the spin-$1$ K-operator given in Appendix~\ref{Ap:ex-K-op} and the one-dimensional representations of the currents~\eqref{eq:epsgp}-\eqref{eq:epswm}, it is straightforward to check  that the  spin-1 K matrix in~\eqref{KM-spin1} indeed agrees with~\eqref{eq:Kaqj-eps} for $j=1$.
We finally notice that its specialization at $u=1$ is quite simple:
\beqa
&& \qquad   K^{(1)}(u)|_{u\to1} = c(q)\left(\varepsilon_+^2  + \varepsilon_-^2 + \varepsilon_+\varepsilon_-(q+q^{-1})\right) {\mathbb I}_3\ . 
\label{eq:K1}
\eeqa

		\end{example}
			\begin{rem}\label{rem:K-mat-spin1}
			 The spin-$1$ K-matrix given in the previous example specializes to the one obtained in~\cite[eq.\,(3.3)]{Inami1996} with the identification from Example~\ref{ex:R-spin1} and
			\begin{align}
				&\quad \zeta \rightarrow i \ln(\varepsilon_+)	\ , \quad 	\mu \rightarrow -k_+ \frac{\sqrt{q+q^{-1}}}{q-q^{-1}} \ , \quad \, \tilde{\mu} \rightarrow - k_- \frac{\sqrt{q+q^{-1}}}{q-q^{-1}} \ ,
			\end{align}
			and $\varepsilon_-=-\varepsilon_+^{-1}$.
			It was obtained by solving directly the reflection equation for $j_1=j_2=1$. This K-matrix also corresponds to the one constructed using another fusion procedure~\cite{MN91}, see ~\cite[eq.\,(5.3)]{Inami1996} for the formula and precise correspondence.
		\end{rem}

    \begin{rem}\label{rem:K-spin-gen}
        For   K-matrices of arbitrary spins,  a fusion procedure was developed in~\cite{KRS81,MN91}, and we would like to compare it with our Definition~\ref{def:fused-K-mat}, or with the result of  evaluation of our fused K-operators given in~\eqref{fused-K}. The final expression of the recursion~\eqref{fused-K} is\footnote{Here, the product stands for the usual matrix product and the products are ordered from left to right in an increasing way in the indices. We also recall that  indices in $K$ and $R$ indicate positions in the $(2j)^{\text{th}}$ tensor product of $\mathbb{C}^2$. For example, $K^{(\h)}_k$ stands for the matrix $\mathbb{I}_{2^{k-1}}\otimes K^{(\h)}\otimes \mathbb{I}_{2^{2j-k}}$.} 
			\begin{multline}
               \label{dvpK}
					{K}^{(j)}(u)  =  \!\left (  \displaystyle{\prod_{m=0}^{2j-2}} \mathbb{I}_{2^m} \otimes \mathcal{F}^{(j-\frac{m}{2})} \! \right ) \! \displaystyle{\prod_{k=1}^{2j} }\! \left \{\! {K}_k^{(\h)}(u q^{k-j-\h}) \! \left [ \displaystyle{ \prod_{\ell=0}^{2j-k-1}} R_{k \, 2j-\ell}^{(\h,\h)}(u^2q^{-2j+2k+\ell}) \right ] \!  \right \}\! \! \\
					 \times \left (  \displaystyle{\prod_{m=0}^{2j-2}} \mathbb{I}_{2^{2j-2-m}} \otimes \mathcal{E}^{(1+\frac{m}{2})}    \right ) \ ,
			\end{multline}
            where $\mathbb{I}_{n}$ is the $n\times n$ identity matrix while the $\Loop$ intertwining operator 	 $\mathcal{E}^{(k)}$ with its pseudo-inverse $\mathcal{F}^{(k)}$ is recalled in Appendix~\ref{Ap:EF}.
            The expression~\eqref{dvpK} was obtained along the same lines as for the analogous statement for fused K-operators in~\cite[Eq.\,(5.34)]{LBG}.
	 	Using this explicit expression, we have  computed K-matrices up to the spin value $j=2$ and compared with the result of~\cite{KRS81,MN91}. Both expressions turned out to match after applying similarity transformations and extracting the only non-zero block of  size $(2j+1)\times (2j+1)$ because the approach of~\cite{KRS81,MN91} produces K-matrices of much  larger size which is $2^{2j}\times 2^{2j}$. We thus expect that the two fusion approaches lead to the same K-matrices for arbitrary spins~$j$. 
    \end{rem}

 In view of applications to spin chains  in Section~\ref{sec5}, in the next two subsections we introduce  normalized R- and K-matrices and highlight some of their important properties.
 
\subsection{Properties of R-matrices}
By straightforward calculations, one  shows that the fused R-matrix~\eqref{R-Rqg} is symmetric, and moreover  enjoys unitarity and crossing symmetry properties:
\begin{align}\label{symRj}
	\left [ R^{(\h,j)}(u) \right ]^{t_{12}} &=  R^{(\h,j)}(u) \ , \\
	\label{unitarityJ}
	R^{(\h,j)}(u)  R^{(\h,j)}(u^{-1}) &= \beta^{(j)} (u)  \ \mathbb{I}_{4j+2} \ ,  \\ \label{crossingJ}
	\left[ R^{(\h,j)}(u) \right]^{t_1} \left[   R^{(\h,j)}(u^{-1} q^{-2}) \right]^{t_1} &= \xi^{(j)}(u) \ \mathbb{I}_{4j+2} \ ,
\end{align}
where  $t_{12}$ (resp.\ $t_1$) stands for the transposition applied to the space $V^{(j_1)}  \otimes V^{(j_2)}$ (resp.\ $V^{(j_1)}$), and
\begin{align}
	\beta^{(j)}(u) =  \prod_{k=0}^{2j-1} - c(u q^{j+\h-k}) c(uq^{-j-\h+k})   \ , 
    \qquad	\xi^{(j)}(u) = \prod_{k=0}^{2j-1} - c(u q^{j-k-\h}) c (u q^{-j+k+\frac{5}{2}})  \ . \label{beta}
\end{align}

	Based on the symmetry property~\eqref{symRj}, we furthermore show   that all the fused R-matrices $R^{(j_1,j_2)}(u)$ from~\eqref{v2Rj1j2} are equally symmetric and have the unitarity property:
		\begin{lem} \label{lem:symRj1j2}
			The fused R-matrices defined in~\eqref{v2Rj1j2} are symmetric  for any $j_1, j_2 \in \h \mathbb{N}_+$:
			\begin{equation} \label{symRj1j2}
				\lbrack R^{(j_1,j_2)}(u) \rbrack^{t_{12}} = R^{(j_1,j_2)}(u) \ ,
			\end{equation}
            with the unitarity property:
            \beqa\label{eq:R-unitarityJ}
            R^{(j_1,j_2)}(u) R^{(j_1,j_2)}(u^{-1}) = \beta^{(j_1,j_2)}(u)  \mathbb{I}_{(2j_1+1)(2j_2+1)} 
            \eeqa
            with 
\beqa\label{eq:betaj1j2}
 \beta^{(j_1,j_2)}(u)=\prod_ {k = 0}^{2 j_ 1 - 1} \prod_ {\ell = 0}^{2 j_ 2 -
    1} c (uq^{j_ 1 + j_ 2 - k - \ell}) c (u q^{-j_ 1 - j_ 2 + k + \ell})\ . 
\eeqa

			\begin{proof}
				We first prove~\eqref{symRj1j2} by induction in $j_1$ and then similarly in $j_2$. The case $(j_1,j_2)= (\h,j_2)$ holds due to~\eqref{symRj}. Fixing  $j_2$ and assuming the R-matrix $R^{(j_1-\h,j_2)}(u)$ is symmetric for any given value of $j_1$, we show that $R^{(j_1,j_2)}(u)$ is also symmetric.
				Apply the transposition to~\eqref{v2Rj1j2} to get
				\begin{align*}
					\lbrack R^{(j_1,j_2)}(u) \rbrack^{t_{12}} &= \lbrack \mathcal{E}^{(j_1)}_\fu \rbrack^{t} \lbrack R_{23}^{(j_1-\h,j_2)}(u q^{\h})  \rbrack^{t_{23}}  \lbrack R_{13}^{(\h,j_2)}(u q^{-j_1+\h}) \rbrack^{t_{13}}
					\lbrack \mathcal{F}^{(j_1)}_\fu \rbrack^t  \\
					&= \mathcal{H}_\fu^{(j_1)} \mathcal{F}_\fu^{(j_1)}
					R_{23}^{(j_1-\h,j_2)}(u q^{\h})  R_{13}^{(\h,j_2)}(u q^{-j_1+\h})
					\mathcal{E}_\fu^{(j_1)} \lbrack \mathcal{H}_\fu^{(j_1)} \rbrack^{-1} \\
					& = \mathcal{F}_\fu^{(j_1)}  R_{12}^{(\h,j_1-\h)}(q^{j_1}) R_{23}^{(j_1-\h,j_2)}(u q^{\h})  R_{13}^{(\h,j_2)}(u q^{-j_1+\h})  \mathcal{E}_\fu^{(j_1)} \lbrack \mathcal{H}_\fu^{(j_1)} \rbrack^{-1}
				\end{align*}
				where we used~\eqref{rel1}  and the induction assumption to obtain the second line, while we used~\eqref{usefulEFH} for the third line. Now, using the Yang-Baxter equation
				\begin{equation}\label{v2YBj1j2}
					R_{13}^{(j_1,j_3)}(u_1) R_{23}^{(j_2,j_3)}(u_2) R_{12}^{(j_1,j_2)}(u_2/u_1)
					=
					R_{12}^{(j_1,j_2)}(u_2/u_1) R_{23}^{(j_2,j_3)}(u_2)R_{13}^{(j_1,j_3)}(u_1) \ ,
				\end{equation}
				which follows from~\eqref{YBj1j2} and the unitarity property  of the R-matrices~\cite[Eq.\,(4.44)]{LBG}, together with~\eqref{usefulEFH}, we get
				\begin{align*}
					\lbrack R^{(j_1,j_2)}(u) \rbrack^{t_{12}} &=\mathcal{F}_\fu^{(j_1)}
					R_{13}^{(\h,j_2)}(u q^{-j_1+\h})
					R_{23}^{(j_1-\h,j_2)}(u q^{\h})  
					R_{12}^{(\h,j_1-\h)}(q^{j_1}) \mathcal{E}_\fu^{(j_1)} \lbrack \mathcal{H}_\fu^{(j_1)} \rbrack^{-1} \\
					&= R^{(j_1,j_2)}(u)  \ .
				\end{align*}

				We now proceed with induction by $j_2$. The case $(j_1,\h)$ holds by applying the  matrix $\mathcal{P}$ permuting two tensor factors to~\eqref{symRj}. Fix now $j_1$ and  assume \eqref{symRj1j2} holds for $(j_1,j_2-\h)$ for any given value of $j_2$. Using an equivalent expression for~\eqref{v2Rj1j2} from~\cite[Lem.\,C.2, Eq.\,(C.9)]{LBG}
				\begin{equation}
					R^{(j_1,j_2)}(u)  =\mathcal{F}_{\langle 23 \rangle}^{(j_2)} R_{13}^{( j_1, j_2 -\h   )} (u q^{-\h}) R_{12}^{(j_1,\h)}(u q^{j_2-\h}) \mathcal{E}_{\langle 23 \rangle}^{(j_2)} \ ,
				\end{equation}
				it is shown, in the same way as above, that this R-matrix is symmetric.

Finally, to prove~\eqref{eq:R-unitarityJ}, we first recall that in~\cite[Eq.\,(4.44)]{LBG} we  showed that 
$$
{R}^{(j_1,j_2)}(u) {R}^{(j_1,j_2)}(u^{-1}) \propto \mathbb{I}\ .
$$
Therefore, it is sufficient to look for the matrix entry $(1,1)$ in the fusion formula~\eqref{v2Rj1j2}.
We first notice by~\eqref{Ep1}-\eqref{Fp1} that the first line/column of the matrices $\mathcal{F}^{(j_1)}$ and $\mathcal{E}^{(j_1)}$ in~\eqref{v2Rj1j2} have zero entries except the entry $(1,1)$ which is 1. We therefore need just to analyze the entry $(1,1)$ of the product of two R-matrices in~\eqref{v2Rj1j2}, and this is done by induction in $j_1$ using~\eqref{R-Rqg} and in agreement with~\eqref{eq:betaj1j2}.
			\end{proof}
		\end{lem}

 {\it Normalized}  R-matrix $\tilde{R}^{(j_1,j_2)}(u) $ defined by 
\beqa
\tilde{R}^{(j_1,j_2)}(u) = \left( \prod_ {k = 0}^{2 j_ 1 - 1} \prod_ {\ell = 0}^{2 j_ 2 -
    1} c (uq^{j_ 1 + j_ 2 - k - \ell})  \right)^{-1} R^{(j_1,j_2)}(u)\ \label{renormR}
\eeqa
using~\eqref{eq:cu-inv} and Definition~\ref{def:fusR},
satisfies normalized version of~\eqref{eq:R-unitarityJ}: 
\beqa
\tilde{R}^{(j_1,j_2)}(u) \tilde{R}^{(j_1,j_2)}(u^{-1}) =\mathbb{I}\ .\label{eq:Rtilde-unit}
\eeqa

\begin{prop}\label{prop:R-M-P}
    The normalized R-matrices satisfy the following:
    \begin{enumerate}
    \item 
    \beqa\label{eq:Rtilde-M}
    \tilde{R}^{(j_1,j_2)}(u) =\left(\prod_{k=0}^{2j_1-1}c(uq^{j_1+j_2-k})\right)^{-1} \times M^{(j_1,j_2)}(u)\ ,
    \eeqa
where $M^{(j_1,j_2)}(u)$ is a $(2j_1+1)^2\times (2j_2+1)^2$ matrix with entries that are  Laurent polynomial in~$u$.
Therefore, $\tilde{R}^{(j_1,j_2)}(u)$ has no pole at $u=1$ for all values of $j_1$ and $j_2\geq j_1$.

\vspace{1mm}
    \item The entry (1,1) of  $\tilde{R}^{(j_1,j_2)}(1)$ equals 1 for all $j_1$, $j_2$. Moreover, $\tilde{R}^{(j,j)}(1)$ is the permutation operator 
    \begin{equation}\label{eq:R-P}
        \tilde{R}^{(j,j)}(1) = \cal P^{(j,j)}: v\otimes w \mapsto w\otimes v\ . 
    \end{equation}

    \end{enumerate}
\end{prop}
\begin{proof}
    \mbox{}
\begin{enumerate}
 \item The proof is  by induction using~\eqref{R-Rqg} and~\eqref{v2Rj1j2}.

\item   
From the previous point we see that $\tilde{R}^{(j,j)}(1)$ is well defined. 
To show~\eqref{eq:R-P}, we first recall  Jimbo's result~\cite[Rem.\,2]{Jim8586} that Jimbo's R-matrix denoted ${R}_J^{(j,j)}(u)$ evaluates  at $u=1$ to  the permutation operator. Moreover, ${R}_J^{(j,j)}(u)$ is  a unique solution (up to multiplication by a scalar function of $u$) of the intertwining condition involving the coproduct $\Delta_J$ that is $\Delta^{op}$ in our conventions. The use of opposite coproduct implies that Jimbo's R-matrix is proportional to the specialization of the flipped universal R-matrix $\mathcal{R}_{21}$ with respect to our coproduct. Now, using~\cite[Eq.\,(4.23)]{LBG} we conclude
$ \tilde{R}^{(j,j)}(u) \propto {R}_J^{(j,j)}(u^{-1})$,
and by~\eqref{eq:Rtilde-unit} we establish that actually 
\beqa
\tilde{R}^{(j,j)}(1) = \pm \cal P^{(j,j)} \label{eq:RjjeqP}
\eeqa
because $\tilde{R}^{(j,j)}(1)$ squares to the identity. It remains to determine the sign in~\eqref{eq:RjjeqP}, which is done by computing the entry $(1,1)$ of ${\tilde R}^{(j_1,j_2)}(u)\propto R^{(j_1,j_2)}(u)$, and then fix $j_1=j_2=j$. Consider the r.h.s. of~\eqref{v2Rj1j2} and first observe that  
\begin{equation}\label{eq:E-F-1a}
    \mathcal{E}^{(j)}_{a,1} = \mathcal{F}^{(j)}_{1,a} = \delta_{1,a}\quad \text{and} \quad  \mathcal{E}^{(j)}_{1,b} = \mathcal{F}^{(j)}_{b,1} = \delta_{1,b}\ , 
\end{equation} for  $1 \leq a \leq 4j$ and $1\leq b\leq 2j+1$. Thus, it is sufficient to compute the $(1,1)$ entry of 
    \beqa
 R_{13}^{(\h,j_2)}(u q^{-j_1+\h}) R_{23}^{(j_1-\h,j_2)}(u q^{\h})  \  . \label{blockR}   
 \eeqa 
As we have by~\eqref{R-Rqg}
\begin{equation}\label{eq:R-1c}
     R^{(\h,j)}(u)_{1,c}= R^{(\h,j)}(u)_{c,1}=
  \prod_ {\ell = 0}^{2 j - 1} c (uq^{j+\h  - \ell})\delta_{1,c}, \qquad 1 \leq c \leq 4j+2\ ,
\end{equation}
the presence of $\delta_{1,c}$ here implies that the $(1,1)$  entry of~\eqref{blockR} is the product of the (1,1) entries of $R^{(\h,j_2)}(uq^{-j_1+\h})$ and $R^{(j_1-\h,j_2)}(u q^{\h})$. First, from the definition~\eqref{renormR} we see that $\tilde{R}^{(\h,j)}(u)_{11} =1$. We then proceed by induction in $j_1$ and 
  get ${\tilde R}^{(j_1,j_2)}(u)_{11}=1$ indeed. Specializing $j_1=j_2$, we obtain $\tilde{R}^{(j,j)}(1)_{11} =1$, which finally together with~\eqref{eq:RjjeqP} shows~\eqref{eq:R-P}.
\end{enumerate}
\end{proof}

\subsection{Properties of K-matrices} \label{sec:prop-K} In this section, we first show that the spin-$j$ K-matrices defined by~\eqref{fused-K} satisfy a pair of intertwining relations already studied in the literature~\cite{DN02},  which allows us to uncover some of the properties of a normalized K-matrix.

\subsubsection{Intertwining relations} 

 The spin-$j$ K-matrix  $K^{(j)}(u)$ satisfies a pair of intertwining relations that is derived as follows.
    First of all by~\cite[Prop.\ 5.14]{LBG}, the spin-$j$ K-operator satisfies the intertwining relations  for any~$b\in\cal A_q$:
    \beqa
    \cal K^{(j)}(u)(\id \otimes \pi^{j})[\delta_{u^{-1}}(b)] = (\id \otimes \pi^{j})[\delta_u(b)]\cal K^{(j)}(u)\ ,\label{eq:intwK}
    \eeqa
    where $\delta_u\colon\cal A_q \rightarrow \cal A_q \otimes \Uq$ is the so-called evaluated coaction, see \cite[Eq.\,(4.70)]{LBG}. Actually, for our purpose it is convenient to collect those relations using the generating functions $\tW_\pm(v),\tG_\pm(v)$ given by~\eqref{c1},~\eqref{c2} with the substitution $u\rightarrow v$. They read:
\beqa
   && \cal K^{(j)}(u)(\id \otimes \pi^{j})[\delta_{u^{-1}}(\tW_\pm(v))] = (\id \otimes \pi^{j})[\delta_u(\tW_\pm(v))]\cal K^{(j)}(u)\ ,\label{eq:intwcurKW} \\
   && \cal K^{(j)}(u)(\id \otimes \pi^{j})[\delta_{u^{-1}}(\tG_\pm(v))] = (\id \otimes \pi^{j})[\delta_u(\tG_\pm(v))]\cal K^{(j)}(u)\ .\label{eq:intwcurKG}
    \eeqa
    
     Now, recall the algebra map $\epsilon_0: \cal A_q \rightarrow \mathbb C$ from Example~\ref{mapQSP}.
 Let us take the image of~\eqref{eq:intwcurKW},~\eqref{eq:intwcurKG}  under  $\epsilon_0 \otimes \id$. On one hand,  we specialize~\eqref{eq:Kaqj-eps} for $\epsilon\rightarrow \epsilon_0$ which gives $(\epsilon_0 \otimes \id){\cal K}^{(j)}(u) \propto K^{(j)}(u)$. On the other hand, we compute
 \beqa\label{eq:eps-W-cur}
 (\epsilon_0 \otimes \pi^{j})[\delta_{u}(\tW_\pm(v))]\quad \mbox{and} \quad (\epsilon_0 \otimes \pi^{j})[\delta_{u}(\tG_\pm(v))]\ 
 \eeqa
 using  explicit expressions of $\delta_{u}(\tW_\pm(v)),\delta_{u}(\tG_\pm(v))$ given in~\cite[Eqs.\,(5.37), (5.38)]{LBG} together with
 \beqa
 \epsilon_0(\tW_\pm(v)) = \frac{V^{-1}}{1-V^{-2}}(\varepsilon_+ + V^{-1}\varepsilon_-) \ ,\quad
 \epsilon_0(\tG_\pm(v)) = 0 \ ,\nonumber
 \eeqa
where $V=\frac{qv^2+q^{-1}v^{-2}}{q+q^{-1}}$, which is a generating function version of~\eqref{eq:eps-0}. The final result for~\eqref{eq:eps-W-cur} then takes  the form
\beqa
 (\epsilon_0 \otimes \pi^{j})[\delta_{u}(\tW_\pm(v))]&=& \frac{V^{-1}}{1-V^{-2}}\left(O^{\pm(0)}(u)+ V^{-1}O^{\pm(1)}(u) + V^{-2}O^{\pm(2)}(u)\right)\ ,\label{eq:epsdelW}\\
 (\epsilon_0 \otimes \pi^{j})[\delta_{u}(\tG_\pm(v))]&=& V^{-1}O^{\pm(3)}(u)\ ,\label{eq:epsdelG}
\eeqa
where $O^{\pm(i)}(u)$, for $i=0,...,3$, are Laurent polynomials in $u$ of order 2 and  with coefficients given by linear and quadratic combinations of the matrices $S_\pm$ and $q^{S_3/2}$ defined in~\eqref{Bdef}. Inserting~\eqref{eq:epsdelW},~\eqref{eq:epsdelG} into the image of~\eqref{eq:intwcurKW},~\eqref{eq:intwcurKG} under the map $\epsilon_0 \otimes \id$, and using the fact that $V$ is an indeterminate, we get the system of $8$ equations for the spin-$j$ K-matrix
\beqa
 K^{(j)}(u)O^{\pm(i)}(u^{-1}) =  O^{\pm(i)}(u)K^{(j)}(u)\ \label{eq:intwO} \ ,\qquad i=0,...,3\ .
\eeqa

We now show that this system collapses to a pair of independent relations.
For $i=0$, we get the pair of intertwining relations:
\beqa
&& K^{(j)}(u) \left[ k_+q^{\h}u S_+ q^{S_3/2}  + k_-q^{-\h}u^{-1} S_- q^{S_3/2} + \varepsilon_+ q^{S_3} \right]\ \label{eq:int1}\\
     && \qquad \qquad = \left[ k_+q^{\h}u^{-1} S_+ q^{S_3/2}  + k_-q^{-\h}u S_- q^{S_3/2} + \varepsilon_+ q^{S_3} \right] K^{(j)}(u)  \nonumber\ ,\\
     && K^{(j)}(u) \left[ k_+q^{-\h}u^{-1} S_+ q^{-S_3/2}  + k_-q^{\h}u S_- q^{-S_3/2} + \varepsilon_- q^{-S_3} \right] \label{eq:int2}\ \\
    && \qquad \qquad =  \left[ k_+q^{-\h}u S_+ q^{-S_3/2}  + k_-q^{\h}u^{-1} S_- q^{-S_3/2} + \varepsilon_- q^{-S_3} \right] K^{(j)}(u)  \nonumber\ .
    \eeqa
    For $i=1$, one easily shows that the matrix $O^{\pm(1)}(u)$ is a linear combination of $(u^2+u^{-2}){\mathbb I}_{2j+1}$ (invariant under the change $u\rightarrow u^{-1}$) and the Casimir operator on the spin-$j$ representation of $\Uq$. Thus,~\eqref{eq:intwO} for $i=1$ do not give any constraint on $K^{(j)}(u)$. 
Furthermore, we find that the matrix $O^{\pm(2)}(u)$ is a linear combination of $O^{\pm(0)}(u)$ and $O^{\pm(1)}(u)$, with coefficients that do not depend on $u$, while 
$O^{\pm(3)}(u)$ is a linear combination of $[O^{\mp(0)}(u),O^{\pm(0)}(u)]_q$ and a scalar term. 
Thus,~\eqref{eq:intwO} for $i=2,3$ do not give any new constraints on $K^{(j)}(u)$, besides~\eqref{eq:int1},~\eqref{eq:int2}.

All together, we conclude that the image of~\eqref{eq:intwcurKW},~\eqref{eq:intwcurKG}  under the algebra map  $\epsilon_0 \otimes \id$ reduces to the pair of intertwining relations~\eqref{eq:int1},~\eqref{eq:int2} satisfied by the spin-$j$ K-matrix.
This pair of intertwining relations has been previously studied in the literature, see~\cite[eqs.\,(3.5),\,(3.4)]{DN02} with  correspondence between our conventions and those in~\cite{DN02} given by
    \beqa
    u \rightarrow e^u\ , \quad q \rightarrow e^\eta\ ,\quad \varepsilon_\pm \rightarrow \pm e^{\pm \xi}\ ,\quad \frac{k_\pm}{q-q^{-1}} \rightarrow \kappa\ .\label{map-conv}
    \eeqa

\subsubsection{Normalized K-matrices and their asymptotics}
    \begin{lem}\label{lem:limK} 
    \beqa
    \lim_{\varepsilon_+\rightarrow \infty} \frac{K^{(j)}(u)_{11}}{\varepsilon_+^{2j}}=u^{2j}\prod_{\ell=0}^{2j-2}c(u^2q^{1-\ell}) \ .\label{eq:Kjlim}
    \eeqa
    \end{lem}
    
    \begin{proof}
    The proof is done by induction on $j$. 
    We first notice from~\eqref{KM-spin1/2} that the limit~\eqref{eq:Kjlim} holds indeed for $j=\h$:
    \begin{equation*}
    \lim_{\varepsilon_+\rightarrow \infty} \frac{K^{(\h)}(u)_{11}}{\varepsilon_+}=u \ .
   \end{equation*}
Thus, assume~\eqref{eq:Kjlim}  holds for $j$ replaced by $j-\h$  for some fixed $j>1$. 
    To show~\eqref{eq:Kjlim}, we compute explicitly the $(1,1)$ entry of $K^{(j)}(u)$  using the r.h.s.\ of~\eqref{fused-K}, divided by $\varepsilon_+^{2j}$ and taking the limit $\varepsilon_+ \rightarrow \infty$.
    To do so, we first recall~\eqref{eq:E-F-1a}. It is thus sufficient to compute the $(1,1)$ entry of 
    \beqa
 K_1^{(\h)}(u q^{-j+\h}) R^{(\h,j-\h)}(u^2 q^{-j+1}) K_2^{(j-\h)}(u q^{\h})\ . \label{block}   
 \eeqa
 Then, due to presence of $\delta_{1,c}$ in~\eqref{eq:R-1c}  the (1,1) entry of~\eqref{block} is the product of the (1,1) entries of $K^{(\h)}(u)$, $K^{(j-\h)}(u)$ and $R^{(\h,j-\h)}(u)$. Now, taking the limit of such product and under the induction hypothesis we get the result as in~\eqref{eq:Kjlim}.
    \end{proof}
  Similarly to normalized R-matrices~\eqref{renormR},  let us introduce the \textit{normalized} spin-$j$ K-matrix: \beqa
\tilde K^{(j)}(u) = \left(\prod_{\ell=0}^{2j-2}c(u^2q^{1-\ell})\right)^{-1} K^{(j)}(u) \ ,\label{renormK}
\eeqa
with the convention on $c(u)^{-1}$ in~\eqref{eq:cu-inv}.

\begin{lem}\label{lem:K-poly}
    The entries of the normalized K-matrix $\tilde K^{(j)}(u)$ are Laurent polynomials in $u$.
\end{lem}
\begin{proof}
From the definition~\eqref{renormK} and using the fusion~\eqref{fused-K} we obtain
	\begin{align*} 
        \tilde{K}^{(j)}(u) &= \frac{c(u^2 q^2)}{c(u^2q^{-2j+4})c(u^2q^{-2j+3})} 
        \mathcal{F}^{(j)} K_1^{(\h)}(u q^{-j+\h}) R^{(\h,j-\h)}(u^2 q^{-j+1}) \tilde{K}_2^{(j-\h)}(u q^{\h}) \mathcal{E}^{(j)} \ ,
	\end{align*}
while the R-matrix in this expression has its entries Laurent polynomials in $u$ times the overall factor
    $\prod_{k=0}^{2j-3} c(u^2 q^{-k})$ which proves the statement.
\end{proof}

This normalization ensures that $\tilde{K}^{(j)}(u)$ specialized at $u=1$ is non-zero.
Indeed, we have the following corollary of Lemma~\ref{lem:limK}:
\begin{cor}\label{cor:lim-K11-norm}
    \beqa
    \lim_{\varepsilon_+\rightarrow \infty} \frac{\tilde{K}^{(j)}(u)_{11}}{\varepsilon_+^{2j}}=u^{2j} \ ,\label{eq:Kjlim-renorm}
    \eeqa
and therefore for generic parameters $\varepsilon_\pm$ and $k_\pm$ the normalized K-matrix $\tilde{K}^{(j)}(u)$ specialized at $u=1$ is non-zero.
\end{cor}

In~\cite{DN02},  entries of the K-matrix 
solving the intertwining relations~\eqref{eq:int1},\eqref{eq:int2}, are expressed in terms of  solutions of certain recursion relations. 
It is however not clear from~\cite{DN02} whether those entries are uniquely determined for generic values of the boundary parameters $k_\pm$, etc. Nevertheless, the entries drastically simplify and become uniquely (up to a scalar) determined in two cases which lead to the following results:
\begin{enumerate}
\item Fixing $u=0$ in~\cite[eq.\,(3.8)]{DN02}, one finds 
that the unique (up to a scalar) solution of the intertwining relations~\eqref{eq:int1},\eqref{eq:int2} is a diagonal matrix, and we thus get too: 
\beqa\label{eq:K-delta}
\tilde K^{(j)}(1)_{n,m}\propto \delta_{n,m}\ .
\eeqa
\item From the limit $\xi \rightarrow \infty$  and $\kappa=0$ of~\cite[eq.\,(3.8)]{DN02} and using~\eqref{map-conv}, as well as Corollary~\ref{cor:lim-K11-norm}, we get
\beqa
\lim_{\varepsilon_+\rightarrow \infty}\frac{\tilde K^{(j)}(u)_{mm}}{\varepsilon_+^{2j}} =u^{2j-2m+2} \ ,\quad m=1,...,2j+1\ .\label{eq:limKj}
\eeqa

\end{enumerate} 

 From explicit calculations for small values of spin $j=\h,1,\frac{3}{2},2$, see e.g.~\eqref{eq:Kh},~\eqref{eq:K1} we come to the following natural conjecture for all values of $j$: 
\beqa
\tilde K^{(j)}(1)\propto {\mathbb I}_{2j+1}\ .
\eeqa

\subsubsection{Symmetry property}
We see from above Example~\ref{ex:K-spin1} and from~\eqref{KM-spin1/2}
 that $K^{(j)}(u)=\left(K^{(j)}(u)\right)^t\rvert_{k_+ \leftrightarrow k_-}$ for $j=\h,1$. This symmetry actually holds for all spins:
			\begin{equation}
				\label{symKj}
				\lbrack K^{(j)}(u) \rbrack^t \Big |_{k_\pm \rightarrow k_\mp} = K^{(j)}(u) \ .
			\end{equation}
Indeed, the proof is similar to the one of Lemma~\ref{lem:symRj1j2}:
			taking the transpose of~\eqref{fused-K} and using properties of the intertwining operators from Lemma~\ref{lem:rel}, we obtain the equalities
		\begin{align*}
			\lbrack K^{(j)}(u) \rbrack^t &= \lbrack \mathcal{E}^{(j)}_\fu \rbrack^{t} \ \lbrack K_2^{(j-\h)}(u q^\h) \rbrack^t \   \lbrack R_{12}^{(\h,j-\h)}(u^2 q^{-j+1})  \rbrack^{t}  \  \lbrack K_1^{(\h)}(u q^{-j+\h}) \rbrack^t \ \lbrack \mathcal{F}^{(j)}_\fu \rbrack^{t} \\
			& = \mathcal{F}_\fu^{(j)} \ R_{12}^{(\h,j-\h)}(q^{j}) \ K_2^{(j-\h)}(u q^\h)  \  R_{12}^{(\h,j-\h)}(u^2 q^{-j+1}) \ K_1^{(\h)}(u q^{-j+\h}) \ \mathcal{E}_\fu^{(j)} \lbrack \ \mathcal{H}_\fu^{(j)} \rbrack^{-1} \ .
		\end{align*}
		Then, using the following reflection equation, which is an easy consequence of~\eqref{REKop} and the unitarity property  of the R-matrices~\cite[eq.\,(4.44)]{LBG},
		\begin{equation*}
			R^{(j_1,j_2)}(v/u) K_2^{(j_2)}(v) R^{(j_1,j_2)}(uv)K_1^{(j_1)}(u) = K_1^{(j_1)}(u) R^{(j_1,j_2)}(uv) K_2^{(j_2)}(v) R^{(j_1,j_2)}(v/u) \ ,
		\end{equation*}
		for $(j_1,j_2)=(\h,j-\h)$ and $u\rightarrow u q^{-j+\h}$, $v\rightarrow u q^\h$, we find that $K^{(j)}(u)$ is invariant with respect to the transposition.

\subsection{Dual K-matrix}
Finally, the spin-$j$ K-matrix solution of the dual reflection equation (\ref{REKdual})  can be constructed as \cite{Skly88}
	\begin{equation} \label{def:kp}
		K^{+(j)}(u)= \frac{1}{f^{(j)}(u)}\left(K^{(j)}(u^{-1}q^{-1})\right)^t\bigg\rvert_{\varepsilon_\pm \rightarrow  \ov{\varepsilon}_\mp,  k_\pm \rightarrow - \ov{k}_\mp } \ ,
	\end{equation}
	where the normalization factor $f^{(0)}(u)=f^{(\h)}(u)=1$ and
	\begin{equation} \label{eq:fj} f^{(j)}(u)= \displaystyle{\prod_{k=1}^{2j-1}} \displaystyle{\prod_{\ell=1}^k}  c(u^2 q^{k+\ell+2-2j}) c(u^{-2} q^{-k-\ell+2j})  \ 
	\end{equation}
	are introduced for further convenience. Here, to distinguish a K-matrix solution to its dual, we use the parameters $\overline{\varepsilon}_\pm, \overline{k}_\pm\in\mathbb{C}$. From~\eqref{def:kp} and~\eqref{fused-K} the spin-$j$ dual K-matrices for any $j\in \h\mathbb{N}_+$ are given by the recursion
	\begin{equation}\label{fused-KP}
	K^{+(j)}(u) = \frac{f^{(j-\h)}(u q^{-\h})}{f^{(j)}(u)} \big [ \mathcal{E}^{(j)}_\fu  \big ]^t   K_2^{+(j-\h)} (u q^{-\h}) R^{(\h,j-\h)}(u^{-2} q^{-j-1}) K_1^{+(\h)}(u q^{j-\h}) \big [ \mathcal{F}^{(j)}_\fu \big ]^t \ 
	\end{equation}
	with $K^{+(0)}(u)=1$ and $K^{+(\h)}(u)$ is given by \eqref{def:kp} with \eqref{KM-spin1/2} .

\subsection{Quantum determinants} \label{sec:qdet}
Central elements of the reflection algebra (\ref{REKop}) are derived from the so-called quantum determinant of the K-operator following \cite{Skly88}. In the  analysis below, we will only need the spin-$\h$ case. As usual, define ${\cal P}^-_{12}=(1-{\cal P})/2$ where $\cal P=R^{(\frac{1}{2},\frac{1}{2})}(1)/(q-q^{-1})$ and with the spin-$\h$ R-matrix from~\eqref{Rhh}.
The quantum determinant  associated with the reflection equation~\eqref{REKop}  is given by:
\begin{equation}
	\Gamma(u)=\normalfont{\text{tr}}_{12}\big({\cal P}^{-}_{12} {\cal K}_1^{(\h)}(u)  R^{(\frac{1}{2},\frac{1}{2})}(qu^2)  {\cal K}_2^{(\h)}(u q) \big) \ ,  \label{gammaform}
\end{equation}
where   $\normalfont{\text{tr}}_{12}$ stands for the trace over $({\mathbb C}^2)_1\otimes ({\mathbb C}^2)_2$. 
	\begin{prop}[\cite{BasBel,Ter21}] \label{prop:qdet} The quantum determinant  
		\begin{equation}
			\Gamma(u)= 
			\frac{(u^2q^2-u^{-2}q^{-2})}{2(q-q^{-1})}\left( \Delta(u) - \frac{2\rho}{q-q^{-1}}\right)\label{gamma}
		\end{equation}
		with $\rho$ from~\eqref{rho} and
		\begin{align}
			\Delta(u)&= -(q-q^{-1})(q^2+q^{-2})\Big(\cW_+(u)\cW_+(uq) + \cW_-(u)\cW_-(uq)\Big) \nonumber\\
			& +(q-q^{-1})(u^2q^2+u^{-2}q^{-2})\Big(\cW_+(u)\cW_-(uq) + \cW_-(u)\cW_+(uq)\Big) \nonumber \\
			& - \frac{(q-q^{-1})}{\rho} \Big(\cG_+(u)\cG_-(uq) + \cG_-(u)\cG_+(uq)\Big) - \ \cG_+(u) - \cG_+(uq) - \cG_-(u) - \cG_-(uq)  \label{deltau}
		\end{align}
		is  such that  $\big[\Gamma(u),({\cal K}^{(\frac{1}{2})}(v))_{ij}\big]=0$  for all $i,j\in\{1,2\}$.  
	\end{prop}
	Note that the explicit expression of central elements of ${\cal A}_q$ in terms of the alternating generators~$\{{\tW}_{-k},{\tW}_{k+1}$, ${\tG}_{k+1},\tilde{\tG}_{k+1}|k\in{\mathbb N}\}$ are obtained through the expansion of $\Delta(u)$ as a formal power series  in  $u^{-2}$ \cite{BasBel}:
	\begin{equation}
		\Delta(u)  = \displaystyle{\sum_{k=0}^\infty}  u^{-2k-2} c_{k+1} \Delta_{k+1} \quad \mbox{with} \quad  c_{k} = - \frac{ (q+q^{-1})^{k} (q^{k} + q^{-k})} {q^{2k}} \ . \label{eq:ckp1}
	\end{equation}
	 As shown in \cite{Ter21}, the center ${\cal Z}$ of ${\cal A}_q$ is generated by $\{\Delta_{k+1}\}_{k=0}^{\infty}$.

We note that on the one-dimensional representation of ${\cal A}_q$ defined by~\eqref{eq:Kaq-KM}  the quantum determinant takes the value  
		\begin{equation} \label{eq:eps-gam}
			\epsilon(\Gamma(u)) = \varsigma^{(\h)}(u) \varsigma^{(\h)}(u q)  \, \Gamma_-(u)  \ 
		\end{equation}
	with
	\beqa
		\label{gammaKM} \Gamma_-(u)
		=  c(u^2) \left(  \varepsilon_+^2+ \varepsilon_-^2+ \varepsilon_+ \varepsilon_-(u^2q^2+u^{-2}q^{-2}) - k_+ k_- \frac{c(u^2q^2)^2 }{(q-q^{-1})^2} \right) \ ,
	\eeqa
	where
	$c(u)$ is given in~\eqref{eq:cu}.
		\begin{example}\label{detQSP} Consider the map $\epsilon_0$ from Example \ref{mapQSP}. Then, using \eqref{eq:varsigh} one gets:  
			\beqa
			\epsilon_0(\Gamma(u))= \frac{(q+q^{-1})^2}{c(u^2)c(u^2q^2)}\Gamma_-(u)\ .
			\eeqa
	\end{example}
Later in the text, we will also need the value of the dual quantum determinant associated with the dual reflection equation~\eqref{REKdual} of the  spin-$\h$ dual K-matrix.  It is defined    as follows:
\begin{align}
	\label{gammaKP} \Gamma_+(u) &= \normalfont{\text{tr}}_{12} (\mathcal{P}_{12}^-K_1^{+(\h)}(uq) R^{(\h,\h)}(u^{-2}q^{-3})K_2^{+(\h)}(u)  ) \\ \nonumber
	&=  \Gamma_- (u^{-1} q^{-2})   \bigg\rvert_{\varepsilon_\mp \rightarrow\ov{\varepsilon}_\pm ,   k_\mp  \rightarrow -\ov{k}_\pm}\ .
\end{align}
 This expression is obtained from   \eqref{gammaform} by  replacing $ {\cal K}^{(\h)}(u) \rightarrow K^{(\h)}(u)$ with \eqref{KM-spin1/2}, shifting $u \rightarrow u^{-1}q^{-1}$, then using the invariance of the trace by cyclicity and under transpose, the dual reflection equation for $j=\h$, and  identifying \eqref{def:kp}. Finally, replacing $u\rightarrow uq$,  one gets~\eqref{gammaKP}.

	\section{Generating functions \texorpdfstring{$\bt^{(j)}(u)$}{T(j)(u)} and universal TT-relations for \texorpdfstring{${\cal A}_q$}{Aq}}\label{sec3}
	In this section,
	  we construct a family of generating functions $\bt^{(j)}(u)$ with  $j \in \h \mathbb{N}_+$, for mutually commuting elements in ${\cal A}_q$. Following \cite[Thm.\,1]{Skly88}, the generating function is built from the matrix product of the fused K-operator ${\cal K}^{(j)}(u)$ for ${\cal A}_q$ given by~\eqref{fused-K-op} and the fused dual K-matrix with scalar entries $K^{+(j)}(u)$ given by (\ref{def:kp}): 
	\begin{equation}
	\bt^{(j)}(u) = \normalfont{\text{tr}}_{V^{(j)}}\bigl(K^{+{(j)}}(u){\cal K}^{(j)}(u)\bigr) 
	\ ,\label{tg} 
	\end{equation}
	and the trace is taken over the space $V^{(j)}\equiv {\mathbb C}^{2j+1}$ corresponding to those matrices. In particular, $\bt^{(0)}(u) =1$ because $K^{+{(0)}}(u)={\cal K}^{(0)}(u)=1$. We call $\bt^{(j)}(u)\in {\cal A}_q((u^{-1}))$ the {\it universal} spin-$j$ transfer matrix in view of its connection with actual transfer matrices on spin chains discussed in Section \ref{sec5}.

	Below, we show that the commuting family $\{\bt^{(j)}(u)\}$ is generated using universal TT-relations starting from the fundamental generating function $\bt^{(\h)}(u)$. 
	In Section~\ref{sub:TT-conj}, we prove the universal TT-relations~\eqref{normTT} for any $j \in \h \mathbb{N}_+$ in Theorem~\ref{TTrel} assuming one technical relation involving the fused K-operators from Definition~\ref{def:fusedK}. This relation  was checked for a few first values of spins and proven in general in~\cite{LBG} based on a very explicit and fundamental conjecture~\cite[Conj.\,6.1]{LBG} on relation of the fused K-operators to a universal K-matrix of $\mathcal{A}_q$. 
	 In Section~\ref{sub:TT-PBW},  we provide an independent proof of the universal TT-relation for $j=1$ and $j=\tha$ based on the use of  a PBW basis for $\mathcal{A}_q$ and without any assumptions made.
	  As a corollary of the TT relations, any generating function $\bt^{(j)}(u)$ defined by (\ref{tg}) is  a polynomial of order $2j$ in the currents $\mathsf{I}(u)$ given by (\ref{t12init}) with shifted arguments. 
      In Section~\ref{sec:normalTT}, we introduce normalized version of $\bt^{(j)}(u)$ and observe that the resulting normalized TT-relations reflect the product of evaluation representation of $\Loop$ in the Grothendieck ring.
      Finally, in Section~\ref{sub:TT-qOA}, we give the TT-relations for the $q$-Onsager algebra generated by $\cW_0,\cW_1$, which is a quotient of $\mathcal{A}_q$.

	\subsection{The universal TT-relations} \label{sub:TT-conj}	
			Here, we show that the universal TT-relation  holds for any $j \in \h \mathbb{N}_+$ provided a conjecture is true. This conjecture is now briefly recalled. We refer the reader to~\cite{LBG} for details and precise definitions of the universal K-matrix.
			 \smallskip

			 It is well-known that a universal R-matrix $\mathfrak{R}$ which satisfies a set of axioms~\cite{Dr0} leads to R-matrices through specializations. For instance, for $\mathfrak{R}$ in the completion of 
			 $\Uqhat \otimes \Uqhat$, defining  $\mathcal{R}^{(j_1,j_2)}(u) \in {\mathbb C}[[u^{-1}]]\otimes \End(\mathbb{C}^{2j_1+1})\otimes \End(\mathbb{C}^{2j_2+1})$  as the image of $\mathfrak{R}$ under the tensor product of evaluation representation of $\Uqhat$, one has~\cite{Boos2012}
			\begin{equation} \label{eq:evalR}
			 \mathcal{R}^{(\h,\h)}(u) = \pi^{\h}(\mu(u)) R^{(\h,\h)}(u) \ ,
			\end{equation}
			where  $\pi^\h(\mu(u))$ is the scalar function
				\begin{equation} \label{pimu}
				\pi^{\h}(\mu(u)) = u^{-1}q^{-\h} \exp \left ( \displaystyle{\sum_{k=1}^\infty}\frac{q^{2k} + q^{-2k}} {1+q^{2k}} \frac{u^{-2k}}{k} \right )	\ ,
			\end{equation}
			 and $R^{(\h,\h)}(u)$ is given in~\eqref{Rhh}, see~\cite[Example 4.3]{LBG}.

			 Similarly to the universal R-matrix, we assume  existence of a universal K-matrix $\mathfrak{K} \in {\cal A}_q \otimes \Loop$  (or rather in an appropriate completion of the tensor product)  that respects the axioms in \cite[Def.\,2.7]{LBG} involving certain twist. In our particular case of the comodule algebra ${\cal A}_q$, the choice of the twist is described in~\cite[Sec.\,4.2]{LBG}. Based on supporting evidences given in \cite[Sec.\,6]{LBG}, we expect that the fused K-operators ${\cal K}^{(j)}(u)$ can be obtained through specialization of $\mathfrak{K}$. More precisely, recalling the definition of ${\bf K}^{(j)}(u) \in \mathcal{A}_q((u^{-1})) \otimes \End(\mathbb{C}^{2j+1})$  from~\cite[Def.\,4.17]{LBG} as the image of $\mathfrak{K}$ under the spin-$j$ formal evaluation representation of $\Loop$, 
             we conjectured\footnote{We omit the expression of the evaluated coaction from~\cite[Conj.\,6.1]{LBG} since it is not needed here. 
              In~\cite{LBG}, the conjecture was formulated for all spins $j$ and not  as here just for $j=\h$. We observed that the statement in \cite[Conj.\,6.1]{LBG}  for all spins $j>\h$ is actually a corollary, here formulated as Corollary~\ref{cor:conj1}.} the following. 
	\begin{conj}[\cite{LBG}] \label{conj1} We have
		\begin{equation}
			{\bf K}^{(\h)}(u) =   \nu^{(\h)}(u)   {\cal K}^{(\h)}(u)\ ,\label{evalKj-half}
		\end{equation}
		where $\mathcal{K}^{(\h)}(u)$ is defined in~\eqref{K-Aq} and $\nu(u)$ is an invertible central element in ${\cal A}_q[[u^{-1}]]$ defined by the functional relation
		\begin{equation} \label{funct-nu}
			\quad \pi^{\h}(\mu(u^2 q)) \nu(u) \nu(u q)  \Gamma(u)  = 1 \ ,
		\end{equation}
		where $\Gamma(u)$ is given in~\eqref{gammaform} and $\pi^{\h}(\mu(u))$ in~\eqref{pimu}.
	\end{conj}
 We note that   the functional relation \eqref{funct-nu} has a unique -- up to a sign -- solution  $\nu(u)$ as a power series in $u^{-2}$, and it is given in~\cite[Lem.\,6.2]{LBG}. Furthermore, it follows from the analysis in~\cite[Sec.\,6]{LBG} that the all spins $j$ statement analogous to~\eqref{evalKj-half} is actually a corollary of Conjecture~\ref{conj1}:

	\begin{cor}\label{cor:conj1}
    For $j\in \h \mathbb{N}$, we have
		\begin{equation}
			{\bf K}^{(j)}(u) =   \nu^{(j)}(u)   {\cal K}^{(j)}(u)\ ,\label{evalKj}
		\end{equation}
		where $\mathcal{K}^{(j)}(u)$ is defined in~\eqref{fused-K-op} with  $\nu^{(\h)}(u) \equiv \nu(u)$ and
		\begin{align} \label{exp-nujh}
			\nu^{(j)}(u) &= \displaystyle{\left ( \prod_{m=0}^{2j-1}   \nu(uq^{j-\h-m})  \right) } \left ( \displaystyle \prod_{k=0}^{2j-2} \prod_{\ell=0}^{2j-k-2} \pi^{\h} (\mu (u^2q^{2j-2-2k-\ell})) \right ) \ .
		\end{align}
	\end{cor}
	We now show the  universal  TT-relations for any spin-$j$, which is the main result of this paper.   
	\begin{thm}\label{TTrel} Assume Conjecture~\ref{conj1}.
		The following TT-relations hold for all $j \in \h{\mathbb N}_+$:
		\begin{equation} \label{TT-rel}
			\bt^{(j)}(u) = \bt^{(j-\h)}(u q^{-\h}) \bt^{(\h)}(u q^{j-\h}) + \frac{\Gamma (u q^{j-\tha}) \Gamma_+ (uq^{j-\tha})}{ c(u^2 q^{2j}) c(u^2 q^{2j-2}) } \bt^{(j-1)}(u q^{-1})  \ 
		\end{equation}
		with $\bt^{(0)}(u) =  1$  and the  quantum determinants (\ref{gammaform}), (\ref{gammaKP}).
	\end{thm}

 	The proof of this theorem is quite technical and is presented in Appendix~\ref{Ap:proofT34}, while here we explain main ideas of it.  The proof essentially relies on two types of equations, one is in Lemma~\ref{lem:TT} that can be phrased ``Fusion of K-operators agrees with the reduction property of K-operators", and the other is the 2nd K-operator expression~\eqref{v2fused-K-op} obtained from the opposite coproduct. Both these equations were proven using Conjecture~\ref{conj1}, that is why the assumption made in Theorem~\ref{TTrel}, but they can also be checked or proven directly, at least for the first few values of spins till $j=2$ which is indeed the case~\cite[Prop.\,5.9\,\&\,Rem.\,5.11]{LBG}. 

    Another key property used in the proof is that $R^{(\frac{1}{2},j)}(q^{j+\frac{1}{2}})$ admits a decomposition in terms of $\mathcal{E}^{(j+\frac{1}{2})}$, $\mathcal{F}^{(j+\frac{1}{2})}$, and $\mathcal{H}^{(j+\frac{1}{2})}$, along with some properties they satisfy as detailed in Appendix~\ref{Ap:propEHF}.
	The main steps   of the proof are the   following. First, in the definition of $\bt^{(j)}(u)$ in~\eqref{tg}, replace $\mathcal{K}^{(j)}(u)$ and $K^{+(j)}(u)$ with their fused formulas given in~\eqref{fused-K-op} and~\eqref{fused-KP}, respectively. Then, using the relation $\mathcal{E}^{(j)} \mathcal{F}^{(j)} = \mathbb{I}_{4j} - \bar{\mathcal{E}}^{(j-1)} \bar{\mathcal{F}}^{(j-1)}$, express $\bt^{(j)}(u)$ as a sum of two terms.
	Finally,  the first term is identified with $\bt^{(j-\h)}(u q^{-\h}) \bt^{(\h)}(u q^{j-\h})$, while the second one
	is computed using two relations from Lemma~\ref{lem:TT}.
\vspace{1mm}

	\begin{rem} \label{rem:TT-invar}
	The  universal  TT-relations~\eqref{TT-rel} are invariant under a change of normalization of $\mathcal{K}^{(\h)}(u)$. Indeed, let $\lambda^{(\h)}(u)$ be a central Laurent series in ${\cal A}_q((u^{-1}))$ and $\mathcal{K}^{(\h)}(u) \rightarrow \lambda^{(\h)}(u) \mathcal{K}^{(\h)}(u)$. Recall $\mathcal{K}^{(j)}(u)$ is defined via the fusion relation~\eqref{fused-K-op}, and therefore we have a change of normalization $\mathcal{K}^{(j)}(u) \rightarrow \lambda^{(j)}(u) \mathcal{K}^{(j)}(u)$. It implies that
	\begin{align}  \label{eq:gj}
		\lambda^{(j)}(u) = \lambda^{(j-\h)}(u q^\h) \lambda^{(\h)}(u q^{-j+\h}) 
			= \displaystyle{\prod_{k=0}^{2j-1}} \lambda^{(\h)}(u q^{j-\h-k}) \ , 
	\end{align}
	where we used that $\mathcal{K}^{(j)}(u)$ is invertible~\cite[Rem.\,5.10]{LBG}. Note that the normalization of $\Gamma(u)$ defined in (\ref{gammaform}) also changes accordingly.
	Now, consider the  universal  TT-relation~\eqref{TT-rel} under this change of normalization of $\mathcal{K}^{(\h)}(u)$, it is given by
	\begin{align} \nonumber \label{TT-rel-renor}
	 \bt^{(j)}(u) &= \frac{\lambda^{(j-\h)}(u q^{-\h})\lambda^{(\h)}(uq^{j-\h})}{{\lambda^{(j)}(u)} }  \bt^{(j-\h)}(u q^{-\h}) \bt^{(\h)}(u q^{j-\h}) \\
	 &+ \frac{\lambda^{(\h)}(u q^{j-\tha}) \lambda^{(\h)}(u q^{j-\h}) \lambda^{(j-1)}(u q^{-1}) }{\lambda^{(j)}(u)} \frac{\Gamma (u q^{j-\tha}) \Gamma_+ (uq^{j-\tha})}{ c(u^2 q^{2j}) c(u^2 q^{2j-2}) } \bt^{(j-1)}(u q^{-1})  \ ,
	\end{align}
	and using~\eqref{eq:gj}, the latter equation simplifies to~\eqref{TT-rel}.
	\end{rem}

\begin{rem}\label{rem:secondTT}
    We notice that a slightly different TT relation equally holds:
    \begin{equation} \label{secondTT-rel}
    \bt^{(j)}(u) = \bt^{(j-\h)}(u q^{\h}) \bt^{(\h)}(u q^{-j+\h}) + \frac{\Gamma (u q^{-j+\h}) \Gamma_+ (uq^{-j+\h})}{ c(u^2 q^{-2j+2}) c(u^2 q^{-2j+4}) } \bt^{(j-1)}(u q)  \ .
\end{equation}
This relation can be proven along the same lines as for Theorem~\ref{TTrel} with the two key relations, like in Lemma~\ref{lem:TT}: one is just the K-operator relation~\cite[Eq.\,(6.29)]{LBG}, while the second is obtained by applying 
	 $(\epsilon \otimes \id)$ to~\cite[Eq.\,(6.28)]{LBG} instead of~\cite[Eq.\,(6.29)]{LBG}, recall the definition of $\epsilon$ from Proposition~\ref{prop:eps}, and rewriting the resulting relation in terms of the dual K-matrix $K^+(u)$. We show at the end of Section~\ref{sec:normalTT} that the TT-relation~\eqref{secondTT-rel} is actually a consequence of the TT-relation~\eqref{TT-rel}.
\end{rem}

	We now use the universal TT-relations to provide explicit mode expansion of $\bt^{(j)}(u)$ in terms of the generators of~${\cal A}_q$. We first recall from~\cite{Skly88} that the spin-$\h$ universal transfer matrices commute:
	\begin{equation}\label{eq:Th-commute}
		\big[ \bt^{(\h)}(u), \bt^{(\h)}(v) \big]=0 \ .
	\end{equation}
    Using (\ref{K-Aq}) and~\eqref{def:kp} at $j=\h$, one gets the expression for $\bt^{(\h)}(u)$ as in~\eqref{t12init} with
	\begin{align}\label{Igen}
		\mathsf{I}(u)& =  \overline{\varepsilon}_+ \mathsf{W}_+(u) +\overline{\varepsilon}_- \mathsf{W}_-(u) + \frac{1}{q^2-q^{-2}} \left ( \frac{\overline{k}_+}{k_+} \mathsf{G}_-(u) + \frac{\overline{k}_-}{k_-} \mathsf{G}_+(u)  \right) \ , \\
		\mathsf{I}_0 &=  \frac{\rho}{(q-q^{-1})(q^2-q^{-2})} \left ( \frac{\overline{k}_+}{k_+} +\frac{\overline{k}_-}{k_-} \right) \ ,\label{I0}
	\end{align}
	and $\rho$ is given in~\eqref{rho}.
	Inserting (\ref{c1}), (\ref{c2}) into (\ref{Igen}),
    we arrive at the following definition.

\begin{defn}\label{def:subalg-I}
The commutative subalgebra $\mathcal{I}\subset {\cal A}_q$ is generated by the modes of $\bt^{(\frac{1}{2})}(u)$, or equivalently by the mutually commuting elements $\{\mathsf{I}_{2k+1}\}_{k\in{\mathbb N}}$ defined by
\begin{equation}\label{eq:def-Ik}
    	\mathsf{I}_{2k+1}=\overline{\varepsilon}_+{\tW}_{-k} + \overline{\varepsilon}_-{\tW}_{k+1} + \frac{1}{q^2-q^{-2}} \frac{\overline{k}_-}{k_-}{\tG}_{k+1}  + \frac{1}{q^2-q^{-2}} \frac{\overline{k}_+}{k_+}\tilde{\tG}_{k+1}\ .
\end{equation}
\end{defn}

	The  generating functions $\bt^{(j)}(u)$ for $j=1, \frac 32$, follow from the  universal TT-relation (\ref{TT-rel}). In terms of the generating function~\eqref{Igen}, they read:
	\begin{align}
	\bt^{(1)}(u) &= c(u^{2}q) c(u^2 q^3)  \left(\mathsf{I}(u q^{\frac{1}{2}}) + \mathsf{I}_0 \right) \left(\mathsf{I}(u q^{-\frac{1}{2}}) + \mathsf{I}_0 \right) + f_0^{(1)}(u)\ ,\label{eq:TT-gen1}\\
	\bt^{(\tha)}(u) &=  c(u^2q^4)c(u^2q^2) c(u^2) ( \mathsf{I}(uq) + \mathsf{I}_0)(\mathsf{I}(u) + \mathsf{I}_0)(\mathsf{I}(u q^{-1}) + \mathsf{I}_0) \nonumber\\ \label{eq:TT-gen2} 
	&+ c(u^2q^4) f_0^{(1)}(u q^{-\h}) ( \mathsf{I}(uq)+\mathsf{I}_0) + c(u^2) f_0^{(\tha)}(u) (\mathsf{I}(u q^{-1}) + \mathsf{I}_0 )\ ,
	\end{align}
	where $c(u)$ is given in~\eqref{eq:cu} and
	\begin{equation}
	f_0^{(1)}(u)= \frac{\Gamma(u q^{-\h}) \Gamma_+(u q^{-\h})}{c(u^2q^2)c(u^2)}\ ,\qquad f_0^{(\tha)}(u) = \frac{\Gamma(u ) \Gamma_+(u)}{c(u^2q^3)c(u^2q)} \ .
	\end{equation}

    More generally, the  universal TT-relations \eqref{TT-rel} imply that $\bt^{(j)}(u)$ is a polynomial of order $2j$ in the generating functions  $\mathsf{I}(u)$ with shifted arguments, and coefficients that are central in $\mathcal{A}_q$. Furthermore, from the universal TT-relations, it follows that all $\bt^{(j)}(u)$ (or rather their modes in the expansion in $u$) belong to the same commutative subalgebra $\mathsf{I}$ generated by the modes $\mathsf{I}_{2k+1}$ of $\bt^{(\h)}(u)$. 
	We thus have the following corollary of Theorem \ref{TTrel} and the commutativity~\eqref{eq:Th-commute}.
	\begin{cor} 
		$\big[\bt^{(j)}(u), \bt^{(j')}(v)\big]=0$ for all $j,j'\in\h\mathbb{N}$. 
	\end{cor}

	\subsection{Proof for small $j$  using a PBW basis}	\label{sub:TT-PBW}
	Relaxing the assumption on the existence of a universal K-matrix from which the fused K-operators are derived, we show that the  universal  TT-relation~\eqref{TT-rel} holds for $j=1,\tha$, by a straightforward calculation using a PBW basis in $\mathcal{A}_q$ and the explicit expression of the fused K-operators. 
    We use the PBW basis based on the alternating generators~\cite{BasBel,Ter21}
	\begin{equation}\label{eq:WGGW}
	\{\tW_{-k}\}_{k\in {\mathbb N}}\ ,\quad \{\tG_{\ell+1}\}_{\ell\in {\mathbb N}}\ ,\quad \{\tilde\tG_{m+1}\}_{m\in {\mathbb N}} \ ,\quad \{\tW_{n+1}\}_{n\in {\mathbb N}} 
	\end{equation}
	and the following linear order $<$ on them: 
	\begin{equation}
	\tW_{-k}<   \tG_{\ell+1}< \tilde\tG_{m+1} < \tW_{n+1}\ , \quad k,\ell,m,n \in {\mathbb N}\ ,\label{order}
	\end{equation}
 Note that due to the commutativity relations~\eqref{qo2} and~\eqref{qo10} we don't need to specify an order within each of the four families in~\eqref{eq:WGGW}. Then, the PBW basis for ${\cal A}_q$ consists of the products $x_1x_2\cdots x_n$, for any $n\in {\mathbb N}$ and with $x_i$'s from the set~\eqref{eq:WGGW} and such that  $x_1 \leq x_2 \leq \cdots \leq x_n$.

According to this choice of PBW basis, we get the following ordering relations  \cite[App.\,B]{LBG}. 
	\begin{align*}
		\cG_-(v) \cG_+(u) &= \cG_+(u) \cG_-(v)+\rho(q^2-q^{-2}) \Big ( \cW_+(u) \cW_+(v) - \cW_-(u) \cW_-(v)\\ \nonumber
		&+\frac{1-UV}{U-V} ( \cW_+(u) \cW_-(v) - \cW_+(v) \cW_-(u)) \Big ) \ , \\ 
		\cW_-(v)\cW_+(u)&=\cW_+(u) \cW_-(v) + \frac{1}{V-U} \Big ( \frac{(q-q^{-1})}{\rho(q+q^{-1})} \big (\cG_+(u) \cG_-(v) - \cG_+(v) \cG_-(u) \big )\\ \nonumber
		&+\frac{1}{q+q^{-1}} \big ( \cG_+(u)-\cG_-(u)+\cG_-(v)-\cG_+(v) \big ) \Big )\ , \\
		\cW_-(v) \cG_+(u) &= \frac{q}{U-V} \Big ( \big(U q^{-1} - V q\big) \cG_+(u) \cW_-(v) -(q-q^{-1}) \bigl( ( \cW_+(u) \cG_+(v)- \cW_+(v)\cG_+(u)  \\ \nonumber
		&- U \cG_+(v)\cW_-(u) \bigr)+ \rho \big (U \cW_-(u) - V \cW_-(v) - \cW_+(u) + \cW_+(v) \big ) \Big ) \ , \\
		\cW_-(v)\cG_-(u)&= \frac{1}{q(V-U)}\Big ( \big ( V q^{-1}-U q\big) \cG_-(u) \cW_-(v) -(q-q^{-1})\big ( \cW_+(u)\cG_-(v) - \cW_+(v)\cG_-(u) \\\nonumber
		&- U \cG_-(v) \cW_-(u) \big)+ \rho \big ( U \cW_-(u)-V\cW_-(v)-\cW_+(u)+\cW_+(v) \big ) \Big) \ , \\
	\cG_+(v)\cW_+(u) &= \frac{q}{U-V} \Big ( \big( U q - Vq^{-1} \big ) \cW_+(u) \cG_+(v)-(q-q^{-1})\big (\cG_+(v)\cW_-(u)- \cG_+(u)\cW_-(v) \\
		& + V\cW_+(v)\cG_+(u) \big)+ \rho \big ( U \cW_+(u) - V \cW_+(v)-\cW_-(u)+\cW_-(v) \big) \Big) \ ,\\ 
		\cG_-(v) \cW_+(u)&=\frac{1}{q(V-U)} \Big ( \big (Vq-Uq^{-1}\big)\cW_+(u) \cG_-(v) -(q-q^{-1}) \big ( \cG_-(v) \cW_-(u) - \cG_-(u)\cW_-(v) \\\nonumber
		&+V \cW_+(v)\cG_-(u) \big)+ \rho \big ( U \cW_+(u) - V \cW_+(v)-\cW_-(u)+\cW_-(v) \big) \Big) \ ,
	\end{align*}
	where $V=(qv^2+q^{-1}v^{-2})/(q+q^{-1})$.
	
With these relations, we first show the  universal  TT-relation for $j=1$, using the   expression for $\bt^{(\h)}(u)$ given  in~\eqref{t12init} with~\eqref{Igen} and computing  $\bt^{(1)}(u)$ directly from the definition~\eqref{tg} in terms of the generating functions $\cW_\pm(u),\cG_\pm(u)$. For this we  use the explicit expression of the spin-1 K-operator given in~\eqref{expK-spin1} together with the spin-1 dual K-matrix in~\eqref{def:kp} and with~\eqref{KM-spin1}.
	   One finds that the difference $\bt^{(1)}(u)-\bt^{(\h)}(u q^{-\h})\bt^{(\h)}(u q^\h)$ reads as a quadratic combination of the  generating functions (\ref{c1}), (\ref{c2}). Using the ordering relations for the currents listed above, this difference simplifies to  $\Gamma(u q^{-\h}) \Gamma_+(u q^{-\h})/c(u^2q^2)c(u^2)$,  and where we used the quadratic expression of $\Gamma(u)$ in Proposition~\ref{prop:qdet}. This gives~\eqref{TT-rel} indeed. The universal  TT-relation for $j= \frac 32$ is checked along the same lines. The expressions being lengthy, a symbolic Mathematica program is used to simplify all expressions.

	\subsection{Normalized TT-relations}\label{sec:normalTT}
		In this section,  we derive {\it normalized} TT-relations. Recall the definition of the generating function  $\nu^{(j)}(u)$ of central elements in \eqref{exp-nujh}  that relates  the spin-$j$ fused K-operator ${\cal K}^{(j)}(u)$ to its normalized analog  ${\bf K}^{(j)}(u)$ in \eqref{evalKj}. 
			Define the {\it normalized} generating functions for mutually commuting elements in ${\cal A}_q$ by:
		\beqa
		\bt^{(j)}_{norm}(u) =  	 \normalfont{\text{tr}}_{V^{(j)}}(\bar K^{+{(j)}}(u)	{\bf K}^{(j)}(u) )
		\ ,\label{normTj}
		\eeqa
		where the normalized spin-$j$ dual K-matrix is given by
		\beqa
		\bar K^{+{(j)}}(u)= \nu_+^{(j)}(u) K^{+{(j)}}(u)\ ,
		\eeqa
		with \eqref{def:kp} and where we introduced the normalization factor
		\begin{align} \label{exp-nujph}
			\nu_+^{(j)}(u) &= \displaystyle{\left ( \prod_{m=0}^{2j-1}   \nu_+(uq^{j-\h-m})  \right) } \left ( \displaystyle \prod_{k=0}^{2j-2} \prod_{\ell=0}^{2j-k-2} \pi^{\h} (\mu (u^2q^{2j-2-2k-\ell})) \right )^{-1} \ 
		\end{align}
		where $\nu_+(u)$ is defined by the functional relation
		\begin{equation} \label{funct-nuplus}
			\quad \pi^{\h}(\mu(u^2 q^2)) \nu_+(u) \nu_+(u q)  \Gamma_+(u)  = -1 \ ,
		\end{equation}
		with~\eqref{pimu} and $\nu_+(u)\equiv \nu_+^{(\h)}(u)$. Note that the definition of $\nu_+^{(j)}(u)$ resembles the definition of $\nu^{(j)}(u)$  in~\eqref{exp-nujh} and in~\eqref{funct-nu}.
			\begin{cor} The following normalized TT-relations hold for all $j \in \h{\mathbb N}_+$:
			\begin{equation} \label{TT-relnorm}
				\bt^{(j)}_{norm}(u) = \bt^{(j-\h)}_{norm}(u q^{-\h}) \bt^{(\h)}_{norm}(u q^{j-\h}) - \bt^{(j-1)}_{norm}(u q^{-1})  \ 
			\end{equation}
			with $\bt^{(0)}_{norm}(u) =  1$.
		\end{cor}
		\begin{proof} Firstly, we obtain
				\beqa
			\nu^{(j)}(u) &=& 	\nu^{(j-\h)}(uq^{-\h}) 	\nu(uq^{j-\h})     \prod_{k=0}^{2j-2}  \pi^{\h} (\mu (u^2q^{k}))\ \label{nu-id1}
			\eeqa 
			using \eqref{exp-nujh}.  As $\nu^{(\h)}(u)$ is invertible, it follows $\nu^{(0)}(u)=1$. 
			Expressing the term $\nu^{(j-\h)}(u)$ with the use of~\eqref{nu-id1}, then rewriting the product $\nu(uq^{j-\frac{3}{2}})\nu(uq^{j-\frac{1}{2}})$ with the help of  the functional relation~\eqref{funct-nu}, and  applying the  formal power series relation
				$\pi^\h(\mu(u))\pi^\h(\mu(u q)) =\frac{1}{c(u)c(uq^2)}$,
we finally get
			\beqa
			\nu^{(j)}(u) &=& 	\nu^{(j-1)}(uq^{-1}) \frac{1}{ \Gamma(uq^{j-3/2})} \times\Bigl(\prod_{k=0}^{2j-3} c(u^2q^{k+1})c(u^2q^{k-1})\Bigr)^{-1}\ .
			\eeqa
			Secondly, we obtain
			\beqa
			\nu_+^{(j)}(u) &=& 	\nu_+^{(j-\h)}(uq^{-\h}) 	\nu_+(uq^{j-\h})   \left(   \prod_{k=0}^{2j-2}  \pi^{\h} (\mu (u^2q^{k}))\right)^{-1}\  
			\eeqa
			using~\eqref{exp-nujph}.
		Note that $\nu_+^{(0)}(u)=1$. Proceeding similarly as above, we get
			\beqa
			\nu_+^{(j)}(u) &=& -	\nu_+^{(j-1)}(uq^{-1})  \frac{c(u^2q^{2j})c(u^2 q^{2j-2})}{\Gamma_+(uq^{j-3/2})} \times  \prod_{k=0}^{2j-3}  c(u^2q^{k+1})c(u^2q^{k-1})
			\ . \label{nu-id4}
			\eeqa
		We then observe that $\bt^{(0)}_{norm}(u)=\bt^{(0)}(u)=1$.
		Multiplying both sides of the TT-relations in Theorem~\ref{TTrel} by $\nu_+^{(j)}(u)\nu^{(j)}(u)$ and using \eqref{evalKj} as well as the identities~\eqref{nu-id1}-\eqref{nu-id4},
		we get \eqref{TT-relnorm}.
			\end{proof}
            
	 We can rewrite the above TT relation~\eqref{TT-relnorm} as
	 \begin{equation} \label{TT-relnormGr}
	  \bt^{(\h)}_{norm}(u q^{j+1})
	  \bt^{(j)}_{norm}(u q^{\h})  =
	  \bt^{(j+\h)}_{norm}(u q) + \bt^{(j-\h)}_{norm}(u)  \ 
	 \end{equation}
	which shows that the normalised universal transfer matrices $\bt^{(j)}_{norm}(u)$ enjoy the  property of $q$-characters. That is, their multiplication  agrees with decomposition of the product  of corresponding evaluation representations of $\Loop$ used in the definition~\eqref{normTj} of $\bt^{(j)}_{norm}(u)$.
	More precisely, using the results of tensor product analysis in~\cite[Sec.\,3]{LBG}, we have a (non-split) short exact sequence
    $$
    0 \longrightarrow \mathbb{C}_{u}^{2j} \longrightarrow \mathbb{C}_{u q^{j+1}}^{2} 
	\otimes \mathbb{C}_{uq^{\h}}^{2j+1} \longrightarrow \mathbb{C}_{uq}^{2j+2} \longrightarrow 0
    $$
    Or in the sense of Grothendieck ring relations denoting by $[V]$ the split Grothendieck class of any $\Loop$-module~$V$, we have
	\begin{equation}
	    \label{eq:Gr-rels}
	\left[\mathbb{C}_{u q^{j+1}}^{2} 
	\otimes \mathbb{C}_{uq^{\h}}^{2j+1}\right] =  
	\mathbb{C}_{uq}^{2j+2}
	+
	\mathbb{C}_{u}^{2j}
	\end{equation}
	which is the `reduction' relation in the terminology of~\cite{LBG}. We thus see that the product~\eqref{TT-relnormGr} is exactly the same as the product~\eqref{eq:Gr-rels} of the corresponding representations in the Grothendieck ring.

    There is also the `fusion' relation in the Grothendieck ring
    	$$
	\left[\mathbb{C}_{u q^{-j-1}}^{2} 
	\otimes \mathbb{C}_{uq^{-\h}}^{2j+1}\right] =  
	\mathbb{C}_{uq^{-1}}^{2j+2}
	+
	\mathbb{C}_{u}^{2j}
	$$
    that suggests a 2nd TT-relation:
    \begin{equation} \label{secondTT-relnormGr}
      \bt^{(\h)}_{norm}(u q^{-j-1})
      \bt^{(j)}_{norm}(u q^{-\h})  =
      \bt^{(j+\h)}_{norm}(u q^{-1}) + \bt^{(j-\h)}_{norm}(u)  \ .
     \end{equation}
     This relation can be indeed obtained via a direct calculation using~\eqref{secondTT-rel}, recall discussion in Remark~\ref{rem:secondTT}, and the definition~\eqref{normTj}.
     The relation~\eqref{secondTT-relnormGr} is however not an independent one. Indeed, it is easy to see by induction that it follows from~\eqref{TT-relnormGr} and the fact that $\big[ \bt^{(\h)}_{norm}(u), \bt^{(\h)}_{norm}(v) \big]=0$.
	
	 The above property that the TT relations for the $\Loop$-comodule algebra $\mathcal{A}_q$ reflect multiplication of corresponding $\Loop$-modules in the Gro\-then\-dieck ring is analogous to the property~\cite[Thm.\,6.7.1]{AV24} observed in  a different context of quantum symmetric pairs.

	\subsection{The case of the \texorpdfstring{$q$}{q}-Onsager algebra} \label{sub:TT-qOA}
	We now study  the TT-relations  for the   $q$-Onsager algebra~$O_q$. Recall that $O_q$ is generated by $W_0,W_1$, subject to the $q$-Dolan-Grady relations \cite{Ter03,B1}: 
	\begin{equation}
		[W_0,[W_0,[W_0,W_1]_q]_{q^{-1}}]=\rho[W_0,W_1] \ ,\qquad
		[W_1,[W_1,[W_1,W_0]_q]_{q^{-1}}]=\rho[W_1,W_0] \ ,
		\label{qDGW} 
	\end{equation}
	where $\rho \in \mathbb{C}^*$.  As discussed in the series of works \cite{BasBel,Ter21,Ter21b}, $O_q$ is obtained by applying a central reduction of the algebra ${\cal A}_q$. 
 Recall Definition \ref{def:Aq0} and the center $\cal Z$ of ${\cal A}_q$ described below \eqref{eq:ckp1}. We have an isomorphism of algebras \cite{Ter21,Ter21b}:
\beqa
{\cal A}_q \cong O_q \otimes {\cal Z}\ . \label{isoAqOqZ}
\eeqa
Therefore, $O_q$ is obtained from ${\cal A}_q$ by taking the image of  $\Delta_{k+1}$'s to some scalars (a central reduction is applied to ${\cal A}_q$). Following \cite{BasBel}, introduce the surjective homomorphism 
\beqa
\gamma^{(\delta)}: {\cal A}_q \rightarrow O_q \qquad \mbox{such that}\qquad \gamma^{(\delta)}(\Delta_{k+1}) = 2\delta_{k+1}\ ,\qquad \delta_{k+1} \in {\mathbb C}\ ,\label{mapDeldel}
\eeqa
with \eqref{eq:ckp1}. Note that the $\Delta_{k}$'s are algebraically independent \cite
{Ter21} and so all $\delta_{k}$ can be fixed independently. Due to the isomorphism (\ref{isoAqOqZ}), any choice of scalars $\delta_{k}$  gives isomorphic quotients $O_q$. 

 The TT-relations for $O_q$ are thus obtained by taking the image of~\eqref{TT-rel}  under  the homomorphism~$\gamma^{(\delta)}$. To this end, below we successively consider  the image in $O_q$ of the fundamental generating function $\bt^{(\frac{1}{2})}(u)$ with (\ref{t12init}) and of the quantum determinant $\Gamma(u)$ with (\ref{gamma}), (\ref{eq:ckp1}). 

	 Firstly, we consider the image of the generating function $\bt^{(\frac{1}{2})}(u)$ under $\gamma^{(\delta)}$. The image of the alternating generators of ${\cal A}_q$ in $O_q$  is defined by \cite{BasBel,Ter21c}:
	\begin{equation}
	\tW_{-k}\mapsto {W}_{-k}\ ,\quad \tW_{k+1}\mapsto {W}_{k+1}\ ,\quad\tG_{k+1}\mapsto {G}_{k+1}, \quad \tilde{\tG}_{k+1}\mapsto\tilde{G}_{k+1}\ ,\label{mapAOq}
	\end{equation}
    which are certain polynomials in $\tW_0$ and $\tW_1$ that depend on $\delta_k$'s, as we now explain.
	In $O_q$, we say that for a fixed $k\in \mathbb{N}$  the alternating generators~\eqref{mapAOq} are of index $k$. Given the  index $k$ fixed, the alternating generators  can be expressed as polynomials of  alternating generators with $k$-indices of lower  or equal value, see~\cite[Prop.\,3.1]{BasBel}. In more details, let $\{\delta_k|k\in{\mathbb N}\}$ be fixed scalars, then we have
	\begin{align}
	G_{k+1} &= -\sum_{\ell=0}^{\lfloor\frac{k}{2}\rfloor}\sum_{i+j=2\ell+1-\overline{k+1}} \frac{f^{(k)}_{ij}}{2}{\mathbb W}_{ij}  - \sum_{\ell=0}^{\lfloor\frac{k}{2}\rfloor-\overline{k+1}}\sum_{i+j=2\ell+\overline{k+1}}  \frac{e^{(k)}_{ij}}{2}{\mathbb F}_{ij} -  \sum_{\ell=0}^{\lfloor\frac{k}{2}\rfloor-1}\frac{g^{(k)}_{\ell}}{2}{\mathbb G}_{2(l+1)-\overline{k+1}} \label{recGk}\\
	\qquad \qquad  & + \frac{(q+q^{-1})}{2}\big[W_{k+1},W_0\big] + \delta_{k+1} \ ,
	\nonumber
	\end{align}
	where  
	\begin{align}
	{\mathbb W}_{ij}&=(q-q^{-1})(W_{-i}W_{j+1} +W_{i+1}W_{-j} ) \ ,\nonumber\\
	{\mathbb F}_{ij}&=(q-q^{-1})\Big((q^2+q^{-2})(W_{-i}W_{-j}+W_{i+1}W_{j+1})+\frac{1}{\rho} (G_{j+1}\tilde{G}_{i+1} + \tilde{G}_{j+1}G_{i+1})\Big)\nonumber \ ,\nonumber\\
		{\mathbb G}_{i+1}&= G_{i+1} + \tilde{G}_{i+1} \ , \nonumber
	\end{align}
	and
	\begin{align*}
	e^{(k)}_{ij}&=-c_{k+1}^{-1}\sum_{m=0}^{\lfloor\frac{k}{2}\rfloor-\frac{i+j+\overline{k+1}}{2}}w_{\lfloor\frac{k}{2}\rfloor-\frac{i+j+\overline{k+1}}{2}-m}^iw_m^jq^{-2j-4m-2} \ ,\nonumber\\
	f^{(k)}_{ij}&=c_{k+1}^{-1}\sum_{m=0}^{\lfloor\frac{k}{2}\rfloor-\frac{i+j+\overline{k+1}-1}{2}}w_{\lfloor\frac{k}{2}\rfloor-\frac{i+j+\overline{k+1}-1}{2}-m}^i \Big(w_m^j+w_{m-1}^j\Big)q^{-2j-4m} \ ,\nonumber\\
	g^{(k)}_\ell &=-c_{k+1}^{-1} w_{\lfloor\frac{k}{2}\rfloor-\ell}^{2\ell+1-\overline{k+1}}(1+q^{-2k-2}) \ ,\nonumber 
	\end{align*}
	with  $c_k$ in \eqref{eq:ckp1} and
	\begin{equation}
 w_{m}^i=(-1)^m\frac{(m+i)!}{m!i!}(q+q^{-1})^{i+1}q^{-i-2m-1} \ .\nonumber
	\end{equation}
	Also, one has
	\begin{equation} {W}_{-k-1}=\frac{1}{\rho}\big[{W}_0,{G}_{k+1}\big]_q+{W}_{k+1} \ .  \label{recWk}
	\end{equation}
	Recall the automorphism $\Omega$ of $O_q$ such that $\Omega(W_0)=W_1$ and $\Omega(W_1)=W_0$. The expressions for the other alternating generators are given by:
	\begin{equation}\label{om}
	{W}_{k+1}=\Omega({W}_{-k})\ , \quad {W}_{-k}=\Omega({W}_{k+1})\ , \quad \tilde{G}_{k+1}=\Omega({G}_{k+1})\ , \quad {G}_{k+1}=\Omega(\tilde{G}_{k+1}) \ . 
	\end{equation}
	Iterating (\ref{recGk}), (\ref{recWk}) the alternating generators can be written as polynomials of $W_0,W_1$ and scalars~$\delta_k$. 
   \begin{example}
           For instance,
	\begin{align}
	\ {G}_{1} &= q{W_1}W_0-q^{-1}{W_0}W_1 + \delta_1 \ ,\label{defel}\\
	{W}_{-1} &= \frac{1}{\rho}\left( (q^2+q^{-2})W_0W_1W_0 -W_0^2W_1 - W_1 W_0^2\right) + W_1 + \frac{\delta_1(q-q^{-1})}{\rho}W_0 \ ,\nonumber\\
	{G}_{2} &= \frac{1}{\rho(q^2+q^{-2})} \Big( (q^{-3}+q^{-1}) W_0^2{W_1}^2 - (q^{3}+q){W_1}^2W_0^2 + (q^{-3}-q^{3})(W_0{W_1}^2W_0 + {W_1}W_0^2{W_1})  \nonumber\\
& - (q^{-5}+q^{-3} +2q^{-1}) W_0{W_1}W_0{W_1} + (q^{5}+q^{3} +2q){W_1}W_0{W_1}W_0 +  \rho(q-q^{-1})(W_0^2 + {W_1}^2
	)\Big) \nonumber\\
	&+\ \frac{\delta_1(q-q^{-1})}{\rho}\big(q{W_1}W_0-q^{-1}{W_0}W_1 \big)  + \delta_2- \frac{\delta_1^2(q-q^{-1})}{\rho(q^2+q^{-2})} \ ,\nonumber \\
	{W}_{-2} &= w_1W_0^3{W_1}^2 + w_2{W_0}^2W_1^2W_0  + w_3\big(W_1{W_0}^2W_1W_0 + W_0^2W_1W_0W_1\big) + w_4W_0W_1W_0W_1W_0 \nonumber \\
	&+ w_5{W_1}^2{W_0}^3+w_6W_0{W_1}^2{W_0}^2+w_7W_1{W_0}^3W_1\nonumber\\
&+ w_8W_0W_1^2+w_9 W_1^2W_0
	+ w_{10}\big(W_0^2W_1 +W_1W_0^2\big)+\ w_{11} W_1W_0W_1 + w_{12}W_0W_1W_0 \nonumber\\
	&+\ w_{13}W_0^3 + w_{14}W_0+ \ w_{15}W_1 \ , \nonumber
	\end{align}
	where
	\begin{align}
	w_1&=\frac{1}{\rho^2} , \quad w_2= - \frac{[2]_{q}[8]_{q}}{\rho^2[4]_{q}^2}, \quad w_3=-\frac{[4]_{q}}{\rho^2[2]_{q}},\quad  w_4=\frac{1}{\rho^2}\left(\frac{[2]_q[3]_q[8]_{q}}{[4]_q^2} +1\right),\nonumber\\
	 w_5&=\frac{1}{\rho^2[3]_q}, \quad w_6=-\frac{[2]_q[8]_q}{\rho^2[3]_q[4]_q},\quad w_7 =\frac{[2]^2_q}{\rho^2[3]_q[4]_q},\quad  w_8= -\ \frac{1}{\rho},\quad w_9=-\ \frac{1}{\rho[3]_q} \ ,\nonumber\\
	w_{10}&=- \frac{(q-q^{-1})}{\rho^2}\delta_1,\quad  w_{11}=\frac{1}{\rho^2[3]_q}\left(\frac{[2]_q[3]_q[8]_{q}}{[4]_q^2} +1\right),\quad w_{12}=\frac{(q-q^{-1})[4]_{q}}{\rho^2[2]_{q}}\delta_1 \ ,\quad \nonumber\\
	w_{13}&= \frac{1}{\rho}\frac{(q-q^{-1})^2[2]_{q}}{[4]_{q}}\ ,\quad  w_{14} =1-\frac{(q-q^{-1})^2[2]_{q}}{\rho^2[4]_{q}}\delta_1^2+\frac{(q-q^{-1})}{\rho}\delta_2,\quad w_{15}= \frac{(q-q^{-1})}{\rho}\delta_1\ .\nonumber
	\end{align}
       \end{example}

	\begin{rem}\label{deg} Let ${\rm deg}\colon O_q \rightarrow {\mathbb N}$ denote the total degree of a polynomial in the alternating generators $W_0,W_1$ with ${\rm deg}(W_{0})= {\rm deg}(W_1)=1$. We then have
		\begin{equation*}
			{\rm deg}(W_{-k})={\rm deg}(W_{k+1})=2k+1 \ , \qquad {\rm deg}(G_{k+1})={\rm deg}(\tilde{G}_{k+1})=2k+2   \ .
		\end{equation*} 
	\end{rem}
Let us define
 \beqa
  \bt^{(\frac{1}{2})}_{O_q}(u) = \gamma^{(\delta)}(\bt^{(\frac{1}{2})}(u)) \ . 
  \eeqa
	\begin{lem} The fundamental generating function for mutually commuting elements in $O_q$ is given by:
		\beqa
		\bt^{(\frac{1}{2})}_{O_q}(u) = c(u^2q^2)\left( \sum_{k=0}^\infty U^{-k-1} I_{2k+1}  + \mathsf{I}_0\right) \label{TOqfund}
		\eeqa 
		with \eqref{I0} and 
		\beqa
		 I_{2k+1}=\overline{\varepsilon}_+{W}_{-k} + \overline{\varepsilon}_-{W}_{k+1} + \frac{1}{q^2-q^{-2}} \frac{\overline{k}_-}{k_-}{G}_{k+1}  + \frac{1}{q^2-q^{-2}} \frac{\overline{k}_+}{k_+}\tilde{G}_{k+1} \ .\label{IkOq}
		\eeqa
	\end{lem}
\begin{proof}
We apply 	\eqref{mapAOq} to (\ref{t12init}) with $\bar{\kappa}_\pm=\frac{1}{q^2-q^{-2}} \frac{\overline{k}_\mp}{k_\mp}$ and with  $\mathsf{I}_{2k+1}$ given in~\eqref{eq:def-Ik}.
\end{proof}

In \eqref{IkOq}, the alternating generators of $O_q$ are defined recursively in terms of $W_0,W_1$.  Inserting the expressions (\ref{recWk}) and (\ref{recGk}) into (\ref{IkOq}), it follows:  												\begin{rem} The coefficients $I_{2k+1}$ 
are polynomials in $W_0,W_1$, such that 
\begin{equation}
	{\rm deg}(I_{2k+1})=2k+2\ .\label{degIk}
\end{equation} 
For instance for $k=0$ in~\eqref{IkOq}, we have
	\begin{equation}
		I_1= \overline{\varepsilon}_+ W_0 + \overline{\varepsilon}_- W_1  +\frac{1}{q^2-q^{-2}} \left ( \frac{\overline{k}_-}{k_-}  \big ( \lbrack W_1,W_0 \rbrack_q + \delta_1 \big )   +  \frac{\overline{k}_+}{k_+} \big ( \lbrack W_0,W_1\rbrack_q + \delta_1   \big ) \right ) \ .
	\end{equation}
For $k=1$ in \eqref{IkOq}, we use the expressions \eqref{defel} and \eqref{om}.
\end{rem}

Secondly, we turn to the image  
		   of the quantum determinant $\Gamma(u)$  under $\gamma^{(\delta)}$.
	\begin{lem}  \label{lem:qdetOq}  
	\begin{equation}
		\gamma^{(\delta)}(\Gamma(u))=
		\frac{c(u^2q^2)}{(q-q^{-1})}\left(  \sum_{k=1}^{\infty} c_k\delta_k u^{-2k} - \frac{\rho}{q-q^{-1}}\right)\ \in u^2{\mathbb C}[[u^{-2}]]\label{imgamma}\ ,
	\end{equation}
with $c_k$ from~\eqref{eq:ckp1}.
\end{lem}
\begin{proof}  Apply $\gamma^{(\delta)}$ to the expression of $\Gamma(u)$ given by (\ref{gamma}), where~\eqref{eq:ckp1} is used.  
\end{proof}
As an example, let us identify the values of the $\delta_k$'s such that the images of the quantum determinant of ${\cal A}_q$ under the map $\gamma^{(\delta)}$ in \eqref{mapAOq} and under the map $\epsilon_0$ of Example \ref{mapQSP} match. 
\begin{example}\label{gameps} For the central reduction $\gamma^{(\delta)}$ defined by \eqref{mapDeldel} such that
	\beqa
	\delta_{2p+2} &=&  (q+q^{-1})^{2}(q-q^{-1})(p+1) q^{-4p-4}c_{2p+2}^{-1}(\epsp^2+  \epsm^2) \ ,\qquad\label{delpair}\\
	\delta_{2p+1} &=&  (q+q^{-1})^{2}(q-q^{-1})(2p+1) q^{-4p-2}c_{2p+1}^{-1} \epsp\epsm \ , \quad p\in {\mathbb N}\ ,\label{delimp} 
	\eeqa
	we get:
	\beqa
	\gamma^{(\delta)}(\Gamma(u))&=&	\frac{(q+q^{-1})^2}{(q^2u^2-q^{-2}u^{-2})}\left( \epsp^2 + \epsm^2 + \epsp\epsm(q^2u^2+q^{-2}u^{-2}) - \frac{\rho(q^2u^2-q^{-2}u^{-2})^2}{(q^2-q^{-2})^2}\right)\ \label{eq:durho}\\
	&=& \frac{(q+q^{-1})^2}{c(u^2)c(u^2q^2)}\Gamma_-(u) \ .\nonumber
	\eeqa
	Comparing  with Example \ref{detQSP}, we see that the image
	$\gamma^{(\delta)}(\Gamma(u))$ agrees with the value $\epsilon_0(\Gamma(u))$.
\end{example}

\vspace{1mm}

 Finally, we obtain the TT-relations for $O_q$ as a corollary of Theorem \ref{TTrel}. Define the generating functions for mutually commuting elements in $O_q$ by:
 \beqa
 \bt^{(j)}_{O_q}(u) = \gamma^{(\delta)}(\bt^{(j)}(u)) \label{map:AqOq}
 \eeqa
 for the choice of $\gamma^{(\delta)}$ in Example \ref{gameps}.
\begin{cor} The following TT-relations hold for all $j \in \h{\mathbb N}_+$:
	\begin{equation} \label{TT-relOq}
	\bt^{(j)}_{O_q}(u) = \bt^{(j-\h)}_{O_q}(u q^{-\h}) \bt^{(\h)}_{O_q}(u q^{j-\h}) + \frac{(q+q^{-1})^2\Gamma_- (u q^{j-\tha}) \Gamma_+ (uq^{j-\tha})}{ \prod_{k=0}^3 c(u^2 q^{2j-k}) } \bt^{(j-1)}_{O_q}(u q^{-1})  \ 
\end{equation}
with $\bt^{(0)}_{O_q}(u) =  1$ and the  quantum determinants (\ref{gammaKM}), (\ref{gammaKP}).
\end{cor}
	
	\section{Truncated generating functions 	$\bt^{(j,N)}(u)$  and TT-relations for $\{\mathcal{A}_q^{(N)}\}_{N\in{\mathbb N}_+}$} \label{sec4} 
 In this section, we prepare the material to apply the  universal  TT-relations from Theorem~\ref{TTrel} to integrable open spin chains.
	In Section~\ref{sec:quotientAq}, we consider quotients of ${\cal A}_q$ denoted $\{\mathcal{A}_q^{(N)}\}_{N\in{\mathbb N}_+}$ via certain linear relations in generators of ${\cal A}_q$ that occur in spin chains of length $N$.
	   In Section~\ref{sec:fusKAqN}, introducing the corresponding algebra map  $\varphi^{(N)}\colon\mathcal{A}_q \rightarrow \mathcal{A}_q^{(N)}$, we derive spin-$j$ K-operators for $\mathcal{A}_q^{(N)}$ from $\mathcal{K}^{(j)}(u)$ given in Definition~\ref{def:fusedK}, see~\eqref{KjN}. Importantly for applications to spin-chains, entries of those K-operators are now Laurent polynomials in $u$.
	   In Section~\ref{sec:TjN}, we give the generating functions of mutually commuting elements in  $\mathcal{A}_q^{(N)}$ denoted $\bt^{(j,N)}(u)$, which are also Laurent polynomials in $u$, as well as their TT-relations.  In Section~\ref{sec:fusK-dresK}, the link to integrable open spin chains is established through the introduction of the maps $\psi^{(N)}_{\bj,\bar{v}}$ in Definition~\ref{def:psiN}.  They provide finite-dimensional representations of $\mathcal{A}_q^{(N)}$, and thus of $\mathcal{A}_q$, that depend on spins and boundary parameters chosen. Spin-$j$ K-operators for spin-chains are obtained using this map, see Proposition~\ref{prop:Kjspinchain}. 
	
	\subsection{The quotients \texorpdfstring{$\{\mathcal{A}_q^{(N)}\}_{N\in{\mathbb{N}_+ }}$}{Aq(N)}}  \label{sec:quotientAq}
	For the Onsager algebra, a class of quotients has been considered in \cite{DaviesA,DaviesB,Ro91}, that arises in the analysis of integrable systems such as the Ising or chiral Potts models. See also \cite[Sec.\,4]{BC18} for a description of those quotients in the alternating presentation of the Onsager algebra. Here, we consider an analogous construction for the alternating central extension of the $q$-Onsager algebra. 
	\begin{defn}\label{defAqN}
		Let $N$ be a non-negative integer. $\mathcal{A}_q^{(N)}$ is the quotient of $\mathcal{A}_q$ by the relations:
		\begin{align} \label{generalized-d1}
			\displaystyle{\sum_{k=0}^{N}}d_k^{(N)} \tW_{-k-\ell}+ \overline{\ell+1} \, \varepsilon_+^{(N)} + \overline{\ell} \, \varepsilon_-^{(N)}&=0 \ , \qquad \displaystyle{\sum_{k=0}^{N}}  d_k^{(N)} \tG_{k+1+\ell} =0 \ , \\ \label{generalized-d2}
			\displaystyle{\sum_{k=0}^{N}}d_k^{(N)} \tW_{k+1+\ell}+ \overline{\ell+1} \, \varepsilon_-^{(N)} + \overline{\ell} \, \varepsilon_+^{(N)}&=0 \ , \qquad  \,\displaystyle{\sum_{k=0}^{N}} d_k^{(N)} \tilde{\tG}_{k+1+\ell} =0 \ ,
		\end{align}
		for any $\ell \in \mathbb{N}$, we use $\overline{\ell}= \ell \mod 2$, and where 
		$d_k^{(N)}\in {\mathbb C}^*$, $\varepsilon_\pm^{(N)}\in {\mathbb C}$ .
	\end{defn}
	
	Note that the relations (\ref{generalized-d1}), (\ref{generalized-d2}), can be interpreted as $q$-deformed analogs of Davies' relations \cite{DaviesA,DaviesB} for the alternating central extension of the $q$-Onsager algebra. For $q=1$, see \cite[eq.\,(4.20)]{BC18}. \smallskip
	
	Let $\{  \tW^{(N)}_{-k}, \tW^{(N)}_{k+1}, \tG^{(N)}_{k+1},\tilde{\tG}^{(N)}_{k+1}|k=0,1,...,N-1\}$ denote the alternating generators of $\mathcal{A}_q^{(N)}$. Define a surjective homomorphism $ \varphi^{(N)}\colon \mathcal{A}_q \rightarrow \mathcal{A}_q^{(N)} $ by:
	\begin{equation}
	 \tW_{-k} \mapsto \tW^{(N)}_{-k}\ ,\quad \tW_{k+1} \mapsto \tW^{(N)}_{k+1}\ ,\quad \tG_{k+1} \mapsto \tG^{(N)}_{k+1}\ ,\quad \tilde{\tG}_{k+1} \mapsto \tilde{\tG}^{(N)}_{k+1} \quad \mbox{for $k=0,1,...,N-1$.}\label{mapAqAN}
	\end{equation}
	For $k\geq N$, the image of the other generators of ${\cal A}_q$ follows from (\ref{generalized-d1}), (\ref{generalized-d2}).

    \begin{example}\label{exp:Aq-N0}
     ${\cal A}_q^{(0)}$ is the trivial algebra ${\mathbb C}$ and the quotient maps $\varphi^{(0)}: {\cal A}_q \to {\cal A}_q^{(0)}$ agree with  the one-dimensional representations $\epsilon_0$ introduced in Example~\ref{mapQSP},  after setting $d_0^{(0)}=-1$ and
	$\varepsilon_\pm^{(0)} = \varepsilon_\pm$  in~\eqref{generalized-d1},~\eqref{generalized-d2}. 
	\end{example}
	
 The simplest non-trivial example of quotient is given at $N=1$ by the Askey-Wilson algebra.
	For a review on the Askey-Wilson algebra and how it arises in mathematical physics, see \cite{revAW}.   
	\begin{example}  
		\label{ex1}
The Askey-Wilson algebra AW is generated by   $A,A^*$ subject to the defining relations~\cite{Z91}  
		\begin{align}\label{AW1} [A,[A,A^*]_q]_{q^{-1}} &= \rho A^* + \omega A + \eta \ , \\ \label{AW2}
			[A^*,[A^*,A]_q]_{q^{-1}}&=\rho A+\omega A^* + \eta^* \ ,
		\end{align}
		where $\rho$, $\omega$, $\eta$, $\eta^*$ are scalars. There is an algebra isomorphism  $\mathcal{A}_q^{(1)} \rightarrow AW$ such that
		\begin{equation}
		{\tW}_{0}^{(1)}\mapsto A\ ,\quad
		{\tW}_{1}^{(1)}\mapsto A^*\ ,\quad
		{\tG}_{1}^{(1)}\mapsto [A^*,A]_q\ ,\quad
		\tilde{\tG}_{1}^{(1)}\mapsto [A,A^*]_q\ 
		\end{equation}
		for the choice  $d_1^{(1)}=\rho$, $d_0^{(1)}=\omega$, $\varepsilon_+^{(1)} = \eta$, $\varepsilon_-^{(1)} = \eta^*$ in ~\eqref{generalized-d1},~\eqref{generalized-d2}.
			More generally, using (\ref{generalized-d1}), (\ref{generalized-d2}) we also find that
		\begin{align}
			\normalfont{{\tW}}_{-k-1}^{(1)} &\mapsto \left (  - \frac{\omega}{\rho} \right)^{k+1} \left ( A + \frac{\eta}{\rho} \displaystyle{\sum_{\ell =0}^{\lfloor\frac{k}{2} \rfloor} \Big ( \frac{\omega}{\rho} \Big)^{-1-2\ell}  } - \frac{\eta^*}{\rho} \displaystyle{\sum_{\ell =0}^{\lfloor \frac{k-1}{2} \rfloor- \delta_{0,k}} \Big ( \frac{\omega}{\rho} \Big)^{-2(\ell+1)}  }   \right ) \ , \\
			\normalfont{{ \tW}}_{k+2}^{(1)} &\mapsto \left (  - \frac{\omega}{\rho} \right)^{k+1} \left ( A^* + \frac{\eta^*}{\rho} \displaystyle{\sum_{\ell =0}^{\lfloor \frac{k}{2} \rfloor} \Big ( \frac{\omega}{\rho} \Big)^{-1-2\ell}  } - \frac{\eta}{\rho} \displaystyle{\sum_{\ell =0}^{\lfloor \frac{k-1}{2} \rfloor - \delta_{0,k}} \Big ( \frac{\omega}{\rho} \Big)^{-2(\ell+1)}  }   \right ) \ , \\
			\normalfont{{ \tG}}_{k+1}^{(1)} &\mapsto \left ( - \frac{\omega}{\rho} \right )^k [A^*,A]_q \ , \qquad \tilde{\normalfont{{ \tG}}}_{k+1}^{(1)} \mapsto \left ( - \frac{\omega}{\rho} \right )^k [A,A^*]_q \ .
		\end{align}
	\end{example}	

	\subsection{Fused K-operators  $\mathcal{K}^{(j,N)}(u)$ for \texorpdfstring{$\mathcal{A}_q^{(N)}$}{Aq(N)}} \label{sec:fusKAqN}
 We now consider   images  of the fused K-operators of ${\cal A}_q$ under the algebra map $\varphi^{(N)}$ defined in~\eqref{mapAqAN}. 
  Due to the relations (\ref{generalized-d1}) and (\ref{generalized-d2}), in the quotient $\mathcal{A}_q^{(N)}$ the generating functions (\ref{c1}), (\ref{c2}) appropriately normalised truncate to finite sums. Indeed, introduce: 
\begin{equation} \label{WPMN}
	\begin{aligned}
		\CWP(u)&=\WMu \ , \qquad  \CWM(u)=\WPu \ , \\
		\cG_+^{(N)}(u)&=\Gu \ , \qquad \cG_-^{(N)}(u)= \tGu \ ,
	\end{aligned}
\end{equation}
where 
\begin{equation} \label{PMK}
	P_{-k}^{(N)}(u) = -\frac{1}{q+q^{-1}} \displaystyle{\sum_{n=k}^{N-1}} d_{n+1}^{(N)} U^{n-k} \ .
\end{equation}
and we recall $U$ from~\eqref{eq:U-def}
\begin{lem}\label{lemAqN}  The action of $\varphi^{(N)}$ on the generating functions of $\mathcal{A}_q$ is given by
	\begin{align*} \nonumber
		h_0^{(N)}(u) \varphi^{(N)}(\cW_+(u)) &= \cW_+^{(N)}(u)+ \frac{u^2q+u^{-2}q^{-1}}{(u^2-u^{-2})(u^2q^2-u^{-2}q^{-2})}\varepsilon_+^{(N)} +\frac{q+q^{-1}}{(u^2-u^{-2})(u^2q^2-u^{-2}q^{-2})}\varepsilon_-^{(N)} ,     \\
		h_0^{(N)}(u)  \varphi^{(N)}(\cW_-(u)) & = \cW_-^{(N)}(u)+ \frac{q+q^{-1}}{(u^2-u^{-2})(u^2q^2-u^{-2}q^{-2})}\varepsilon_+^{(N)} +\frac{u^2q+u^{-2}q^{-1}}{(u^2-u^{-2})(u^2q^2-u^{-2}q^{-2})}\varepsilon_-^{(N)}, \\ \nonumber
		\qquad h_0^{(N)}(u) \varphi^{(N)}(\cG_{\pm}(u)) &= \cG_\pm^{(N)}(u) \ , \nonumber
	\end{align*}
	with the following Laurent polynomial in $u^{2}$, with its inverse in ${\mathbb C}[[u^{-1}]]$,
	\begin{equation} \label{h0N}
		h_0^{(N)}(u)= -\frac{1}{q+q^{-1}}\left (\displaystyle{\sum_{k=0}^{N}} d_k^{(N)}U^k \right ).
	\end{equation}
\end{lem}
\begin{proof} 
	Firstly, consider the image of the generating functions~\eqref{c1},~\eqref{c2}. For instance, from~\eqref{c1} one gets:
	\begin{equation} \label{mapWP} \cW_+(u) \mapsto \displaystyle{\sum_{k\in \mathbb{N}}} \tW_{-k}^{(N)}U^{-k-1} = \displaystyle{\sum_{k=0}^{N-1}} \cW_{-k}^{(N)}U^{-k-1}+\displaystyle{\sum_{k=N}^{\infty}} \tW_{-k}^{(N)} U^{-k-1}\ ,
	\end{equation}
	where $U=(qu^2+q^{-1}u^{-2})/(q+q^{-1})$.
	Now, using the linear relation~\eqref{generalized-d1}, it follows:
	\begin{align} \nonumber\displaystyle{\sum_{k=N}^{\infty}}& \tW_{-k}^{(N)} U^{-k-1} = \displaystyle{\sum_{p=0}^{\infty}}\tW_{-N-p}^{(N)} U^{-N-p-1}\\ \nonumber
		&=-\displaystyle{\sum_{p=0}^{\infty} \sum_{k=0}^{N-1}} \frac{d_k^{(N)}}{d_N^{(N)}} \tW_{-k-p}^{(N)} U^{-N-p-1}-\displaystyle{\sum_{p=0}^{\infty}}\overline{p+1} \, U^{-p-1-N} \frac{\varepsilon_+^{(N)}}{d_N^{(N)}} - \displaystyle{\sum_{p=0}^{\infty}} \overline{p} \,  U^{-p-1-N} \frac{\varepsilon_-^{(N)}}{d_N^{(N)}} \\ \nonumber
		&=- \left (\displaystyle{\sum_{k=0}^{N-1}} \frac{d_k^{(N)}}{d_N^{(N)}} U^{k-N} \right ) \cW_+(u) +\displaystyle{\sum_{k=1}^{N-1} \sum_{p=0}^{k-1}} \frac{d_k^{(N)}}{d_N^{(N)}} \tW_{-p}^{(N)} U^{-N-p-1+k}- \displaystyle{\sum_{n=0}^{\infty}}U^{-2n-1-N} \left(\frac{\varepsilon_+^{(N)}}{d_N^{(N)}} + U^{-1}
		\frac{\varepsilon_-^{(N)}}{d_N^{(N)}}\right) \\ \nonumber 
		&=- \left (\displaystyle{\sum_{k=0}^{N-1}} \frac{d_k^{(N)}}{d_N^{(N)}} U^{k-N} \right ) \cW_+(u) 
		+\displaystyle{\sum_{k=0}^{N-1}}\tW_{-k}^{(N)} \left ( \displaystyle{\sum_{p=k+1}^{N-1}}\frac{d_p^{(N)}}{d_N^{(N)}}U^{p-N-k-1} \!\!\right )  \!-\! \displaystyle{\sum_{n=0}^{\infty}}U^{-2n-1-N}\!\! \left(\frac{\varepsilon_+^{(N)}}{d_N^{(N)}} + U^{-1}
		\frac{\varepsilon_-^{(N)}}{d_N^{(N)}}\!\right)\!.
	\end{align}
	Inserting the above expression into (\ref{mapWP}) and multiplying the result by $d_N^{(N)}$, one gets:
	\begin{align}\nonumber\left(\displaystyle{\sum_{k=0}^{N}} d_k^{(N)} U^k \right) \cW_+(u) &=\displaystyle{\sum_{k=0}^{N-1}} \tW_{-k}^{(N)} \left ( \displaystyle{\sum_{n=k}^{N-1}}d_{n+1}^{(N)}U^{n-k} \right )-\displaystyle{\sum_{n=0}^{\infty}}U^{-2n-1} \left(\varepsilon_+^{(N)} + U^{-1}
		\varepsilon_-^{(N)} \right) \\
		\nonumber
		&= \displaystyle{\sum_{k=0}^{N-1}} \tW_{-k}^{(N)} \left ( \displaystyle{\sum_{n=k}^{N-1}}d_{n+1}^{(N)}U^{n-k} \right ) \\ \nonumber
		& \, -\frac{(q+q^{-1})(u^2q+u^{-2}q^{-1})}{(u^2-u^{-2})(u^2q^2-u^{-2}q^{-2})} \varepsilon_+^{(N)} -  \frac{(q+q^{-1})^2}{(u^2-u^{-2})(u^2q^2-u^{-2}q^{-2})} \varepsilon_-^{(N)} \ .
	\end{align}
	Dividing both sides by $-(q+q^{-1})$, the expression of $\varphi^{(N)}(\cW_+(u))$  follows. The action of $\varphi^{(N)}$ on the other generating functions is obtained similarly.
\end{proof}

The spin-$\h$ K-operator for $\mathcal{A}_q^{(N)}$,  denoted $\mathcal{K}^{(\h,N)}(u)$, is defined as follows.\footnote{We note that though K-operators appear in the context of comodule algebras, we are not aware of any comodule algebra structure on $\mathcal{A}_q^{(N)}$.} Recall the spin-$\h$ K-operator for $\mathcal{A}_q$ in (\ref{K-Aq}). Taking its image by $\varphi^{(N)} \otimes \id$, we define:
\begin{equation}
\mathcal{K}^{(\h,N)}(u)=	(u^2-u^{-2})h_0^{(N)}(u) (\varphi^{(N)} \otimes \id) ( \mathcal{K}^{(\h)}(u) )\ ,\label{K12N}
\end{equation}
where the normalisation with $h_0^{(N)}(u)$ from~\eqref{h0N}  is chosen for later convenience.
Using Lemma \ref{lemAqN}, we explicitly find:
\begin{equation} \label{KN(u)}  
	\mathcal{K}^{(\h,N)}(u)=\begin{pmatrix}
		\mathcal{A}^{(N)}(u) & \mathcal{B}^{(N)}(u) \\ 
		\mathcal{C}^{(N)}(u)& \mathcal{D}^{(N)}(u)
	\end{pmatrix},
\end{equation}
with
\begin{align}\label{A(N)}\mathcal{A}^{(N)}(u)&= u \epsp^{(N)} + u^{-1} \epsm^{(N)} + (u^2 - u^{-2}) \big( u q \CWP(u) - u^{-1} q^{-1} \CWM(u) \big) \ , \\ \label{D(N)}
	\mathcal{D}^{(N)}(u)&= u\epsm^{(N)} +u^{-1} \epsp^{(N)} + (u^2-u^{-2}) \big (u q \CWM(u) -u^{-1} q^{-1} \CWP(u) \big) \ , \\ \label{B(N)}
	\mathcal{B}^{(N)}(u)&=\frac{ (u^2-u^{-2})}{k_-} \Bigl( \frac{ k_+ k_- ( u^2 q +u^{-2} q^{-1})}{q-q^{-1}} P_0^{(N)}(u)+\frac{1}{q+q^{-1}} \cG_+^{(N)}(u) + \omega_0^{(N)} \Bigr) \ , \\ \label{C(N)}
	\mathcal{C}^{(N)}(u)&= \frac{(u^2-u^{-2})}{k_+}\Bigl(\frac{k_+ k_- (u^2 q+u^{-2} q^{-1})}{q-q^{-1}} P_0^{(N)}(u) + \frac{1}{q+q^{-1}} \cG_-^{(N)}(u)+ \omega_0^{(N)} \Bigr) \ ,
\end{align}
where $P_0^{(N)}(u)$ is given in~\eqref{PMK} and 
\begin{equation}
\omega_0^{(N)}= - \frac{k_+ k_-}{q-q^{-1}} d_0^{(N)}\ .
\end{equation}

 We note that entries of $\mathcal{K}^{(\h,N)}(u)$ are Laurent \textsl{polynomials} in $u$, in contrast to the power series in the case of   $\mathcal{A}_q$,
because the truncated generating functions $\cW_\pm^{(N)}(u)$, $\cG_\pm^{(N)}(u)$ from~\eqref{WPMN} have only positive powers in $U$ up to the power $N-1$. This is the reason of the normalization introduced in~\eqref{K12N}. And we thus call $\mathcal{K}^{(\h,N)}(u)$ \textit{truncated} K-operator.

Using the fact that $\varphi^{(N)}$ is an algebra homomorphism and the definition of the truncated K-operator from~\eqref{K12N}, it is clear that $\mathcal{K}^{(\h,N)}(u)$ satisfies the reflection equation 
\begin{equation} \label{eq:REAqN}
	{R}^{(\frac{1}{2},\frac{1}{2})}(u/v) {\cal K}_1^{(\frac{1}{2},N)}(u) { R}^{(\frac{1}{2},\frac{1}{2})}(uv) {\cal K}_2^{(\frac{1}{2},N)}(v)
	=  {\cal K}_2^{(\frac{1}{2},N)}(v)  { R}^{(\frac{1}{2},\frac{1}{2})}(uv) {\cal K}_1^{(\frac{1}{2},N)}(u) { R}^{(\frac{1}{2},\frac{1}{2})}(u/v) .
\end{equation}

\begin{example}
For $N=0$, the  K-operator for ${\cal A}_q^{(0)}={\mathbb C}$  reduces to the K-matrix $K^{(\h)}(u)$ introduced in~\eqref{KM-spin1/2}. Indeed, as we noted in Example~\ref{exp:Aq-N0}, $\varphi^{(0)}$ agrees with $\epsilon_0$ introduced in Example~\ref{mapQSP}.
Then from the definition~\eqref{K12N} we get (notice that $h_0^{(0)}=(q+q^{-1})^{-1}$ because we have set $d_0^{(0)}=-1$)
    $$
    \mathcal{K}^{(\h,0)}(u) = \frac{c(u^2)}{q+q^{-1}} (\epsilon_0 \otimes \id) ( \mathcal{K}^{(\h)}(u) ) =  K^{(\h)}(u) \ ,
    $$
    where in the last equality we used Proposition~\ref{prop:Kaq-Km} and noticed that $\zeta^{(\h)}(u)= \frac{q+q^{-1}}{c(u^2)}$.
\end{example}

\begin{example}
	The fundamental K-operator for the  case  $N=1$ -- recall Example \ref{ex1} -- corresponds to the K-operator of the Askey-Wilson algebra~\cite[Prop.\,2.1]{BP19}, up to an overall scalar factor, with the identification $k_+=\chi(q+q^{-1})^{-1}$, $k_-=\rho\chi^{-1}(q+q^{-1})^{-1}$, with $\chi \in \mathbb{C}^*$.
\end{example}

\smallskip

The spin-$j$ fused K-operators associated with $\mathcal{A}_q^{(N)}$ and  denoted by $\mathcal{K}^{(j,N)}(u)$ are constructed similarly to $\mathcal{K}^{(j)}(u)$ from Definition~\ref{def:fusedK}. 

	\begin{defn}\label{def:jKN}
		We define   spin-$j$ fused K-operators for $\mathcal{A}_q^{(N)}$ as
		\begin{align} \label{fused-K-opN}
			\mathcal{K}^{(j,N)}(u) &=  \mathcal{F}^{(j)}_\fu \mathcal{K}_1^{(\h,N)}(u q^{-j+\h}) R^{(\h,j-\h)}(u^2 q^{-j+1}) \mathcal{K}_2^{(j-\h,N)}(u q^{\h}) \mathcal{E}^{(j)}_\fu  \ ,
		\end{align}
	with $\mathcal{K}^{(\h,N)}(u)$ in~\eqref{KN(u)}.
	\end{defn}

By construction they satisfy the reflection equation
\begin{align} \nonumber
	R^{(j_1,j_2)}(u/v) \mathcal{K}^{(j_1,N)}_1(u) R^{(j_1,j_2)}(u v) \mathcal{K}^{(j_2,N)}_2(v) &= \\  \mathcal{K}_2^{(j_2,N)}(v) R^{(j_1,j_2)}(uv)  \mathcal{K}_1^{(j_1,N)}(u)& R^{(j_1,j_2)}(u/v)\ .
\end{align}
Then, using~\eqref{fused-K-op},  one shows by induction the spin-$j$ analog of~\eqref{K12N}
\begin{equation}\label{KjN}
\mathcal{K}^{(j,N)}(u) = \left [ \displaystyle{\prod_{k=0}^{2j-1}} c(u^2 q^{2j-1-2k}) h_0^{(N)}(u q^{j-\h-k}) \right ] (\varphi^{(N)} \otimes \id ) (\mathcal{K}^{(j)}(u)) \ ,
\end{equation}

We finally notice that  by induction every matrix entry of r.h.s.\ of~\eqref{fused-K-opN} is a product of Laurent polynomials. Therefore,  all entries of the spin-$j$ fused K-operators are equally Laurent polynomials, i.e.\  $$
\mathcal{K}^{(j,N)}(u) \in \mathcal{A}_q^{(N)}[u,u^{-1}] \otimes \End(\mathbb{C}^{2j+1})\ , $$
and we call them simply \textit{truncated} fused K-operators.

\subsection{Truncated commuting generating functions}\label{sec:TjN}
We recall the commutative subalgebra $\mathcal{I}\subset{\cal A}_q$ from Definition~\ref{def:subalg-I}. Here, we study its image in $\mathcal{A}_q^{(N)}$ under the quotient map $\varphi^{(N)}$.
 Let us first introduce
\begin{equation}
	\bt^{(\h,N)}(u) = \normalfont{\text{tr}}_{{\mathbb C}^{2}}\bigl(K^{+{(\h)}}(u){\cal K}^{(\h,N)}(u)\bigr) \ ,
\end{equation}
similarly to $\bt^{(\h)}(u)$ in~\eqref{tg}.
Let us introduce the generating function
\begin{align}\label{eq:def-IN}
	\mathsf{I}^{(N)}(u) & =  \overline{\varepsilon}_+ \cW_+^{(N)}(u) +\overline{\varepsilon}_- \cW_-^{(N)}(u) + \frac{1}{q^2-q^{-2}} \left ( \frac{\overline{k}_+}{k_+} \cG_-^{(N)}(u) + \frac{\overline{k}_-}{k_-} \cG_+^{(N)}(u)  \right)
\end{align}
which is the same as
\begin{align}
    \mathsf{I}^{(N)}(u)  &= \sum_{k=0}^{N-1} P_{-k}^{(N)}(u) \mathsf{I}^{(N)}_{2k+1} \qquad \mbox{with} \qquad 
    \mathsf{I}^{(N)}_{2k+1}= \varphi^{(N)}(\mathsf{I}_{2k+1})\ ,\label{calIN}
\end{align}
and $\mathsf{I}_{2k+1}$ are as in~\eqref{eq:def-Ik},
while $P_{-k}^{(N)}(u)$ are introduced in~\eqref{PMK}.
Then, $\bt^{(\h,N)}(u)$ can be computed explicitly:
\begin{align}
	\bt^{(\h,N)}(u) &= c(u^2)h_0^{(N)}(u) \varphi^{(N)} \bigl( \bt^{(\h)}(u) \bigr)\ \nonumber \\  &=c(u^2)c(u^2q^2) \! \left( \! \mathsf{I}^{(N)}(u) + h_0^{(N)}(u) \mathsf{I}_0 \right ) \label{tjsIN}\\ \nonumber
	& \quad + \overline{\varepsilon}_+  \left (  (u^2 q + u^{-2} q^{-1})\varepsilon_+^{(N)} + (q+q^{-1}) \varepsilon_-^{(N)} \right ) \!+ \!\overline{\varepsilon}_- \left ( (u^2q+u^{-2}q^{-1}) \varepsilon_-^{(N)} + (q+q^{-1}) \varepsilon_+^{(N)}    \right ) \ ,
\end{align}
where $c(u)$ is given in~\eqref{eq:cu}, 
we used \eqref{K12N} for the first equality, and~\eqref{t12init} with~\eqref{Igen} and Lemma~\ref{lemAqN} for the second equality,
where the scalar $\mathsf{I}_0$ is from~\eqref{I0}.

We  notice from~\eqref{calIN} that $\bt^{(\h,N)}(u)\in \mathcal{A}_q^{(N)}[u,u^{-1}]$ and the commutative subalgebra $\mathcal{I}^{(N)}:=\varphi^{(N)} (\mathcal{I})$  in $\mathcal{A}_q^{(N)}$ extracted from the modes of
$\bt^{(\h,N)}(u)$
 is generated by finitely-many elements 
 $$\{\mathsf{I}_{2k+1}^{(N)} | k=0,1, \ldots, N-1\}\ .$$

The spin-$j$ version of $\bt^{(\h,N)}(u)$ is defined similarly to  $\bt^{(j)}(u)$ from~\eqref{tg}:
\begin{equation}
	\bt^{(j,N)}(u) = \normalfont{\text{tr}}_{{\mathbb C}^{2j+1}}\bigl(K^{+{(j)}}(u){\cal K}^{(j,N)}(u)\bigr)
	\ ,\label{tgintN} 
\end{equation}
where the truncated fused K-operator $\mathcal{K}^{(j,N)}(u)$ and $K^{+(j)}(u)$ are given in~\eqref{fused-K-opN} and~\eqref{def:kp}, respectively. We note that  $\bt^{(j,N)}(u)$ is a Laurent polynomial in $u$ and with coefficients from  $\mathcal{A}_q^{(N)}$ with an overall factor the scalar power series $1/f^{(j)}(u)$ introduced in~\eqref{eq:fj}, i.e.\ 
\begin{equation}\label{eq:TjN-space}
\bt^{(j,N)}(u) \in \frac{1}{f^{(j)}(u)}\mathcal{A}_q^{(N)}[u,u^{-1}]\ .    
\end{equation}

By Remark~\ref{rem:TT-invar}, the  TT-relations from Theorem \ref{TTrel} are invariant under a change of normalization of $\cal K^{(j)}(u)$. Applying the algebra map $\varphi^{(N)}$ with~\eqref{mapAqAN}  to~\eqref{TT-rel}, and taking into account the normalization~\eqref{KjN} and the invariance property, we find:
\begin{equation} 
	\bt^{(j,N)}(u) = \bt^{(j-\h,N)}(u q^{-\h}) \bt^{(\h,N)}(u q^{j-\h}) + \frac{\Gamma^{(N)} (u q^{j-\tha}) \Gamma_+ (uq^{j-\tha})}{ c(u^2 q^{2j}) c(u^2 q^{2j-2}) } \bt^{(j-1,N)}(u q^{-1})  \ ,\label{TT-relTN}
\end{equation}
with 
\begin{align}
	\Gamma^{(N)}(u) &= \normalfont{\text{tr}}_{12}\big({\cal P}^{-}_{12} {\cal K}_1^{(\h,N)}(u)  R^{(\frac{1}{2},\frac{1}{2})}(qu^2)  {\cal K}_2^{(\h,N)}(u q) \big) \ \label{GammaN}\\
	&\stackrel{\eqref{K12N}}{=}	c(u^2)c(u^2q^2)h_0^{(N)}(u)h_0^{(N)}(uq) \varphi^{(N)}(\Gamma(u))\ .\nonumber
\end{align}

As a corollary of the above TT-relations, we have that $\bt^{(j,N)}(u)$ is a polynomial of order $2j$ in $\mathsf{I}^{(N)}(u)$ with $q$-shifted arguments and with coefficients that are central in $\frac{1}{f^{(j)}(u)}\mathcal{A}_q^{(N)}[u,u^{-1}]$. Examples for first values of~$j$ are similar to the expressions in~\eqref{eq:TT-gen1} and~\eqref{eq:TT-gen2}.

\subsection{Spin-chain representations of $\cal A_q^{(N)}$}
\label{sec:fusK-dresK}
The fused K-operators are now related to the basic building ingredient for integrable open spin chains, namely the so-called dressed K-matrices that are solutions of the reflection equation (\ref{REKop}) according to the standard terminology used in \cite{Skly88} and related works. 
Following the conventions in~\cite{LBG}, introduce the Lax operator, as an element in  $\Uq[u,u^{-1}] \otimes \End(\mathbb{C}^2)$:
\begin{equation} \label{Laxh}
	\mathcal{L}^{(\h)}(u)=
	\begin{pmatrix}
		u q^{\h} K^{\h}-u^{-1}q^{-\h} K^{-\h} & (q-q^{-1}) F \\ 
		(q-q^{-1}) E & u q^{\h} K^{-\h}-u^{-1}q^{-\h} K^{\h}
	\end{pmatrix}  
	\ ,
\end{equation}
where $E$, $F$, $K^{\pm \h}$ are the generators of $\Uq$ that satisfy the defining relations:
	\begin{equation} \label{Uqsl2}
	K^\h E=q EK^\h\ , \qquad K^\h F=q^{-1}FK^\h \ , \qquad [E,F]=\frac{K-K^{-1}}{q-q^{-1}}\ , \qquad K^\h K^{-\h}=K^{-\h}K^\h=1 \ .
\end{equation}

Fused spin-$j$ Lax operators are given by 
\cite[Eq.\,(4.47)]{LBG}: 
\begin{equation}\label{eq:spin-j-fusedL}
\mathcal{L}^{(j)}(u)= \mathcal{F}^{(j)}_\fu \mathcal{L}_1^{(\h)}(u q^{-j+\h}) \mathcal{L}_2^{(j-\h)}(u q^\h) \mathcal{E}^{(j)}_\fu\  \ .
\end{equation}
They  belong to $\Uq[u,u^{-1}] \otimes \End(\mathbb{C}^{2j+1})$ and satisfy 
\begin{equation}
R^{(\h,j)}(u/v)\mathcal{L}_1^{(\h)}(u)\mathcal{L}_2^{(j)}(v)=\mathcal{L}_2^{(j)}(v)\mathcal{L}_1^{(\h)}(u)R^{(\h,j)}(u/v) \   \label{RLL}
\end{equation}
with the R-matrix~\eqref{fused-R-uq}, or equivalently~\eqref{R-Rqg}. 

We first recall that $\pi^{j}$ denotes  the spin-$j$ irreducible representation of $\Uq$:
\begin{equation} \label{eq:shUq}
	\pi^{j}(E) = S_+ \ , \quad \pi^j(F)= S_- \ , \quad \pi^j(K^{\pm\h}) = q^{\pm \frac{S_3}{2}} \ ,
\end{equation}
where the matrices $S_\pm$, $S_3$ are given  in~\eqref{Bdef}. 
We now consider  evaluation of the  1st `leg' of the fused spin-$j$ Lax operators~\eqref{eq:spin-j-fusedL}
 on any representation $\pi^{j'}$:
\begin{equation} \label{straight-L}
	L^{(j',j)}(u)= ( \pi^{j'} \otimes \id ) ( \mathcal{L}^{(j)} (u)) 
	 \ ,
\end{equation}
which is now an element in $\End(\mathbb{C}^{2j'+1})\otimes \End(\mathbb{C}^{2j+1})[u,u^{-1}]$. We will refer to the first tensorand as \textit{quantum} space while the second one as \textit{auxiliary} space. We have\footnote{Combining  the fusion formulas~\eqref{eq:spin-j-fusedL} and~\cite[Eq.\,(4.24)]{LBG} evaluated according to definition~\eqref{straight-L}, and compared with Definition~\ref{def:fusR}, we obtain~\eqref{eq:L-toR} by recursion.}
		\begin{equation}\label{eq:L-toR}
			L_{[n]}^{(j_n,j)}(u) =  \left [ \displaystyle{\prod_{k=0}^{2j_n-2} \prod_{\ell=0}^{2j-1}}	c(u q^{j+j_n-k-\ell-1})  \right ]^{-1} R^{(j_n,j)}_{n,a}(u) \ ,
		\end{equation}
where the fused R-matrices $R^{(j_n,j)}(u)$ are defined in Definition~\ref{def:fusR}. Then,
using Lemma~\ref{lem:symRj1j2}, we observe that  the matrices $L^{(j_n,j)}(u)$ are symmetric for any $(j_n,j)$.

With these evaluated Lax operators $L^{(j',j)}(u)$, we can review the spin-chain representations of $\mathcal{A}^{(N)}_q$ introduced in~\cite{BK07}: 

\begin{defn}\label{def:psiN}
 Let  $N \in \mathbb{N}_+$,  and $\bj := (j_1, \ldots, j_N)$ is the $N$-tuple of arbitrary spins, and $\bar{v}:=(v_1, \ldots, v_N)$ be  scalars from $\mathbb{C}^*$. Then, the spin-chain representations 
\begin{equation} \label{eq:vartheta}
	\psi^{(N)}_{\bj,\bar{v}}\colon \mathcal{A}^{(N)}_q \rightarrow \End(\mathbb{C}^{2j_N+1} \otimes \ldots \otimes \mathbb{C}^{2j_1+1})
\end{equation}
are defined by the identification of entries of the 
$2\times 2$ matrices\footnote{The index $[k]$ in $L_{[k]}$ characterizes the `quantum space' $V_{[k]}$ on which the entries of $L$ act. With respect to the ordering $V_{[2]}\otimes V_{[1]}$, see~\eqref{eq:vartheta}, we set $(L_{[1]} L'_{[2]})_{ij} = \displaystyle{\sum_{\ell=1}^{2}} L'_{\ell j} \otimes  L_{i\ell} $,  $(L'_{[2]} L_{[1]})_{ij} = \displaystyle{\sum_{\ell=1}^{2}} L'_{i\ell} \otimes  L_{\ell j} $ and   $(L_{[1]} K L'_{[1]})_{ij} = \displaystyle{\sum_{k,\ell=1}^{2}} K_{k\ell} L_{ik} L'_{\ell j}$.}
	\begin{align} \label{eq:dressKH}
		(\psi^{(N)}_{\bj,\bar{v}} \otimes \id) \bigl(\mathcal{K}^{(\h,N)}(u)\bigr)=   L^{(j_N,\h)}_{[N]}(uv_N) \cdots  L^{(j_1,\h)}_{[1]}(uv_1)   K^{(\h)}(u)  L^{(j_1,\h)}_{[1]}(uv_1^{-1}) \cdots L^{(j_N,\h)}_{[N]}(uv_N^{-1})   \ ,\! 
	\end{align}
    which are considered as elements in $\End(\mathbb{C}^{2j_N+1} \otimes \ldots \otimes \mathbb{C}^{2j_1+1}) \otimes \End(\mathbb{C}^2)[u,u^{-1}]$, and the K-matrix $K^{(\h)}(u)$ acting on the auxiliary space is defined in~\eqref{KM-spin1/2}. The r.h.s.\ of~\eqref{eq:dressKH} is  called \textit{dressed} K-matrix.
    
\end{defn}

We now comment on why this definition or expression in~\eqref{eq:dressKH} is indeed a representation of $\mathcal{A}_q^{(N)}$, i.e.\ that the r.h.s.\ of~\eqref{eq:dressKH} satisfies the reflection equation and that, additionally, modes of its entries satisfy the linear relations~\eqref{generalized-d1} and~\eqref{generalized-d2}.
Firstly, for the detailed proof that the r.h.s satisfies the reflection equation~\eqref{eq:REAqN} we refer to~\cite{Skly88}, and we use  that the matrices~\eqref{straight-L} are symmetric, as noticed below~\eqref{eq:L-toR}. 
Secondly, the images of the alternating generators 
$$
\{  \tW^{(N)}_{-k},\; \tW^{(N)}_{k+1},\; \tG^{(N)}_{k+1},\; \tilde{\tG}^{(N)}_{k+1}\;|\;k=0,1,...,N-1\}
$$
of $\mathcal{A}^{(N)}_q$ are given explicitly by a recursion in~\cite{BK07}. 
For convenience, they are reported in Appendix~\ref{apD},   and we use the following notations:
	\beqa 
    \begin{split}
	\calW_{-k}^{(N)} &:= \psi^{(N)}_{\bj,\bar{v}}(\tW_{-k}^{(N)})\ , \quad
	&\calW_{k+1}^{(N)} &:= \psi^{(N)}_{\bj,\bar{v}}(\tW_{k+1}^{(N)})\ , \\ 
		\calG_{k+1}^{(N)} &:= \psi^{(N)}_{\bj,\bar{v}}(\tG_{k+1}^{(N)})\ , \quad
		&\tilde{\calG}_{k+1}^{(N)} &:= \psi^{(N)}_{\bj,\bar{v}}(\tilde{\tG}_{k+1}^{(N)})\ .
        \end{split}
        \label{imOqN}
	\eeqa
Here, we give expressions for the first two, namely
    \begin{align}
	 	\calW_0^{(N)}&=  \left(k_+ v_N
	 	q^{\h}S_+q^{\h S_3}
	 	+k_- v_N^{-1} 
	 	q^{-\h}S_-q^{\h S_3}\right)\otimes {\mathbb I}^{(N-1)} +  q^{S_3}\otimes \calW_{0}^{(N-1)}  \ ,\label{W0vN}\\
	 	\calW_1^{(N)}&=  \left(k_+ v_N^{-1}
	 	q^{-\h}S_+q^{-\h S_3} + k_- v_N
	 	q^{\h}S_-q^{-\h S_3}\right)\otimes {\mathbb I}^{(N-1)} +  q^{-S_3}\otimes \calW_{1}^{(N-1)}  \ ,\label{W1vN}
	 \end{align}
with the initial conditions~\eqref{W0-init-cond}.
Furthermore,   the proof of the relations~\eqref{generalized-d1} and~\eqref{generalized-d2} for $\ell=0$ is given in~\cite{BK07} under the identification
\begin{equation}
		\label{dkN}
	d_n^{(N)}=(-1)^{N+1-n} (q+q^{-1})^{n}\sum_{k_1<...<k_{N-n}=1}^{N}\alpha_{k_1}\cdot \cdot \cdot\alpha_{k_{N-n}} , \quad n=0,1,..., N\ , 
\end{equation}
where 
\begin{align*}
	\alpha_1= \frac{(v_1^2+v_1^{-2}) w_0^{(j_1)}}{(q+q^{-1})}+ \frac{\varepsilon^{(0)}_{+}\varepsilon^{(0)}_{-}(q-q^{-1})^2}{k_+k_-(q+q^{-1})}\ ,\qquad \alpha_{n}=\frac{(v_n^2+v_n^{-2})w_0^{(j_n)}}{(q+q^{-1})}\quad \mbox{for}\quad n=2,3, \ldots,N\ ,
\end{align*}
and with
\begin{equation}
	\varepsilon^{(N)}_{\pm}=w_0^{(j_N)}\varepsilon^{(N-1)}_{\mp} - (v_N^2+v_N^{-2})\varepsilon^{(N-1)}_{\pm}\ ,\qquad \varepsilon_\pm^{(0)}=\varepsilon_\pm, \qquad w_0^{(j)}=q^{2j+1} + q^{-2j-1} \label{eps-pmN}
	\ .
\end{equation}
And the proof for $\ell\geq 1$ is  shown by induction as it was done for $\varepsilon^{(0)}_\pm =\varepsilon^{(N)}_\pm = 0$  in~\cite[App.\,B]{BK05}.

\begin{rem}
\begin{enumerate} \mbox{}
\item Using the quotient map $\varphi^{(N)}$, we see that the composition $\psi^{(N)}_{\bj,\bar{v}}\circ \varphi^{(N)}$ defines  a representation of the  sole algebra ${\cal A}_q$ at any values of $N$, $\bj$ and $\bar{v}$.
\item
It will be shown  in the next section that the quantum determinant $\Gamma^{(N)}(u)$ of ${\cal A}_q^{(N)}$, recall~\eqref{GammaN}, evaluates under the spin-chain representation $\psi^{(N)}_{\bj,\bar{v}}$ to a scalar Laurent polynomial in $u$, and thus the algebra map $\psi^{(N)}_{\bj,\bar{v}}$ factorizes through the $q-$Onsager algebra, see Section \ref{sub:TT-qOA}. As a consequence, the image $\psi^{(N)}_{\bj,\bar{v}}\bigl({\cal A}_q^{(N)}\bigr)$ is generated solely by the two operators~\eqref{W0vN} and~\eqref{W1vN}.
\item 	In the context of spin chains,  each quantum space is indexed by a non-zero complex parameter $v_n$ corresponding to the inhomogeneities of the quantum spin-chains and can be also identified with the evaluation parameters of $\Loop$.
\end{enumerate}
\end{rem}

We now turn to the dressed K-matrices that are the images via $\psi^{(N)}_{\bj,\bar{v}}$ of the spin-$j$ fused K-operators $\mathcal{K}^{(j,N)}(u)$. Recall the spin-$j$ fused K-matrix~\eqref{fused-K} with the fundamental K-matrix~\eqref{KM-spin1/2}. 

\begin{prop}\label{prop:Kjspinchain} For all $j\in \h \mathbb{N}_+$ and with notations and conventions from Definition~\ref{def:psiN}, the spin-chain representation map $\psi^{(N)}_{\bj,\bar{v}}$ satisfies 
	\beqa\label{coact-h}
			(\psi^{(N)}_{\bj,\bar{v}} \otimes \id) \bigl(\mathcal{K}^{(j,N)}(u)\bigr)    
			= L^{(j_N,j)}_{[N]}(uv_N)  \cdots  L^{(j_1,j)}_{[1]}(uv_1)   K^{(j)}(u)  L^{(j_1, j)}_{[1]}(uv_1^{-1}) \cdots L^{(j_N,j)}_{[N]}(uv_N^{-1})
	\eeqa
	with  identification of the parameters $d_k^{(N)}$ and $\varepsilon_{\pm}^{(N)}$ of $\psi^{(N)}_{\bj,\bar{v}}$ as in~\eqref{dkN} and~\eqref{eps-pmN},
and the K-matrix $K^{(j)}(u)$ acting on the auxiliary space is defined in~\eqref{fused-K}.
\end{prop}
We note that the identification in~\eqref{coact-h}  is the reason of all the normalizations introduced before.
\begin{proof} 
	The proof is done by induction. The case $j=\h$ holds by  Definition \eqref{eq:dressKH}. 
	Assume~\eqref{coact-h} holds for a fixed $j$. We now show the case $j+\h$. For convenience, consider $N=2$. From the formula of $\mathcal{K}^{(j,N)}(u)$ in~\eqref{fused-K-opN}, we have:
	\begin{align} \nonumber
		\mathcal{K}^{(j+\h,2)}(u) &= \mathcal{F}^{(j+\h)}_\fu \mathcal{K}_1^{(\h,2)}(u q^{-j}) R_{12}^{(\h,j)}(u^2 q^{-j+\h}) \mathcal{K}_2^{(j,2)}(u q^\h)  \mathcal{E}_\fu^{(j+\h)} \mathcal{H}^{(j+\h)}_\fu \lbrack \mathcal{H}^{(j+\h)}_\fu\rbrack ^{-1} \\ \nonumber
		&= \mathcal{F}_\fu^{(j+\h)}  \mathcal{K}_1^{(\h,2)}(u q^{-j}) R_{12}^{(\h,j)}(u^2 q^{-j+\h}) \mathcal{K}_2^{(j,2)}(u q^\h) R_{12}^{(\h,j)}(q^{j+\h})  \mathcal{E}_\fu^{(j+\h)}  \lbrack \mathcal{H}^{(j+\h)}_\fu\rbrack ^{-1} \\
		\label{proofAQAQN}
		& = \mathcal{F}^{(j+\h)}_\fu R_{12}^{(\h,j)}(q^{j+\h}) \mathcal{K}_2^{(j,2)}(u q^\h)  R_{12}^{(\h,j)}(u^2 q^{-j+\h})  \mathcal{K}_1^{(\h,2)}(u q^{-j})  \mathcal{E}_\fu^{(j+\h)} \lbrack \mathcal{H}^{(j+\h)}_\fu\rbrack ^{-1} \ ,
	\end{align}
	where we used~\eqref{usefulEFH} and the reflection equation.
	Below, we underline the steps of calculation and we use the shorthand notation:\footnote{Indices $k,\ell$ in $X_{k[\ell]}$ stand for the auxiliary and quantum space, respectively.} 
	\begin{align}
		& R_{k \ell} = R_{k \ell}^{(j_k,j_\ell)}(u_k/u_\ell) \ ,\qquad \bar{R}_{k\ell}=R_{k\ell}^{(j_k,j_\ell)}(u_\ell/ u_k)\ , \qquad \hat{R}_{k\ell}=R_{k\ell}^{(j_k,j_\ell)}(u_k u_\ell)  \ , \nonumber\\
		& K_\ell=K_\ell^{(j_\ell)}(u_\ell)\ ,\qquad L_{k[\ell]} =\left( L_{k}^{(j_k,j)}(u_k/v_\ell) \right)_{[ \ell ]}, \quad \hat{L}_{k[\ell]} =\left( L_{k}^{(j_k,j)}(u_k v_\ell) \right)_{[ \ell ]} \ .\nonumber
	\end{align} 
	Recall that two operators that act on different auxiliary and quantum spaces commute, i.e.\
	\begin{equation} \label{eq:commutLK}
		\left [L_{k[\ell]}, L_{m[n]} \right ] = 0 \ , \qquad \left [ L_{k[\ell]} , K_m \right ] = 0 \ , \qquad  \text{for $k \neq m$ and $\ell \neq n$ \ .}
	\end{equation}
	Recall also that $L^{(j)}(u)$ and $R^{(\h,j)}(u)$ are invertible, their inverses are proportional to $L^{(j)}(u^{-1})$ and $R^{(\h,j)}(u^{-1})$, respectively~\cite[Sec.\,4.1.6]{LBG}. Then, the reflection equation and the Yang-Baxter algebra yield:
	\begin{align} \label{eq:not-RK}
		K_1 \hat{R}_{12} K_2 \bar{R}_{12} = \bar{R}_{12} K_2 \hat{R}_{12} K_1 \ , \quad   L_{2[2]} \hat{R}_{12} \hat{L}_{1[\ell]} =  \hat{L}_{1[\ell]} \hat{R}_{12} L_{2[\ell]} \ , \quad \bar{R}_{12} \hat{L}_{2 [\ell]} \hat{L}_{1[\ell]} = \hat{L}_{1[\ell]} \hat{L}_{2 [\ell]} \bar{R}_{12}\ .
	\end{align}
	Fix $j_1=\h$, $j_2 =j$ and $u_1 =u q^{-j}$, $u_2=u q^\h$. Then from~\eqref{proofAQAQN}, assuming~\eqref{coact-h} holds for a fixed $j$, one has:
	\begin{align*} 
		\mathcal{K}^{(j+\h,2)}(u)&= \mathcal{F}_\fu\bar{R}_{12} \hat{L}_{2[2]} \hat{L}_{2[1]} K_2 L_{2[1]} \underline{ L_{2[2]} \hat{R}_{12} \hat{L}_{1[2]}}  \hat{L}_{1[1]} K_1 L_{1[1]} L_{1[2]} \mathcal{E}_\fu \mathcal{H}_\fu^{-1} \\
		& \overset{\eqref{eq:not-RK}}{=}  \mathcal{F}_\fu\bar{R}_{12} \hat{L}_{2[2]} \underline{ \hat{L}_{2[1]} K_2 L_{2[1]}  \hat{L}_{1[2]}}  \hat{R}_{12} \underline{L_{2[2]}  \hat{L}_{1[1]} K_1 L_{1[1]}} L_{1[2]} \mathcal{E}_\fu \mathcal{H}_\fu^{-1}  \\
		&\overset{\eqref{eq:commutLK}}{=}  \mathcal{F}_\fu\bar{R}_{12} \hat{L}_{2[2]}  \hat{L}_{1[2]} \hat{L}_{2[1]} K_2  \underline{L_{2[1]}   \hat{R}_{12}  \hat{L}_{1[1]}} K_1 L_{1[1]} L_{2[2]} L_{1[2]} \mathcal{E}_\fu \mathcal{H}_\fu^{-1} \\
		&\overset{\eqref{eq:commutLK}}{=}  \mathcal{F}_\fu\bar{R}_{12} \hat{L}_{2[2]}  \hat{L}_{1[2]} \hat{L}_{2[1]} \underline{ K_2  \hat{L}_{1[1]}}   \hat{R}_{12}  \underline{ L_{2[1]} K_1} L_{1[1]} L_{2[2]} L_{1[2]} \mathcal{E}_\fu \mathcal{H}_\fu^{-1} \\
		&\overset{\eqref{eq:not-RK}}{=}  \mathcal{F}_\fu\underline{\bar{R}_{12} \hat{L}_{2[2]}  \hat{L}_{1[2]} \hat{L}_{2[1]}  \hat{L}_{1[1]} } K_2    \hat{R}_{12}   K_1 L_{2[1]} L_{1[1]} L_{2[2]} L_{1[2]} \mathcal{E}_\fu \mathcal{H}_\fu^{-1} \ . 
	\end{align*}
	Now, using the Yang-Baxter algebra (\ref{RLL}), the reflection equation and~\eqref{usefulEFH} we get:
	\begin{align*}
		\qquad	\mathcal{K}^{(j+\h,2)}(u)&= \mathcal{F}_\fu \hat{L}_{1[2]} \hat{L}_{2[2]} \mathcal{E}_\fu\mathcal{F}_\fu \hat{L}_{1[1]}  \hat{L}_{2[1]} \mathcal{E}_\fu\mathcal{F}_\fu\underline{\bar{R}_{12} K_2    \hat{R}_{12}  K_1}  L_{2[1]} L_{1[1]} L_{2[2]} L_{1[2]} \mathcal{E}_\fu \mathcal{H}_\fu^{-1}  \\
		&\overset{\eqref{eq:not-RK}}{=}  L^{(j+\h)}_{[2]}(uv_2) L^{(j+\h)}_{[1]}(uv_1) \mathcal{F}_\fu K_1 \hat{R}_{12}  K_2  \mathcal{E}_\fu\mathcal{F}_\fu\underline{\bar{R}_{12}    L_{2[1]} L_{1[1]} L_{2[2]} L_{1[2]} } \mathcal{E}_\fu \mathcal{H}_\fu^{-1} \\
		&\overset{\eqref{eq:not-RK}}{=} L^{(j+\h)}_{[2]}(uv_2) L^{(j+\h)}_{[1]}(uv_1) K^{(j+\h)}(u)  \mathcal{F}_\fu L_{1[1]} L_{2[1]} \mathcal{E}_\fu\mathcal{F}_\fu L_{1[2]} L_{2[2]} \underline{\bar{R}_{12} \mathcal{E}_\fu \mathcal{H}_\fu^{-1} } \\
		&\overset{\eqref{usefulEFH}}{=} L^{(j+\h)}_{[2]}(uv_2) L^{(j+\h)}_{[1]}(uv_1) K^{(j+\h)}(u) L^{(j+\h)}_{[1]}(uv_1^{-1}) L^{(j+\h)}_{[2]}(uv_2^{-1}) \ .
	\end{align*}
		The proof extends to general $N$ in a straightforward manner.
	\end{proof}

	\section{Conserved quantities for integrable open spin chains}\label{sec5}
    In this section, we construct  normalized transfer matrices that are suitable for building Hamiltonians of integrable spin chains of length $N$ and any spin values for the auxiliary and quantum spaces. As a result, together with the TT-relations studied in the previous sections, higher Hamiltonians ${\cal H}^{(n)}$ of spin-$j$ are expressed in terms of images on spin-chain representations of mutually commuting elements $\tI_{2k+1}^{(N)}$, with $k=0,\ldots,N-1$, given by~\eqref{IkOqN}.
    This gives the main result of this section -- an explicit and simple algorithm of constructing the hierarchy of higher Hamiltonians ${\cal H}^{(n)}$ in terms of spin matrices for any spin values.
    
    In a greater detail, we begin in Section~\ref{sec:TT-tr-mat} with recalling Sklyanin's construction of spin-chain transfer matrices and relating them to the spin-chain representation $\psi^{(N)}_{\bj,\bar{v}}$ from Definition~\ref{def:psiN} of the spin-$j$ generating functions $\bt^{(j,N)}(u)$ of ${\cal A}^{(N)}_q$ introduced in~\eqref{tgintN}. As a result we get TT-relations for the spin-chain transfer matrices in Proposition~\ref{prop:TTt}. Then, in Section~\ref{sec:renorm-tr-mat} we show   that properly normalized transfer matrices
    have  a well behaved specialization at $u=1$ which is necessary for the higher Hamiltonians construction as logarithmic derivatives of the transfer matrix at $u=1$.
    Explicit examples for the open XXZ spin-$\h$ and spin-1 chains are worked out in details. In Section~\ref{sec:t-expansion}, we provide expansion of (normalized) spin-$j$  transfer matrices in terms of the spin-chain representations of the Laurent polynomials $\mathsf{I}^{(N)}(u)$ from~\eqref{eq:def-IN}, see Proposition~\ref{prop:tIN} and Corollary~\ref{cor:nt}. We give explicit expressions of first higher Hamiltonians in these terms for spin-$\h$, and describe our algorithm of expressing every ${\cal H}^{(n)}$ for arbitrary spins as a polynomial in operators $\psi^{(N)}_{\bj,\bar{v}}(\tI_{2k+1}^{(N)})$, and thus eventually in terms of spin matrices. Finally, in Section~\ref{sec:tr-mat-Ons} we explain how to express both the transfer matrices and its Hamiltonians in terms of the spin-chain representations of the $q-$Onsager generators $\calW_0^{(N)}$, $\calW_1^{(N)}$ given by~\eqref{W0vN} and~\eqref{W1vN}.

    \subsection{TT-relations for spin-chain transfer matrices}\label{sec:TT-tr-mat}
	The explicit construction of conserved quantities associated with various examples of quantum integrable open spin chains of length $N$ with generic integrable boundary conditions, e.g.\ higher XXZ spin chains~\cite{FNR07} and alternating spin chains~\cite{CYSW14}, is now revisited in light of  results of the previous sections.  For these models, it is well-known that  conserved quantities are derived from a transfer matrix~\cite{Skly88}:
	\begin{equation}  \label{tjs}
		{\boldsymbol  t}^{j,\bj}(u)= \normalfont{\text{tr}}_{{\mathbb C}^{2j+1}} \bigl( K_{a}^{+(j)}(u) T_{a,N}^{j, \bj}(u) K_a^{(j)}(u) \hat{T}_{a,N}^{j,\bj}(u) \bigr) \ ,
	\end{equation} 
with the fused K-matrix $K^{(j)}(u)$ defined in~\eqref{fused-K} and its dual one $K^{+(j)}(u)$ from~\eqref{def:kp}, and here $j$ denotes the spin of the auxiliary space indicated by `$a$' and with the trace taken over it, while $\boldsymbol{\bj} := (j_1, \ldots, j_N)$ is the $N$-tuple of  spins at the quantum spaces; we also introduce 
	\begin{align}\label{eq:TAN}
    \begin{split}
		T_{a,N}^{j,\bj}(u)&= 	R_{Na}^{(j_N,j )}(uv_N) \cdots  R_{1a}^{(j_1,j)}(uv_1) \ ,  \\
        \hat{T}_{a,N}^{j,\bj}(u)&= R_{1a}^{(j_1,j)}(uv_1^{-1}) \cdots R_{Na}^{(j_N,j)}(uv_N^{-1}) \ ,
        \end{split}
	\end{align}
	with the R-matrices $R^{(j_1,j_2)}(u)$ given in~\eqref{v2Rj1j2}, and $v_i\in\mathbb{C}^*$ are the inhomogeneities. Typically,  {\it local} conserved quantities  are obtained for $j=j_i$ and $v_i=1$ for all $i=1,...,N$, as will be discussed in next subsections. Since by construction the entries of $R^{(j_1,j_2)}(u)$, $K^{(j)}(u)$ and  $f^{(j)}(u)K^{+(j)}(u)$    with $f^{(j)}(u)$ given in~\eqref{eq:fj} are Laurent polynomials in $u$, see Definition~\ref{def:fusR}, eqs.~\eqref{fused-K} with~\eqref{KM-spin1/2} and~\eqref{def:kp}, it follows that
\beqa
f^{(j)}(u){\boldsymbol  t}^{j,\bj}(u) \in   \Bigl(\bigotimes_{n=1}^{N} \End({\mathbb C}^{2j_n+1})\Bigr)[u,u^{-1}]\ .\label{eq:polyt}
\eeqa

    Now, we relate the  spin-$j$ generating functions~\eqref{tgintN} of commutative subalgebras in $\mathcal{A}_q^{(N)}$ to the transfer matrices from~\eqref{tjs}.

\begin{prop} \label{prop:TNtj}
 Let  $N \in \mathbb{N}_+$,  and $\bj := (j_1, \ldots, j_N)$ is the $N$-tuple of arbitrary spins, and $v_i\in \mathbb{C}^*$, for $1\leq i\leq N$.
	The spin-chain representation of the generating function $\bt^{(j,N)}(u)$ from~\eqref{tgintN}  is proportional to the transfer matrix~\eqref{tjs}:
\begin{equation} 
	\psi^{(N)}_{\bj,\bar{v}}\bigl(\bt^{(j,N)}(u)\bigr) =  \left [\prod_{n=1}^N   \displaystyle{\prod_{k=0}^{2j_n-2}  \prod_{\ell=0}^{2j-1} }  c(u q^{j+j_n-k-\ell-1} v_n) c(u q^{j+j_n-k-\ell-1} v_n^{-1})           \right ]^{-1} {\boldsymbol  t}^{j,\bj}(u)\ ,\label{evalTN}
\end{equation}
where the spin-chain representation $\psi^{(N)}_{\bj,\bar{v}}$ of $\mathcal{A}_q^{(N)}$ is given in Definition~\ref{def:psiN}.
\end{prop}
\begin{proof}
	It is straightforward starting from the definition of $\bt^{(j,N)}(u)$ in~\eqref{tgintN}, and using~\eqref{coact-h} with~\eqref{eq:L-toR}.
\end{proof}

We are now ready to formulate the main result of this section -- the TT-relations of the transfer matrices~\eqref{tjs} for arbitrary spins.
	\begin{prop}\label{prop:TTt}The transfer matrices ${\boldsymbol  t}^{j,\bj}(u)$ from~\eqref{tjs} satisfy the TT-relations for all $j\in\h\mathbb{N}_+$
		\begin{align}\label{normTTN}
			{\boldsymbol  t}^{j,\bj}(u) &= {\boldsymbol  t}^{j-\h,\bj}(u q^{-\h}) {\boldsymbol  t}^{\h,\bj}(u q^{j-\h})  \\ \nonumber
			&+
			\left [	
			\displaystyle{\prod_{n=1}^N} \beta^{(j_n)}(u q^{j-\h} v_n) \beta^{(j_n)}(u q^{j-\h} v_n^{-1})
			\right ]  \frac{	\Gamma_{-} (u q^{j-\tha}) \Gamma_{+} (u q^{j-\tha})       }{ c(u^2 q^{2j}) c(u^2 q^{2j-2})} {\boldsymbol  t}^{j-1,\bj}(u q^{-1}) \ ,
		\end{align}
        with ${\boldsymbol  t}^{0,\bj}(u) =  1$ and
		 the quantum determinants $\Gamma_{\pm}(u)$ introduced in~\eqref{gammaKP} and~\eqref{gammaKM}, respectively, and $\beta^{(j)}(u)$ is given by~\eqref{beta}.
	\end{prop}
	\begin{proof} 
    	Recall the universal TT-relations satisfied by $\bt^{(j,N)}(u)$ in~\eqref{TT-relTN}, and that spin-chain representations for $\bt^{(j,N)}(u)$ lead to transfer matrices in Proposition~\ref{prop:TNtj}. The TT-relations~\eqref{normTTN} are obtained as the images of $\psi^{(N)}_{\bj,\bar{v}}$ on~\eqref{TT-relTN}.
	We are thus  left to compute the image of the quantum determinant $\Gamma^{(N)}(u)$ from~\eqref{GammaN} which is
    $$
    \psi^{(N)}_{\bj,\bar{v}}\bigl(\Gamma^{(N)}(u)\bigr) = \normalfont{\text{tr}}_{12}\Bigl({\cal P}^{-}_{12} \psi^{(N)}_{\bj,\bar{v}}\bigl({\cal K}_1^{(\h,N)}(u)\bigr)  R_{12}^{(\frac{1}{2},\frac{1}{2})}(qu^2)  \psi^{(N)}_{\bj,\bar{v}}\bigl({\cal K}_2^{(\h,N)}(u q)\bigr) \Bigr) \ ,
    $$
    and together with~\eqref{eq:dressKH} and~\eqref{eq:L-toR}, the calculation reduces to the expression of the form 
    $$
    \normalfont{\text{tr}}_{12}\big({\cal P}^{-}_{12} T_{1,N}^{\h,\bj}(u) K_1^{(\h)}(u) \hat{T}_{1,N}^{\h,\bj}(u)  \ R_{12}^{(\frac{1}{2},\frac{1}{2})}(qu^2) T_{2,N}^{\h,\bj}(uq) K_2^{(\h)}(uq) \hat{T}_{2,N}^{\h,\bj}(uq)  \big) \ ,$$
    where we also used the notations in~\eqref{eq:TAN}.
    Then, according to~\cite[Prop.\,6]{Skly88}, it follows that this expression factorizes -- up to an overall factor -- to a product of quantum determinants of K- and $2N$ R-matrices:
		\begin{align}
		&\normalfont{\text{tr}}_{12}\big({\cal P}^{-}_{12} T_{1,N}^{\h,\bj}(u) K_1^{(\h)}(u) \hat{T}_{1,N}^{\h,\bj}(u)  \ R_{12}^{(\frac{1}{2},\frac{1}{2})}(qu^2) T_{2,N}^{\h,\bj}(uq) K_2^{(\h)}(uq) \hat{T}_{2,N}^{\h,\bj}(uq)  \big) \nonumber\\ 
			&= \delta( T_{N}^{\h,\bj}(u)) \delta (\hat{T}_{N}^{\h,\bj}(u) ) \Gamma_{-} (u) \nonumber\\
            &=  \prod_{n=1}^N  \beta^{(j_n)}(u q v_n) \times  \prod_{n=1}^N \beta^{(j_n)}(u q v_n^{-1}) \times  \Gamma_{-} (u) \nonumber
		\end{align}
		where $\beta^{(j)}(u)$ is given in ~\eqref{beta} and we use   the notation  $\delta(T(u)) = \text{tr}_{12}\big({\cal P}^{-}_{12}T_1(u)T_2(uq)\big)$. 
		We then get  the final expression for the image of $\Gamma^{(N)}$: 
        \begin{align}\label{evalgamN}
			\psi^{(N)}_{\bj,\bar{v}} (	\Gamma^{(N)}(u)) =\left [  \displaystyle{ \prod_{n=1}^N  }
				c(u q^{j_n+\frac{3}{2}} v_n)c(u q^{j_n+\frac{3}{2}} v_n^{-1})c(u q^{-j_n+\h} v_n)
				c(u q^{-j_n+\h} v_n^{-1}) \right ]  \ \Gamma_{-} (u) \ .
		\end{align}
        
        After simplifications, using \eqref{evalTN}, \eqref{evalgamN}, one finds that the image of the TT-relations \eqref{TT-relTN} under $\psi^{(N)}_{\bj,\bar{v}}$ is indeed~\eqref{normTTN}. 
	\end{proof}

\begin{rem}
We notice that even if the whole analysis of the TT-relations based on $\cal A_q$ was made with the assumption $k_\pm\ne 0$, the above expression~\eqref{normTTN} is well defined for all values of the boundary parameters including the `diagonal' case of $k_\pm=0$.
\end{rem}
    
    \begin{example}\label{ex:FNR}
         Examples of TT-relations previously conjectured in the literature are recovered by specializations of (\ref{normTTN}). More precisely, in the next two works on TT-relations for open spin chains the corresponding K-matrices of arbitrary spins are those from the fusion approach of~\cite{KRS81,MN91}, and thus they are essentially the same K-matrices used in the construction of our transfer matrices~\eqref{tjs}, recall  Remark~\ref{rem:K-spin-gen}.
    \begin{enumerate}
    \item 
         The TT-relations conjectured for the spin-$j$ XXZ open spin chain of length $N$ with generic boundary conditions~\cite[eq.\,(2.16)]{FNR07} are obtained by setting $j_n=j$ and $v_n=1$ for all $n$ in~\eqref{normTTN} and the suitable identifications from our conventions to the ones in \cite{FNR07}:
         \beqa
u \rightarrow e^{u}, \quad q \rightarrow e^{\eta}, \quad && 
 \varepsilon_\pm \rightarrow  \sinh(\alpha_- \pm \beta_-) , \quad \bar{\varepsilon}_\pm \rightarrow - \sinh(\alpha_+ \mp\beta_+) \ ,\nonumber\\
&& k_\pm \rightarrow e^{\pm\theta_-}
 \sinh(\eta) , \qquad \bar{k}_\pm \rightarrow - e^{\pm\theta_+} \sinh(\eta) \nonumber \ .
\eeqa
    \item 
        The TT-relations conjectured for the alternating spin chain of length $N$ even with generic boundary conditions and generic homogeneity parameters~\cite[eq.\,(3.1)]{CYSW14} are obtained by setting $j_{2p}=s$, $j_{2p+1}=s'$ 
     for all $p$ in~\eqref{normTTN} and suitable identifications from our conventions to the ones in~\cite{CYSW14}.
    \end{enumerate}
    \end{example}

\subsection{Normalized transfer matrix and higher Hamiltonians}\label{sec:renorm-tr-mat}
 Similarly to~\cite{Lu96}, local Hamiltonians can be constructed from a suitably {\it normalized} transfer matrix \cite[Rem. 4]{Skly88}, taking  logarithmic derivatives at $u=1$, see e.g.\ \cite[eq.\,(5.6)]{FNR07}.

\begin{defn}\label{def:tr-mat-norm}
We define the {\it normalized} transfer matrix as
	\begin{equation}  \label{tildetjs}
		\tilde {\boldsymbol  t}^{j,\bj}(u)=  \normalfont{\text{tr}}_{{\mathbb C}^{2j+1}} \bigl(\tilde K_{a}^{+(j)}(u) \tilde T_{a,N}^{j, \bj}(u) \tilde K_a^{(j)}(u) \hat{\tilde T}_{a,N}^{j,\bj}(u) \bigr) \ ,
	\end{equation} 
where the normalized K-matrix $\tilde{K}^{(j)}(u)$ is introduced in~\eqref{renormK}, the normalized dual spin-$j$ K-matrix is  
\beqa
\tilde K^{+(j)}(u)= \left(\tilde K^{(j)}(u^{-1}q^{-1})\right)^t\bigg\rvert_{\varepsilon_\pm \rightarrow  \ov{\varepsilon}_\mp,  k_\pm \rightarrow - \ov{k}_\mp } \ , 
\label{tildeKplus}
\eeqa
and $\tilde T_{a,N}^{j,\bj}(u),\hat{\tilde T}_{a,N}^{j,\bj}(u)$ denote the product of normalized R-matrices  $\tilde{R}^{(j_n,j)}(uv_n)$ from~\eqref{renormR}:
	\begin{align}\label{eq:TtildeAN}
    \begin{split}
		\tilde T_{a,N}^{j,\bj}(u)&= 	\tilde R_{Na}^{(j_N,j )}(uv_N) \cdots  \tilde R_{1a}^{(j_1,j)}(uv_1) \ ,  \\
        \hat{\tilde T}_{a,N}^{j,\bj}(u)&= \tilde R_{1a}^{(j_1,j)}(uv_1^{-1}) \cdots \tilde R_{Na}^{(j_N,j)}(uv_N^{-1}) \ .
        \end{split}
	\end{align}
In particular, the relation to ${\boldsymbol  t}^{j,\bj}(u)$ from~\eqref{tjs} is as follows:
    \beqa
    \tilde {\boldsymbol  t}^{j,\bj}(u) = g^{j,\bj}(u) {\boldsymbol  t}^{j,\bj}(u)\label{renormtrans}
    \eeqa
    where  
    \beqa
    g^{j,\bj}(u) &=& f^{(j)}(u)\left(\prod_ {\ell = 0}^{2 j - 2} 
    c (u^{-2}q^{-1- \ell}) 
    c (u^2q^{1 - \ell})\right)^{-1}  \label{eq:renormgj} \\
    && \times \prod_{n=1}^N\left( \prod_ {k = 0}^{2 j_n - 1} \prod_ {\ell = 0}^{2 j -
   1} 
    c (uq^{j_n+j - k - \ell}v_n) c (uq^{j_n+j - k - \ell}v_n^{-1}) \right)^{-1}  \ ,
    \nonumber
    \eeqa
    with $f^{(j)}(u)$ introduced in~\eqref{eq:fj}.
\end{defn}

For the (higher) Hamiltonians construction in spin chains,
we need first to assure that specialization of the transfer matrix at $u=1$ is well-defined and takes  non-zero value.
This is the reason of the above normalization~\eqref{renormtrans} in Definition~\ref{def:tr-mat-norm}. We now show  that the normalized transfer matrix $\tilde {\boldsymbol  t}^{j,\bj}(u)$ given by~\eqref{tildetjs}
has indeed a well behaved specialization at $u=1$.

  \begin{lem}\label{lem:t-non-zero}
  For any spin $j$ and spins $j_n=j$ and inhomogeneities $v_n=1$ for all values of $1\leq n \leq N$, $\tilde {\boldsymbol  t}^{j,\bj}(1)$ is well-defined and has the limit
  \begin{equation}\label{eq:lim-t}
     \lim_{\varepsilon_+, \bar{\varepsilon}_-\rightarrow \infty} \frac{\tilde {\boldsymbol  t}^{j,\bj}(1)}{(\varepsilon_+\bar{\varepsilon}_-)^{2j}}=[2j+1]_q\ .
    \end{equation}
   In particular, for generic parameters $\varepsilon_\pm, k_\pm, \bar{\varepsilon}_\pm, {\bar k}_\pm\in \mathbb{C}$, we have 
  $$
  \tilde {\boldsymbol  t}^{j,\bj}(1)\neq 0\ .
  $$
    \end{lem}
   \begin{proof} First, using that the normalized R-matrices $\tilde{R}^{(j,j)}(u)$ appearing in the expression~\eqref{tildetjs} with~\eqref{eq:TtildeAN} specialized at $u=1$ are just permutation operators, recall~\eqref{eq:R-P}, we find 
   \beqa
   \tilde {\boldsymbol  t}^{j,\boldsymbol{\bj}}(1)=  \normalfont{\text{tr}}_{{\mathbb C}^{2j+1}} \bigl(\tilde K^{+(j)}(1)\bigr) \tilde K^{(j)}(1)\ . \label{eq:normt1}
   \eeqa
   Recalling that by Lemma~\ref{lem:K-poly} the entries of $\tilde{K}^{(j)}(u)$, and thus those of $\tilde{K}^{+(j)}(u)$, are Laurent polynomials in $u$, we see that $\tilde{\boldsymbol  t}^{j,\boldsymbol{\bj}}(1)$ is well-defined.  Using~\eqref{tildeKplus} with \eqref{eq:limKj} we see that the trace $\mathrm{tr}_{{\mathbb C}^{2j+1}} \bigl(\tilde K^{+(j)}(1)\bigr)$ is a polynomial in $\bar{\varepsilon}_-$ with the coefficient at the leading term $\bar{\varepsilon}_-^{2j}$ given by $[2j+1]_q$. Furthermore, by Corollary~\ref{cor:lim-K11-norm}  $\tilde K^{(j)}(1)$ is a non-zero  matrix, and using~\eqref{eq:K-delta} and~\eqref{eq:limKj} we see that
   $\lim_{\varepsilon_+\rightarrow \infty}\frac{\tilde K^{(j)}(1)}{\varepsilon_+^{2j}}$ 
   is the identity matrix. All together, it  finally implies~\eqref{eq:lim-t}.
   \end{proof}

   We are now ready to discuss conserved quantities generated by the {\it homogeneous} transfer matrix, i.e.\ 	for the choice 
  \beqa\label{eq:j-homo}
  j_n=j \quad \mbox{and} \quad v_n=1\ , \qquad  \ n=1,2,\ldots,N\ , 
  \eeqa 
  in particular local ones that are called higher Hamiltonians and defined as logarithmic derivatives of $\tilde {\boldsymbol  t}^{j,\bj}(u)$ at $u=1$:
    \begin{equation}
		\cal H^{(n)} =  \frac{d^n}{du^n} \ln\bigl(\tilde {\boldsymbol  t}^{j,\bj}(u)\bigr) \big\rvert_{u=1} \ ,\quad n\geq 1\ . \label{HnfromT}
	\end{equation}
    In particular, 
    \beqa
		\cal H^{(1)} &=&  \tilde {\boldsymbol  t}^{j,\bj}(1)^{-1} \frac{d}{du} \tilde {\boldsymbol  t}^{j,\bj}(u)|_{u=1}  \ , \label{H1fromT} \\
       \cal H^{(2)} &=&\tilde {\boldsymbol  t}^{j,\bj}(1)^{-1} \frac{d^2}{du^2}\tilde {\boldsymbol  t}^{j,\bj}(u)|_{u=1} - \tilde {\boldsymbol  t}^{j,\bj}(1)^{-2}\left(\frac{d}{du}\tilde {\boldsymbol  t}^{j,\bj}(u)|_{u=1}\right)^2 \ .\label{H2fromT}
\eeqa
These are well defined due to Lemma~\ref{lem:t-non-zero}.
The first member of the sequence~\eqref{H1fromT} -- the quantity $\cal H^{(1)}$ -- gives (up to an additive scalar term and an overall factor), the defining Hamiltonian of the quantum spin chain \cite{Skly88,CLSW02}. It
	depends on the  boundary parameters $k_\pm$, $\varepsilon_{\pm}$ expressing boundary conditions of the open spin-chain at site $1$ and on the boundary parameters $\bar{k}_\pm$ and $\bar{\varepsilon}_{\pm}$ at site $N$, as can be explicitly seen in the next two examples for spins $\h$ and $1$.

\begin{example}\label{ex:H12} The Hamiltonian of the XXZ spin-$\h$ chain with generic boundary conditions~\cite{CLSW02} is obtained from ${\cal H}^{(1)}$ defined by~\eqref{H1fromT} with normalized transfer matrix~\eqref{tildetjs}  for
 $j=j_n=\h$ for all $n$:
 \beqa
\frac{c(q)}{2}{\cal H}^{(1)}= \cal H_{XXZ} -
A_0 \ ,\label{H1fromYB}
\eeqa
 where $A_0$ is a scalar and with the Hamiltonian
\beqa
\cal H_{XXZ} &=& \displaystyle{ \sum_{k=1}^{N-1}} \Big ( \sigma_{k+1}^x \sigma_{k}^x + \sigma_{k+1}^y\sigma_{k}^y + \frac{q+q^{-1}}{2}\sigma_{k+1}^z \sigma_{k}^z \Big ) \label{HXXZhalf}\\
&+&
\frac{2}{\varepsilon_+ + \varepsilon_-} \Big ( \frac{q-q^{-1}}{4}\big(\varepsilon_+-\varepsilon_-\big) \sigma^z_ 1 + k_+ \sigma^+_1 + k_-\sigma^-_1 \Big ) \nonumber
		\\ &+& \frac{2}{\overline{\varepsilon}_+ + \overline{\varepsilon}_-} \Big (  \frac{q-q^{-1}}{4}\big ( \overline{\varepsilon}_+ - \overline{\varepsilon}_- \big )  \sigma^z_N + \overline{k}_+\sigma^+_N +  \overline{k}_-\sigma^-_N \Big ) \ .\nonumber
\eeqa
Details of calculations are reported in Appendix~\ref{ApE}.
\end{example}
\begin{example}\label{ex:H1}  The Hamiltonian of the XXZ spin-$1$ chain with generic boundary conditions~\cite{FNR07} is obtained from ${\cal H}^{(1)}$ defined by~\eqref{H1fromT} with normalized transfer matrix~\eqref{tildetjs}  for
 $j=j_n=1$ for all $n$. Its explicit expression is derived along the same line as for the spin-$\h$ case, following~\eqref{H1}-\eqref{eq:localham} with the replacement ${\tilde {\boldsymbol  t}}^{(\h,\h)}(u) \rightarrow {\tilde {\boldsymbol  t}}^{(1,1)}(u)$, ${\tilde R}^{(\h,\h)}(u)\rightarrow {\tilde R}^{(1,1)}(u)$, ${\tilde K}^{(\h)}(u)\rightarrow {\tilde K}^{(1)}(u)$.
 Similarly to the spin-$\h$ case, ${\cal H}^{(1)}$ reads as linear combination of ${\cal H}_{XXZ}^{\textsf{\ spin 1}}$ and a scalar term where 
 \begin{align} \label{Ham-params1} \mathcal{H}_{XXZ}^{\textsf{\ spin 1}}&
= 
\displaystyle{\sum_{n=1}^{N-1} \Bigg(\vec{s}_{n+1} \cdot \vec{s}_{n}-(\vec{s}_{n+1}\cdot \vec{s}_{n})^2   - \frac{(q^{\h}-q^{-\h})^2}{2}  \left \{s^z_{n+1}s^z_{n},  s^+_{n+1}s^-_{n}+s^-_{n+1}s^+_{n} \right \}   } \\
&\qquad \quad +\frac{c(q)^2}{2} \left ( s^z_{n+1}s^z_{n} -(s^z_{n+1})^2(s^z_{n})^2 + (s^z_{n})^2 + (s^z_{n+1})^2  \right ) \Bigg) \nonumber\\
&\qquad \quad + {\cal H}_{1}^b(\varepsilon_+,\varepsilon_-,k_+,k_-) + {\cal H}_{N}^b(\bar\varepsilon_+,\bar\varepsilon_-,{\bar k}_+,{\bar k}_-) \nonumber
\end{align}
with
\begin{align}
&\hspace{-0.3cm} {\cal H}_{a}^b(\varepsilon_+,\varepsilon_-,k_+,k_-) = \frac{1}{\kp \km-\epsp \epsm (q+q^{-1}) -( \epsp^2 + \epsm^2) }\nonumber \\ \nonumber
&\hspace{-0.3cm} \times \frac{c(q^2)}{2} \Bigg (  \left(\epsp \epsm (q^{-1}-q) + \kp\km \frac{ q+q^{-1}}{q^{-1}-q} \right)(s^z_a)^2  + (\epsm^2 - \epsp^2) s^z_a + \frac{\kp^2 (s^+_a)^2 + \km^2 (s^-_a)^2}{q-q^{-1}} \\ \nonumber
&\hspace{-0.3cm}+\frac{\sqrt{2(q+q^{-1})} }{ {q-q^{-1}}} \Big (  \epsp \!\left (   k_+ [s_a^+,s_a^z]_{q^\h} \!+ \!k_- [s_a^z,s_a^-]_{q^{\h}}\!
\right )\! +\epsm  \! \left (   k_+ [s_a^+,s_a^z]_{q^{-\h}} + k_- [s_a^z,s_a^-]_{q^{-\h}}  \right ) \!\! \Big) \!\!\Bigg) \ .
\end{align}
Here we introduced the anti-commutator $\{A,B\}= AB+BA$ and  the spin-$1$  angular momentum $\vec{s}$  with components $s^x,s^y,s^z$ that are the spin-1 matrices 
\begin{align} \label{spin1-matrix}
s^x=\frac{1}{\sqrt{2}}\begin{pmatrix}
0 &1  &0 \\
1&  0& 1\\
0& 1 &0
\end{pmatrix}, \quad s^y=\frac{i}{\sqrt{2}}\begin{pmatrix}
0 &-1  &0 \\
1&  0& -1\\
0& 1 &0
\end{pmatrix}, \quad s^z=\begin{pmatrix}
1 &0  &0 \\
0&  0& 0\\
0& 0 &-1
\end{pmatrix}, \\
\quad s^+ = s^x+is^y = \sqrt{2} \begin{pmatrix}
0 &1  &0 \\
0&  0& 1\\
0& 0 &0
\end{pmatrix}, \quad s^-=  s^x-is^y = \sqrt{2} \begin{pmatrix}
0 &  0& 0\\
1&  0& 0\\
0&1  &0
\end{pmatrix}\ , \qquad \qquad
\end{align}
 that satisfy $[s^z,s^\pm]=\pm s^\pm$, $[s^+,s^-]= 2s^z$.
The special case $\varepsilon_-= - \varepsilon_+^{-1}$, $\overline{\varepsilon}_-= - \overline{\varepsilon}_+^{-1}$ was considered in~\cite{Inami1996}, and their expression~\cite[eq.\,(4.3)]{Inami1996} matches with~\eqref{Ham-params1} for the identification  
	$\eta \rightarrow i \ln(q)$, $u \rightarrow i \ln(u)$  and
	\begin{align*}
		\zeta_- &\rightarrow i \ln(\varepsilon_+) \ ,  &\zeta_+ \rightarrow&\ i \ln(\overline{\varepsilon}_+) \  	\ ,\\
		\mu_- &\rightarrow - k_+ \frac{\sqrt{q+q^{-1}}}{q-q^{-1}} \ ,  &\mu_+\rightarrow&\ \bkp\frac{ \sqrt{q+q^{-1}}}{q-q^{-1}}\ ,\\
		\tilde{\mu}_-&\rightarrow -\km\frac{\sqrt{q+q^{-1}}}{q-q^{-1}} \ , &\tilde{\mu}_+\rightarrow&\  \bkm\frac{\sqrt{q+q^{-1}}}{q-q^{-1}}\ .
	\end{align*}
 We also note that  under the same identification of parameters our fused K-matrix $K^{(1)}(u)$ agrees with the K-matrix used in~\cite{Inami1996}, by Remark~\ref{rem:K-mat-spin1}.
\end{example}

\subsection{Transfer matrix expansion}
    \label{sec:t-expansion}
In Section~\ref{sec:TjN}, we have seen that the TT-relations~\eqref{TT-relTN} implied that $\bt^{(j,N)}(u)$ is a polynomial of order $2j$ in  the generating function $\mathsf{I}^{(N)}(u)$, defined in~\eqref{eq:def-IN}, with $q$-shifted arguments and with coefficients that are central in $\frac{1}{f^{(j)}(u)}\mathcal{A}_q^{(N)}[u,u^{-1}]$, recall the normalization $f^{(j)}(u)$ in~\eqref{def:kp}-\eqref{eq:fj}.  To establish a similar statement for the normalized transfer matrix $\tilde {\boldsymbol  t}^{j,\bj}(u)$, we first establish the following relation between the spin-$\h$ normalized transfer matrix and the spin-chain representation $\psi^{(N)}_{\bj,\bar{v}}$ of the generating function $\mathsf{I}^{(N)}(u)$:

\begin{prop}\label{prop:tIN} 
For any $N$-tuple  of spins $\bj := (j_1, \ldots, j_N)$ and $v_i\in \mathbb{C}^*$, for $1\leq i\leq N$, we have
	\begin{multline}\label{trans-mat-jsIN} 
		\tilde {\boldsymbol  t}^{\h,\bj}(u)= \prod_ {n = 1}^{N} 
    \left(c (uq^{j_n+\h }v_n) c (uq^{j_n+\h }v_n^{-1})\right)^{-1}  \times \Bigl[
   c(u^2)c(u^2q^2)\left( \psi^{(N)}_{\bj,\bar{v}}(\mathsf{I}^{(N)}(u)) + h_0^{(N)}(u) \mathsf{I}_0 \right )  \\ 
	  + \overline{\varepsilon}_+  \left (  (u^2 q + u^{-2} q^{-1})\varepsilon_+^{(N)} + (q+q^{-1}) \varepsilon_-^{(N)} \right ) + \overline{\varepsilon}_- \left ( (u^2q+u^{-2}q^{-1}) \varepsilon_-^{(N)} + (q+q^{-1}) \varepsilon_+^{(N)}    \right )
      \Bigr]\ ,
	\end{multline}
	with $h_0^{(N)}(u)$ from~\eqref{h0N}, the scalar $\mathsf{I}_0$ is defined in~\eqref{I0}, and $\varepsilon_\pm^{(N)}$ are introduced in~\eqref{eps-pmN}.
         We furthermore have the expansion of $\psi^{(N)}_{\bj,\bar{v}}\bigl(\mathsf{I}^{(N)}(u)\bigr)$ in~\eqref{trans-mat-jsIN}:
\begin{equation}\label{eq:psi-IN}
     \psi^{(N)}_{\bj,\bar{v}}\bigl(\mathsf{I}^{(N)}(u)\bigr)  = \sum_{k=0}^{N-1} P_{-k}^{(N)}(u) \psi^{(N)}_{\bj,\bar{v}}\bigl(\mathsf{I}^{(N)}_{2k+1}\bigr)
\end{equation}
with the polynomials $P_{-k}^{(N)}(u)$ introduced in~\eqref{PMK} with their coefficients $d_n^{(N)}$ fixed as in~\eqref{dkN}, and
	\beqa
\psi^{(N)}_{\bj,\bar{v}}\bigl(\mathsf{I}_{2k+1}^{(N)}\bigr) &=& 
	 \overline{\varepsilon}_+{\calW}^{(N)}_{-k} + \overline{\varepsilon}_-{\calW}^{(N)}_{k+1} + \frac{1}{q^2-q^{-2}} \frac{\overline{k}_-}{k_-}{\calG}^{(N)}_{k+1}  + \frac{1}{q^ 2-q^{-2}} \frac{\overline{k}_+}{k_+}\tilde{\calG}^{(N)}_{k+1} \ ,\label{IkOqN}
	 \eeqa
     together with explicit expressions for the spin-chain representation of the alternating generators~\eqref{imOqN}  given in Appendix~\ref{apD}.
    \end{prop}
    
    \begin{proof}
   To show~\eqref{trans-mat-jsIN}, we first recall Proposition~\ref{prop:TNtj} where the r.h.s.\  of~\eqref{evalTN} at $j=\h$  after simplifications using~\eqref{renormtrans} takes the form
   \beqa
   \left(\prod_ {n = 1}^{N} 
    c (uq^{j_n+\h }v_n) c (uq^{j_n+\h }v_n^{-1}) \right)\tilde {\boldsymbol  t}^{\h,\bj}(u)\ .
   \eeqa
 Rewriting the l.h.s.\ of~\eqref{evalTN} at $j=\h$, we use the image of~\eqref{tjsIN} under the algebra map $\psi^{(N)}_{\bj,\bar{v}}$ together with the definition of $\mathsf{I}^{(N)}(u)$ in~\eqref{eq:def-IN} or, equivalently, in~\eqref{calIN}. The resulting equality is equivalent to~\eqref{trans-mat-jsIN}.  Finally, we also note $\psi^{(N)}_{\bj,\bar{v}}(\mathsf{I}_{2k+1}^{(N)}) = \psi^{(N)}_{\bj,\bar{v}} \circ \varphi^{(N)}(\mathsf{I}_{2k+1})$
              where  the mutually commuting elements  $\mathsf{I}_{2k+1}$  
	are given by~\eqref{eq:def-Ik}. This and the identification in~\eqref{mapAqAN}
    derives~\eqref{IkOqN}.
    \end{proof}

    \begin{rem}\label{rem:t-Laurent}
    We note that  entries of the normalized transfer matrix	$\tilde {\boldsymbol  t}^{j,\bj}(u)$ have poles, as can be seen in particular in Proposition~\ref{prop:tIN}.
    Using Proposition~\ref{prop:R-M-P} and Lemma~\ref{lem:K-poly} we see  by the form of~\eqref{eq:TtildeAN} in terms of normalized R-matrices that multiplying $\tilde {\boldsymbol  t}^{j,\bj}(u)$ by 
$$
\prod_{n=1}^{N}\prod_{k=0}^{2j_n-1}c(u q^{j_n+j-k}v_n)c(u q^{j_n+j-k}v_n^{-1})
$$
makes all of its entries Laurent polynomials in $u$.
\end{rem}

By Proposition~\ref{prop:TTt} and using~\eqref{renormtrans}, it immediately follows:
\begin{cor}\label{cor:nt} The normalized transfer matrix~\eqref{tildetjs} satisfies the TT-relation
\begin{align}\label{renormTTN}
			 \tilde{\boldsymbol  t}^{j,\bj}(u) &= \frac{g^{j,\bj}(u)}{g^{j-\h,\bj}(uq^{-\h})g^{\h,\bj}(uq^{j-\h})}\tilde{\boldsymbol t}^{j-\h,\bj}(u q^{-\h}) \tilde {\boldsymbol t}^{\h,\bj}(u q^{j-\h})  \\ \nonumber
			&+
			\left [	
			\displaystyle{\prod_{n=1}^N} \beta^{(j_n)}(u q^{j-\h} v_n) \beta^{(j_n)}(u q^{j-\h} v_n^{-1})
			\right ]  \frac{	\Gamma_{-} (u q^{j-\tha}) \Gamma_{+} (u q^{j-\tha})       }{ c(u^2 q^{2j}) c(u^2 q^{2j-2})} \frac{g^{j,\bj}(u)}{g^{j-1,\bj}(uq^{-1})}\tilde {\boldsymbol t}^{j-1,\bj}(u q^{-1}) \ ,
		\end{align}
\end{cor}

   Using the TT-relations~\eqref{renormTTN}   with the initial condition~\eqref{trans-mat-jsIN}, we conclude that the spin-$j$ normalized transfer matrix $\tilde {\boldsymbol  t}^{j,\bj}(u)$ is expressed in terms of $\psi^{(N)}_{\bj,\bar{v}}\bigl(\mathsf{I}^{(N)}(u)\bigr)$ (with $q$-shifted arguments) as a polynomial of order $2j$. The expressions are determined recursively and are similar to~\eqref{eq:TT-gen1} and~\eqref{eq:TT-gen2}. Up to an overall rational function, inverse to the factor in Remark~\ref{rem:t-Laurent},  coefficients of these polynomials are Laurent polynomials in~$u$, and they depend only on the boundary parameters $k_\pm$, $\varepsilon_{\pm}$ and $\bar{k}_\pm$, $\bar{\varepsilon}_{\pm}$ and  the inhomogeneity parameters $v_n$, $1\leq n\leq N$.
   
Similarly,  the higher Hamiltonians in~\eqref{HnfromT} can be written recursively  as polynomials  of the operators $\psi^{(N)}_{\bj,\bar{v}}\bigl(\mathsf{I}_{2k+1}^{(N)}\bigr)$ from~\eqref{eq:psi-IN} and~\eqref{IkOqN}. 
To illustrate this, let us consider the open XXZ spin-$\h$ chain of Example~\ref{ex:H12}. Details of calculations for the following results are reported in Appendix~\ref{ApE}. For the first two members of the hierarchy, we get: 
\beqa
{\cal H}^{(1)}=\frac{2}{c(q)}
\left(\cal H_{XXZ} - \cal H_0\right) = \frac{4c(q)^{-2N+1}}{(\varepsilon_+ + \varepsilon_-)(\bar\varepsilon_+ + \bar\varepsilon_-)} \psi^{(N)}_{\bar{\h},\bar{1}}(\mathsf{I}^{(N)}(1))  \ , \label{H1XXZ}
\eeqa
where ${\cal H}_0$ is a scalar, and
\beqa
{\cal H}^{(2)}&=& \frac{1}{(q+q^{-1})(\varepsilon_+ + \varepsilon_-)(\bar\varepsilon_+ + \bar\varepsilon_-)}\left(-4c(q^2){c(q)^{-2N}} \psi^{(N)}_{\bar{\h},\bar{1}}(\mathsf{I}^{(N)}(1)) \right.\nonumber\\
&&\qquad\qquad\qquad\qquad\qquad\qquad\qquad \qquad \qquad \qquad\quad  \left. +\; 8\frac{d}{du}\left( c(u^2q^2) g^{\h,\bar \h}(u) \psi^{(N)}_{\bar{\h},\bar{1}}(\mathsf{I}^{(N)}(u)) \right)\Big|_{u=1}\right)\ \label{H2XXZ}\\
&&- \frac{1}{(q+q^{-1})^2(\varepsilon_+ + \varepsilon_-)^2(\bar\varepsilon_+ + \bar\varepsilon_-)^2}\Bigl( 4c(q^2)c(q)^{-2N}\left( \psi^{(N)}_{\bar{\h},\bar{1}}(\mathsf{I}^{(N)}(1)) + h_0^{(N)}(1) \mathsf{I}_0 \right )
\nonumber\\
&&\qquad \qquad \qquad \qquad \qquad 
+ \; {2c(q)^{-2N+1}\left(\overline{\varepsilon}_+  \varepsilon_+^{(N)} +  \overline{\varepsilon}_-\varepsilon_-^{(N)} \right) -2N\frac{(q+q^{-1})^2}{c(q)}\left(\overline{\varepsilon}_+  +  \overline{\varepsilon}_-\right)\left(  \varepsilon_++ \varepsilon_- \right)} \Bigr)^2\ ,\nonumber
\eeqa
with $\varepsilon^{(N)}_\pm$ defined in~\eqref{eps-pmN}.

After these examples, we now turn to formulating a general algorithm of constructing $\cal{H}^{(n)}$ in terms of spin matrices.

 \subsubsection{Algorithm of constructing $\cal{H}^{(n)}$}
 \label{sec:alg-Hn}
 For arbitrary spin-$j$, all 
the conserved quantities ${\cal H}^{(n)}$ are explicitly expressed in terms of $\psi^{(N)}_{\bj,\bar{1}}\bigl(\mathsf{I}_{2k+1}^{(N)}\bigr)$ applying the following algorithm:
\begin{enumerate}
\item Take the normalized transfer matrix ${\tilde {\boldsymbol  t}}^{\h,\bj}(u)$ given by~\eqref{trans-mat-jsIN} -- which is written in terms of $\psi^{(N)}_{\bj,\bar{v}}\bigl(\mathsf{I}_{2k+1}^{(N)}\bigr)$ using~\eqref{eq:psi-IN}  and~\eqref{IkOqN} -- and take its specialization for $j_n=j
$ and $v_n=1$ for all $n$;
\item Using the TT-relation~\eqref{renormTTN},  compute the expression for ${\tilde {\boldsymbol  t}}^{j,\bj}(u)$  in terms of $\psi^{(N)}_{\bj,\bar{1}}\bigl(\mathsf{I}_{2k+1}^{(N)}\bigr)$;
\item  Compute the expression for ${\cal H}^{(n)}$ as the logarithmic derivative~\eqref{HnfromT} as a polynomial in $\psi^{(N)}_{\bj,\bar{1}}\bigl(\mathsf{I}_{2k+1}^{(N)}\bigr)$, it is of order $2jn$;

\item Using the expression~\eqref{IkOqN} of $\psi^{(N)}_{\bj,\bar{1}}\bigl(\mathsf{I}_{2k+1}^{(N)}\bigr)$  in terms of the spin-chain representation of the alternating generators of ${\cal A}_q$, and written in terms of tensor products of spin matrices in Appendix~\ref{apD}, we finally obtain a spin-matrix expression of the higher Hamiltonians  ${\cal H}^{(n)}$. 

\end{enumerate}

 \subsection{Transfer matrix and $q-$Onsager generators}\label{sec:tr-mat-Ons}
We now show that the normalized spin-$j$ transfer matrix $\tilde {\boldsymbol  t}^{j,\bj}(u)$ can be  expressed solely in terms of the two fundamental operators $\calW_0^{(N)}$, $\calW_1^{(N)}$ in  the spin-chain representations, given by~\eqref{W0vN}, \eqref{W1vN}. Indeed, one sees from~\eqref{GammaN} and~\eqref{evalgamN} that the quantum determinant $\Gamma(u)$ of $\mathcal{A}_q$ evaluates on the spin-chain representation to a formal Laurent series in $u$ with scalar coefficients:
\begin{equation}\label{eq:Gamma-rep}
   \psi^{(N)}_{\bj,\bar{v}}\circ \varphi^{(N)}\bigl(
   \Gamma(u)\bigr) = 
   \frac{\psi^{(N)}_{\bj,\bar{v}}(\Gamma^{(N)}(u))}{	c(u^2)c(u^2q^2)h_0^{(N)}(u)h_0^{(N)}(uq)} \ ,
\end{equation}
and recall~\eqref{h0N} with~\eqref{dkN}.
Therefore, as discussed in Section~\ref{sub:TT-qOA}, the representation map $\psi^{(N)}_{\bj,\bar{v}}\circ \varphi^{(N)}$ of ${\cal A}_q$, or equivalently of the algebra ${\cal A}_q^{(N)}$, factorizes through the $q$-Onsager algebra $O_q$ via the map $\gamma^{(\delta)}$ defined by~\eqref{mapDeldel}, and with the image $\gamma^{(\delta)}\bigl(\Gamma(u)\bigr)$ as in Lemma~\ref{lem:qdetOq}. In other words, the image $\psi^{(N)}_{\bj,\bar{v}}\bigl({\cal A}_q^{(N)} \bigr)$ is a (finite-dimensional) quotient of $O_q$ which is generated by the fundamental operators $\calW_0^{(N)}$ and $\calW_1^{(N)}$.
In particular, the alternating operators  ${\calW}^{(N)}_{-k}$, ${\calW}^{(N)}_{k+1}$, ${\calG}^{(N)}_{k+1}$, and $\tilde{\calG}^{(N)}_{k+1}$ are certain polynomials~\eqref{recGk}, \eqref{recWk} of  the fundamental operators $\calW_0^{(N)}$, $\calW_1^{(N)}$ with coefficients that depend on set of scalars $\delta_k$ with $1 \leq k\leq N$.
 The scalars $\delta_k$ are determined by combining~\eqref{evalgamN}, \eqref{eq:Gamma-rep} and \eqref{imgamma}, which gives
\begin{align*}
		 \sum_{k=1}^{\infty} c_k\delta_k u^{-2k} = \frac{c(q)\left [  \displaystyle{ \prod_{n=1}^N  }
				c(u q^{j_n+\frac{3}{2}} v_n)c(u q^{j_n+\frac{3}{2}} v_n^{-1})c(u q^{-j_n+\h} v_n)
				c(u q^{-j_n+\h} v_n^{-1}) \right ]  \ \Gamma_{-} (u)}{c(u^2) c(u^2 q^2)^2 h_0^{(N)}(u) h_0^{(N)}(uq) } + \frac{\rho}{c(q)}\  
\end{align*} 
with $c_k$ given in~\eqref{eq:ckp1}.
For instance, we find:
\begin{align}
\mbox{For $N=1$:}\qquad \delta_1&= 0\ ,\nonumber \\
\delta_2 &= \frac {c(q)}{q^2+q^{-2}} \Big ( \frac{\varepsilon_+ \varepsilon_-}{(q+q^{-1})^2} (q^2+q^{-2}) w_0^{(j_1)} (v_1^2+v_1^{-2}) \nonumber \\ 
&- \frac{k_+ k_-}{c(q)^2 (q+q^{-1})^2} c(q^{2j_1}) c(q^{2j_1+2}) c(q v_1^2) c(q^{-1} v_1^2)  - \frac{\varepsilon_+^2 \varepsilon_-^2 c(q)^2 }{\rho}  -\varepsilon_+^2 -\varepsilon_-^2 \Big ) \ . \nonumber\\
\mbox{For $N=2$:}\qquad \delta_1&=0 \ ,\nonumber\\ \delta_2&= 
\delta_2|_{N=1} - \frac{k_+ k_- c(q^{2j_2})  c(q^{2j_2+2}) c(q v_2^2) c(q^{-1} v_2^2)}{c(q^2)(q+q^{-1})(q^2+q^{-2})}\nonumber\ .
\end{align}

\begin{rem} For $N=0$ in~\eqref{eq:Gamma-rep}, the corresponding $\delta_k$'s  coincide with~\eqref{delpair},~\eqref{delimp} of Example~\ref{gameps}. Indeed, for $N=0$ 
we have $h^{(0)}_0(u) = (q+q^{-1})^{-1}$ due to $d_0^{(0)}=-1$. Using~\eqref{evalgamN} at $N=0$,
it follows that the r.h.s.\ of~\eqref{eq:Gamma-rep} matches with the second line of~\eqref{eq:durho}.
\end{rem}

Using the resulting polynomial expressions~\eqref{recGk} and~\eqref{recWk} for the alternating generators of ${\cal A}_q$, 
we get a corresponding polynomial expression for $\psi^{(N)}_{\bj,\bar{v}}\bigl(\mathsf{I}_{2k+1}^{(N)}\bigr)$ in~\eqref{IkOqN} in the fundamental operators $\calW_0^{(N)}$ and $\calW_1^{(N)}$. By Remark~\ref{deg}, note that the polynomial $\psi^{(N)}_{\bj,\bar{v}}\bigl(\mathsf{I}_{2k+1}^{(N)}\bigr)$ is of degree $2k+2$. For instance for $k=0$, taking the image via $\psi^{(N)}_{\bj,\bar{v}}$ of $G_1$ in~\eqref{defel}, we get:
\begin{equation}
	\psi^{(N)}_{\bj,\bar{v}}(\mathsf{I}^{(N)}_1)= \overline{\varepsilon}_+ \mathcal{W}^{(N)}_0 + \overline{\varepsilon}_- \mathcal{W}^{(N)}_1  +\frac{1}{q^2-q^{-2}} \left ( \frac{\overline{k}_-}{k_-}  \lbrack \mathcal{W}^{(N)}_1,\mathcal{W}^{(N)}_0 \rbrack_q     +  \frac{\overline{k}_+}{k_+}  \lbrack \mathcal{W}^{(N)}_0,\mathcal{W}^{(N)}_1\rbrack_q  \right ) \ .
\end{equation}
Finally, using~\eqref{trans-mat-jsIN}
one gets an expression for $\tilde {\boldsymbol  t}^{\h,\bj}(u)$ in terms of $\calW_0^{(N)}$ and $\calW_1^{(N)}$. We see from~\eqref{eq:psi-IN}  that the highest degree term is $\psi^{(N)}_{\bj,\bar{v}}\bigl(\mathsf{I}_{2N-1}^{(N)}\bigr)$, and therefore $\tilde {\boldsymbol  t}^{\h,\bj}(u)$ has total degree $2N$ in $\calW_0^{(N)}$, $\calW_1^{(N)}$.
More generally, using the TT-relation~\eqref{evalTN}, it follows: 
	\begin{cor}\label{corfin} The normalized transfer matrix $\tilde {\boldsymbol  t}^{j,\bj}(u)$ is a polynomial of total degree $4Nj$ in the fundamental alternating operators $\calW_0^{(N)}$, $\calW_1^{(N)}$. 
	\end{cor}
   By this corollary  all local conserved quantities ${\cal H}^{(n)}$ defined by~\eqref{HnfromT} can be expressed as certain polynomials  
    of total degree $4nNj$ in the fundamental operators  $\calW_0^{(N)}$ and $\calW_1^{(N)}$. For instance, for the open XXZ spin-$\h$ chain, ${\cal H}^{(1)}$ and ${\cal H}^{(2)}$ given by~\eqref{H1XXZ} and~\eqref{H2XXZ} are polynomials of total degree $2N$ and~$4N$, respectively.  \medskip

 \section{On exchange relations and symmetries} \label{sec6}
From Section~\ref{sec:alg-Hn}, recall that the spin-chain Hamiltonian ${\cal{H}^{(1)}}$ is a polynomial of order $2j$ in the spin-chain representation $\psi^{(N)}_{\bj,\bar{v}}$ of the elements $\mathsf{I}_{2k+1}\in \cal A_q$ introduced in~\eqref{eq:def-Ik}, using the identification $\psi^{(N)}_{\bj,\bar{v}}(\mathsf{I}^{(N)}_{2k+1})= \psi^{(N)}_{\bj,\bar{v}} \circ\varphi^{(N)}(\mathsf{I}_{2k+1})$. These elements $\mathsf{I}_{2k+1}$ depend on the `right' side boundary parameters $\overline{\varepsilon}_\pm,\bar k_\pm$  and the `left' side boundary parameters $k_\pm$.
The goal of this section is twofold. Firstly, using the PBW basis of $\cal A_q$ we 
study non-trivial exchange relations in $\mathcal{A}_q$ between $\mathsf{I}_{2k+1}$  and the generators $\cW_0$ and $\cW_1$ or their non-trivial linear combination. This leads to Proposition~\ref{prop:lin-comb-Tj} identifying the linear combination that commutes with the universal transfer matrix $\bt^{(j)}(u)$ of spin-$j$  for $\bar k_\pm=0$. Secondly, we consider applications to integrable spin-$j$ chains. Namely, for certain relations between the \textsl{both side} boundary parameters it is shown that the exchange relations lead to non-trivial symmetries of the spin-$j$ Hamiltonians of the spin chain, generalizing the known results for $j=\h$ \cite{Doikou}.  Finally, we give comments on such symmetries at $q=1$ or in the XXX-type Hamiltonians case.

\subsection{Exchange relations in $\mathcal{A}_q$}

Recall that the abelian subalgebra $\mathcal{I}$ of $\mathcal{A}_q$ is generated by the elements $\{ \mathsf{I}_{2k+1} | k \in \mathbb{N}  \}$ given in~\eqref{eq:def-Ik}.
Below we focus successively on the exchange relations that involve only $\cW_0,\cW_1$ and the sum $a\cW_0 + b\cW_1$.
For convenience introduce 
$$
\mathsf{I}_{2k+1}' = \overline{\varepsilon}_+'{\tW}_{-k} + \overline{\varepsilon}_-'{\tW}_{k+1} + \frac{1}{q^2-q^{-2}} \left (  \frac{\overline{k}_-'}{k_-} {\tG}_{k+1}  + \frac{\overline{k}_+'}{k_+} \tilde{\tG}_{k+1} \right ) \ , \qquad k\in\mathbb{N}\ ,
$$
for some $\overline{k}_\pm',\overline{\varepsilon}_\pm' \in \mathbb{C}$. 

\begin{prop}\label{LemmaCondW0}
	We have the following  exchange relations, for all $k\in\mathbb{N}$,
	\begin{equation}\label{eqcondW0} {\normalfont \tW}_0 \mathsf{I}_{2k+1} = \mathsf{I}'_{2k+1}{\normalfont \tW}_0
    \end{equation}
	if and only if the parameters  satisfy
	\begin{equation}\label{condW0}
		\overline{\varepsilon}_+' = \overline{\varepsilon}_+ \ , \quad \overline{\varepsilon}_-'=\overline{\varepsilon}_-=0 \ , \quad \frac{\overline{k}_+}{k_+} = \frac{\overline{k}_-}{k_-} q^{-2} \ , \quad \overline{k}_\pm' = \overline{k}_\pm q^{\pm 2} \ .
	\end{equation}
\end{prop}
\begin{proof}  With respect to the choice of PBW basis~\eqref{eq:WGGW} \& \eqref{order}, the l.h.s.\ of~\eqref{eqcondW0} is a combination of the linearly independent ordered monomials  $\tW_0\tW_{-k}$, $\tW_0\tW_{k+1}$, $\tW_0\tG_{k+1}$, and $\tW_0\tilde{\tG}_{k+1}$. Reordering the r.h.s.\ of~\eqref{eqcondW0} using the defining relations~\eqref{qo1},~\eqref{qo2},~\eqref{qo4} of ${\cal A}_q$, we get:
	\begin{align} 
		\mathsf{I}'_{2k+1}{\tW}_0  =& 
		\overline{\varepsilon}'_+ {\tW}_0 {\tW}_{-k} +\overline{\varepsilon}'_- 
		{\tW}_0 {\tW}_{k+1}  + \frac{1}{q^2-q^{-2}} \Big ( \frac{\overline{k}'_+}{k_+}  q^{-2} {\tW}_0\tilde{\tG}_{k+1} + \frac{\overline{k}'_-}{k_-}  q^2 {\tW}_0 {\tG}_{k+1} \Big) \label{badterms}\\ 
		- & \overline{\varepsilon}'_-\frac{\tilde{\tG}_{k+1}-{\tG}_{k+1}}{q+q^{-1}}  + \frac{1}{q^2-q^{-2}} \rho  \Bigr(\frac{\overline{k}'_+}{k_+}  q^{-1} - \frac{\overline{k}'_-}{k_-}  q \Bigl)\big( {\tW}_{-k-1} - {\tW}_{k+1} \big ) \ .\nonumber
	\end{align}
	With this equation we equate in~\eqref{eqcondW0} coefficients in front of linearly independent monomials 
	$\tW_0\tW_{-k}$, $\tW_0\tW_{k+1}$, $\tG_{k+1}$, $\tilde{\tG}_{k+1}$, $\tW_0\tG_{k+1}$, $\tW_0\tilde{\tG}_{k+1}$, which
	implies that the equality~\eqref{eqcondW0} holds if and only if the constraints~\eqref{condW0} are satisfied.  
\end{proof}
Exchange relation between $\mathsf{I}_{2k+1}$'s and  ${\tW}_1$ is obtained similarly:

\begin{prop}\label{LemmaCondW1} We have the following exchange relation,  for all $k\in\mathbb{N}$,
	\begin{equation} \label{eqcondW1} {\normalfont \tW}_1 \mathsf{I}_{2k+1} = \mathsf{I}_{2k+1}'{\normalfont \tW}_1
    \end{equation}
	if and only if the parameters  satisfy
	\begin{equation}\label{condW1}
		\overline{\varepsilon}_+'=\overline{\varepsilon}_+=0 \ , \quad \overline{\varepsilon}_-' = \overline{\varepsilon}_- \ , \quad \frac{\overline{k}_+}{k_+} = \frac{\overline{k}_-}{k_-} q^{2} \ , \quad \overline{k}_\pm' = \overline{k}_\pm q^{\mp2} \ .
	\end{equation}
\end{prop}
Combining calculations in the proofs of Propositions~\ref{LemmaCondW0} and~\ref{LemmaCondW1},  exchange relations with a linear combination  of $\cW_0,\cW_1$ take the following form:

\begin{prop}\label{LemmaCondW0W1} We have the following exchange relation,  for all $k\in\mathbb{N}$ and non-zero $a,b\in\mathbb{C}$,
	\begin{equation} \label{eqcondW0W1} 
		\left( a{\normalfont \tW}_0 + b{\normalfont \tW}_1\right) \mathsf{I}_{2k+1} = \mathsf{I}_{2k+1}'	\left( a{\normalfont \tW}_0 + b{\normalfont \tW}_1\right)\ \end{equation}
	if and only if the parameters  satisfy
	\begin{equation}\label{condW0W1}
		\overline{\varepsilon}_\pm'=\overline{\varepsilon}_\pm\ ,\quad \overline{\varepsilon}_- = \frac{b}{a}\overline{\varepsilon}_+  \ ,\quad    \overline{k}_\pm  = \overline{k}'_\pm  = 0 \ .
	\end{equation}
\end{prop}
Recall that $\mathsf{I}_{2k+1}$ in~\eqref{eq:def-Ik} depend on parameters $\overline{\varepsilon}_\pm\in \mathbb{C}$.
From the above propositions, we obtain commutativity of the corresponding generating function $\mathsf{I}(u)$, given by~\eqref{t12init}, with certain linear combination of ${\tW}_0$ and ${\tW}_1$ depending on these parameters $\overline{\varepsilon}_\pm$:
\begin{cor} \label{cor:w0-w1}
	For any values $\overline{\varepsilon}_\pm\in\mathbb{C}$ and  $k_\pm\in \mathbb{C}^*$, and
all $k\in\mathbb{N}$, we have the following commutativity relations:
\begin{equation}
		\bigl[ \mathsf{I}_{2k+1},  \overline{\varepsilon}_+{\tW}_0 + \overline{\varepsilon}_-{\tW}_1 \bigr] = 0 \ ,
\end{equation}
and consequently 
$$
\bigl[ \mathsf{I}(u),  \overline{\varepsilon}_+{\tW}_0 + \overline{\varepsilon}_-{\tW}_1 \bigr] = 0\ ,
$$
if and only if $\overline{k}_\pm =0$.
\end{cor}

We now recall the result at the end of Section~\ref{sub:TT-conj} that the universal transfer matrix $\bt^{(j)}(u)$ of spin-$j$ defined in~\eqref{tg} is a polynomial of order $2j$ in the generating function  $\mathsf{I}(u)$ with shifted arguments, and coefficients that are central in $\mathcal{A}_q$, see e.g.\ \eqref{eq:TT-gen1} and~\eqref{eq:TT-gen2}. We thus get the following result from Corollary~\ref{cor:w0-w1}:

\begin{prop}\label{prop:lin-comb-Tj}
	For $j\in\h\mathbb{N}_+$ and any values $\overline{\varepsilon}_\pm\in\mathbb{C}$ and  $k_\pm\in \mathbb{C}^*$, we have
    $$
\bigl[ \bt^{(j)}(u),  \overline{\varepsilon}_+{\tW}_0 + \overline{\varepsilon}_-{\tW}_1 \bigr] = 0\ ,
$$
if and only if $\overline{k}_\pm =0$.
\end{prop}

\subsection{Application to quantum integrable spin chains}
Based on the exchange relations from the previous subsection, the goal is now to exhibit exchange relations between the spin-$j$ Hamiltonians and  the spin-chain representations of $\normalfont {\tW}_0$, $ \normalfont{\tW}_1$, defined by recursion in~\eqref{W0vN} and~\eqref{W1vN} and denoted by $\cal W_0^{(N)}$ and  $\cal W_1^{(N)}$, respectively.
We first recall that for the choice~\eqref{eq:j-homo} and by~\eqref{HnfromT} the spin-$j$ Hamiltonians\footnote{We notice that previous definition $\mathcal{H}_{XXZ}^{\textsf{spin $\h$}}$ and $\mathcal{H}_{XXZ}^{\textsf{spin $1$}}$, given in~\eqref{HXXZhalf} and~\eqref{Ham-params1} respectively, differs from this current one by a constant term and an overall factor which are not important in the following analysis.}
 $$
\mathcal{H}_{XXZ}^{\textsf{spin $j$}} := \mathcal{H}^{(1)}
$$
 are all generated from the normalized transfer matrix $\tilde {\boldsymbol  t}^{j,\bj}(u)$.
 In what follows we use the following convenient parametrization:
 \begin{equation} \label{oldparamH}
	h_\pm = \frac{2 k_\pm}{\varepsilon_+ + \varepsilon_-}, \qquad \bh_\pm = \frac{ 2 \overline{k}_\pm} { \overline{\varepsilon}_+ + \overline{\varepsilon}_- }, \qquad h_z = \frac { \varepsilon_+ - \varepsilon_- }{ \varepsilon_+ + \varepsilon_-}, \qquad \bh_z=\frac{ \overline{\varepsilon}_+ - \overline{\varepsilon}_- }{ \overline{\varepsilon}_+ + \overline{\varepsilon}_- }\ ,
\end{equation}
and to indicate dependence on these boundary parameters we write  $\mathcal{H}_{XXZ}^{\textsf{spin $j$}}(h_\pm, h_z, \bh_\pm,\bh_z)$.
We recall that the transfer matrix $\tilde {\boldsymbol  t}^{j,\bj}(u)$, as well as the Hamiltonian $\mathcal{H}_{XXZ}^{\textsf{spin $j$}}$, is a polynomial in  operators
$\psi^{(N)}_{\bj,\bar{v}}(\mathsf{I}^{(N)}_{2k+1})= \psi^{(N)}_{\bj,\bar{v}} \circ\varphi^{(N)}(\mathsf{I}_{2k+1})$. As the conditions in Propositions~\ref{LemmaCondW0} and~\ref{LemmaCondW1} don't depend on the $k$-index of $\mathsf{I}_{2k+1}$,
we  get a simple but important conclusion:
	\begin{prop} \label{corol1}  For all spin values $j\in \h \mathbb{N}$, we have the exchange relations:
\begin{enumerate}
    \item
		\begin{equation*}
		\mathcal{H}_{XXZ}^{\textsf{spin $j$}}(h_\pm, h_z, \bh_\pm q^{\pm 2},1) {\cal W}_0^{(N)} 
		=
		{\cal W}_0^{(N)}  \mathcal{H}_{XXZ}^{\textsf{spin $j$}}(h_\pm, h_z, \bh_\pm,1),
		\end{equation*}
		under the condition that 
		\begin{equation*}
		h_+\bh_-=q^2h_- \bh_+ ;
		\end{equation*}

\item 
		\begin{equation*}
		\mathcal{H}_{XXZ}^{\textsf{spin $j$}}(h_\pm, h_z, \bh_\pm q^{\mp 2},-1) {\cal W}_1^{(N)} 
		=
		{\cal W}_1^{(N)}  \mathcal{H}_{XXZ}^{\textsf{spin $j$}}(h_\pm, h_z, \bh_\pm,-1),
		\end{equation*}
		under the condition that 
		\begin{equation*}
		h_+\bh_-=q^{-2}h_- \bh_+ .
		\end{equation*}
\end{enumerate}
	\end{prop}

\begin{rem}\label{rem:direct-calc}
In the spin-$\h$ case, the parameters $h_{\pm}$, $h_z$, etc.\ have a clear meaning. The Hamiltonian takes the following form, compare with~\eqref{HXXZhalf}:
	\begin{align} \nonumber 
	H(h_\pm,h_z,\bh_\pm, \bh_z)= \displaystyle{ \sum_{k=1}^{N-1}} \Big ( \sigma_k^x \sigma_{k+1}^x + \sigma_k^y\sigma_{k+1}^y + \frac{q+q^{-1}}{2} \sigma_k^z \sigma_{k+1}^z \Big ) &+ \frac{q-q^{-1}}{2} h_z\sigma^z_1 + h_+\sigma^+_1 + h_-\sigma^-_1 \\ \label{Hamparam}
	&+ \frac{q-q^{-1}}{2}\bh_z\sigma^z_N + \overline{ h}_+\sigma^+_N + \overline{ h}_-\sigma^-_N \ .
	\end{align}
 With this form,  the result (1) of Proposition~\ref{corol1} can be obtained by a direct calculation (and for any values of $q$) in the spin-$\h$ case without the assumption~\eqref{oldparamH} but under the extra conditions
$\varepsilon_+ h_+=k_+(1+h_z)$ and $k_- h_+ =  k_+ h_-$. These conditions hold automatically after requiring~\eqref{oldparamH}. And similarly for the result (2) of Proposition~\ref{corol1}.
We omit details for brevity as the calculation is long but straightforward. An important observation is 
that this result holds even at $q=1$ or XXX spin-chain case because all the operators are well defined at this value of $q$. We notice however that the above analysis in $\cal A_q$ leading to Proposition~\ref{corol1}  is valid  for generic values of  $q$ only, in particular $q\neq 1$. 
\end{rem}

As a consequence, we identify boundary conditions when the  spin-$j$ Hamiltonians commute with the action of generators of $\cal A_q$ or their linear combination.
 Let us introduce notations for Hamiltonians of different boundary conditions type:
\begin{align}
	\mathcal{H}_{XXZ}^{\textsf{spin $j$} \, +}(h_\pm, h_z)
	 &= \mathcal{H}_{XXZ}^{\textsf{spin $j$ }}(h_\pm, h_z,0,-1) \ ,\label{eq:Ham-sym-1} \\
	 	\mathcal{H}_{XXZ}^{\textsf{spin $j$} \, -}(h_\pm, h_z)
	 &= \mathcal{H}_{XXZ}^{\textsf{spin $j$ }}(h_\pm, h_z, 0,1) \ ,\label{eq:Ham-sym-2} \\
	 	\mathcal{H}_{XXZ}^{\textsf{spin $j$}\, *}(h_\pm, h_z, \bh_z)
	 &= \mathcal{H}_{XXZ}^{\textsf{spin $j$}}(h_\pm, h_z,0,\bh_z) \ .\label{eq:Ham-sym-3}
\end{align}
By Proposition~\ref{corol1}, we get 
\begin{prop}\label{prop:HXXZ-comm-cond}
 Fixing all the quantum spins $j_n$ to be the auxiliary spin $j\in \mathbb{N}_+$ and all $v_n=1$, the following commutation relations hold for any $j \in \frac 12 \mathbb{N}_+$ and any values $h_z\in\mathbb{C}$ and $h_\pm\in\mathbb{C}^*$, and $\bar{h}_z$ fixed as in~\eqref{oldparamH}:
	\begin{align} \label{eq:commutrel}
		\big [  \mathcal{H}_{XXZ}^{\textsf{spin $j$} \, -}(h_\pm, h_z) ,   {{\cal W}_0^{(N)}}
		 \big ] = 0 \ ,  
		 \qquad \big [ \mathcal{H}_{XXZ}^{\textsf{spin $j$} \, +}(h_\pm, h_z) , {{\cal W}_1^{(N)}} \big ] = 0 \ 
	\end{align}
and
\beqa \label{eq:commutrel2}
\big [  \mathcal{H}_{XXZ}^{\textsf{spin $j$} \, *}(h_\pm, h_z,\bh_z) ,\    \overline{\varepsilon}_+{\cal W}_0^{(N)} + \overline{\varepsilon}_-{\cal W}_1^{(N)} \big ] = 0 
\eeqa
	with~\eqref{W0vN},~\eqref{W1vN}.
\end{prop}
We can upgrade this result at the level of transfer matrices, using Corollary~\ref{cor:w0-w1}. Combining the TT-relations~\eqref{renormTTN} together with~\eqref{trans-mat-jsIN},~\eqref{eq:psi-IN}, it follows:
\begin{prop}\label{prop:tr-mat-w0-w1}
For any $N$-tuple  of spins $\bj := (j_1, \ldots, j_N)$ and $v_i\in \mathbb{C}^*$, for $1\leq i\leq N$, and  any $j\in\h\mathbb{N}_+$ and any values $\overline{\varepsilon}_\pm\in\mathbb{C}$ and $k_\pm\in\mathbb{C}^*$, assuming $\overline{k}_\pm =0$, we have
$$  \bigl[ \tilde{\boldsymbol  t}^{j,\bj}(u),  \overline{\varepsilon}_+{\W}_0^{(N)} + \overline{\varepsilon}_-{\W}_1^{(N)} \bigr] = 0\ ,$$
with~\eqref{W0vN},~\eqref{W1vN}.
\end{prop} 

Note that for the special case of $j_n=j=\h$ and $v_n=1$, for $1\leq n\leq N$, the result of Proposition~\ref{prop:tr-mat-w0-w1} is consistent with the  $U(1)$ symmetry identified in~\cite{Doikou}. 
 The advantage of our universal approach based on $\cal A_q$ is that it gives the analogous $U(1)$ symmetry for all values of spins and inhomogeneities. For example, in the spin-1 case the Hamiltonian with such symmetry, e.g. in~\eqref{eq:commutrel2}, is obtained from~\eqref{Ham-params1} by setting $\overline{k}_\pm=0$ (i.e. $\bh_\pm=0$). The spin-$1$ Hamiltonian $\mathcal{H}_{XXZ}^{\textsf{spin $j$} \, *}(h_\pm, h_z,\bh_z)$ then becomes~\eqref{Ham-params1} with the only change for the boundary term
\begin{align}
&\hspace{-0.3cm} {\cal H}_{N}^b(\overline{\varepsilon}_+,\overline{\varepsilon}_-,0,0) = \frac{q^2-q^{-2}}{2(\overline{\varepsilon}_+ \overline{\varepsilon}_-(q+q^{-1}) + \overline{\varepsilon}_+^2 + \overline{\varepsilon}_-^2 )}
\Bigg ( \overline{\varepsilon}_+ \overline{\varepsilon}_- (q-q^{-1})(s^z_N)^2  + (\overline{\varepsilon}_+^2 - \overline{\varepsilon}_-^2) s^z_N\Bigg) \nonumber\ ,
\end{align}
which can be also written in terms of $\bh_z$ using $\bh_z \pm 1= \pm 2 \overline{\varepsilon}_\pm/(\overline{\varepsilon}_+ + \overline{\varepsilon}_-)$.

\newcommand{\E}{E}
\renewcommand{\F}{F}
\newcommand{\K}{K}
\subsection{XXX case}\label{sec6:XXX} The Hamiltonian~\eqref{Hamparam} exihibits special properties at $q=1$, as we now show.  At $q=1$, the operator $\W_0^{(N)}$ is the spin-chain representation of $k_+\E + k_- \F + \varepsilon_+ \ds{1}$.
		Similarly $\cal W_1^{(N)}$ is the representation of $k_- \F +  k_+\E+ \varepsilon_-  \ds{1}$. We note that $\W_0^{(N)} \propto \W_1^{(N)}$ up to a constant term. 
 Therefore, in what follows we only consider $\W_0^{(N)}$.
		We further observe that  the XXZ Hamiltonian~\eqref{Hamparam} becomes the XXX Hamiltonian with $h_z=\bh_z=0$, because they have a factor $(q-q^{-1})/2$ that vanishes at $q=1$, and with generic off-diagonal boundary terms on both sides of the spin-chain:
		\begin{equation} \cal H_{XXX} (h_\pm,0,\bh_\pm,0) = \cal H_{XXZ}(h_\pm,h_z,\bh_\pm,\bh_z)\big \vert_{q\rightarrow 1}
		\end{equation}
		and with the parametrization~\eqref{oldparamH}.
	With the comments in Remark~\ref{rem:direct-calc}, we obtain the commutation relation of the XXX Hamiltonian with $\W_0^{(N)}$:
		\begin{equation}\label{XXXU1} \left [\cal H_{XXX}(h_\pm,0,\bh_\pm,0),\W_0^{(N)} \right ]=0 
		 \end{equation}
		under the only condition
		\begin{equation}\label{condq=1}
		 \bh_+h_-=\bh_-h_+\ .
		\end{equation}

For generic values of $q$, eigenvectors and corresponding eigenvalues of $\W_0^{(N)}$ (resp. $\W_1^{(N)}$) given by~\eqref{W0vN} (resp.~\eqref{W1vN}) have been derived for any spin values $j_n$ and inhomogeneities $v_n$  in~\cite[Prop.\,3.1]{BVZ16} in terms of $q$-Pochhammer polynomials. For $j_n=j=\h$, we refer to~\cite[Sect.~2.3]{BK07b} for expressions in spin representation. 
In particular, in the latter case the eigenspaces of $\cal W_0^{(N)}$ have same dimensions, equal to binomials $\binom{N}{n}$,  as for $S_z$, suggesting that $\cal W_0^{(N)}$ plays the same  role for the Hamiltonian  $\mathcal{H}_{XXZ}^{\textsf{spin $j$} \, -}(h_\pm, h_z)$ as the spin operator $S_z$ for the XXZ Hamiltonian with diagonal boundary conditions on both sides.  

For completeness, here we give the analogous result for $q=1$ and all spins $j=j_n=\h$ and $v_n=1$, $1\leq n\leq N$.
		By recurrence we obtain the $N+1$ distinct eigenvalues of $\cal W_0^{(N)}$ , denoted $\lambda_n^{(N)}$, degenerated $\binom{N}{n}$ times:
		\begin{equation} \lambda_n^{(N)} = (N-2n) \sqrt{\eta} + \mu_0, \qquad  n \in \{ 0, \ldots , N\}\ ,
		\end{equation}
		 where  $\eta=k_+/k_-$ and $\mu_0=\varepsilon_+$.  The corresponding eigenvectors $\psi_{n[k]}^{(N)}$ are defined recursively: 
		\begin{align} \label{eq:eigen-W0-1}
		\psi_{n[k]}^{(N+1)} &= \psi_{n[k]}^{(N)} \otimes \big ( \eta^{1/2} \ket{\uparrow} + \ket{\downarrow} \big )  \qquad k \in \bigl\{ 1, \ldots, \binom{N}{n} \bigr\}, \\
        \label{eq:eigen-W0-2}
		\psi_{n[k]}^{(N+1)} &=  \psi_{n-1 [ k - \binom{N}{n}]}^{(N)} \otimes \big ( - \eta^{1/2} \ket{\uparrow} + \ket{\downarrow} \big )  \qquad k \in \bigl\{ 1+ 
        {\binom{N}{n}}, \ldots, {\binom{N+1}{n}} \bigr\},
		\end{align}
		where $\ket{\uparrow} = (1,0), \; \ket{\downarrow} = (0,1)$, with the initial values :
		\begin{equation} \psi_{0[1]}^{(1)} = \eta^{1/2}  \ket{ \uparrow} + \ket{\downarrow}  , \qquad \psi_{1[1]}^{(1)} = -\eta^{1/2} \ket{ \uparrow} + \ket{\downarrow}.
		\end{equation}
		According to~\eqref{XXXU1}, $\cal H_{XXX} (h_\pm,0,\bh_\pm,0) $ enjoys a $U(1)$ symmetry, so that $\cal W_0^{(N)}$ plays the role of the spin operator $S_z$, with a basis for each new  ``spin" sector given by~\eqref{eq:eigen-W0-1} and~\eqref{eq:eigen-W0-2}.

	\section{Further applications and perspectives}\label{sec7}
	We now briefly  present further applications  of our results and possible developments   in representation theory and in  quantum integrable systems. 
    In Section~\ref{sec:TY}, we construct the T- and Y-systems for $\cal A_q$.
    In Section~\ref{subsec:univ-TQ-Aq}, using a limiting procedure we derive universal TQ-relations for $\cal{A}_q$ that match well on spin-chain representations with~\cite{FNR07}, as well as universal TQ-relations for its $q-$Onsager quotient $O_q$.
    In Section~\ref{subsec:univ-TQ-Aq-aug}, we discuss the analogous TT- and TQ-relations for the case of a central extension of the augmented $q$-Onsager algebra responsible for diagonal boundary conditions. In Section~\ref{sec7:Bethe} we motivate diagonalization of $\mathsf{I}_{2k+1}$'s on spin-chains
     in the form of Bethe states and their interpretation within the theory of tridiagonal pairs. Finally, in Section~\ref{sec7:blob}  we give interpretation of the symmetry results from Proposition~\ref{prop:HXXZ-comm-cond} in the context of the so-called blob algebra generated by Hamiltonian densities.
	
\subsection{T- and Y-system for comodule algebras}\label{sec:TY} Universal structures known as {\it T-system} and {\it Y-system} have been extensively studied. For a review on this subject, see \cite{kunib}. Although originally discovered as functional relations among commuting transfer matrices for solvable lattice models in statistical mechanics, T- and Y-systems have also played an important role in the representation theory of quantum affine algebras. Namely, T-systems can be interpreted in the form of exact sequences involving tensor products of Kirillov-Reshetikhin modules \cite{KNS,Nak,Hern}. As a consequence, corresponding  $q$-characters as introduced in \cite{FR98} solve a T-system.
	\smallskip
	
	Starting from the TT-relations (\ref{TT-rel}) satisfied by (\ref{tg}) the T-system associated with ${\cal A}_q$ is easily derived. It reads:
	\begin{equation}
	\bt^{(j)}(u q^{-\h}) \bt^{(j)}(u q^{\h}) = \bt^{(j+\h)}(u) \bt^{(j-\h)}(u) + {\bf g}^{(j)}(u) \ ,\label{Tsys}
	\end{equation}
	where 
	\begin{equation}\label{def:gj}
	{\bf g}^{(j)}(u)= (-1)^{2j}\displaystyle{ \prod_{\ell=0}^{2j-1} } \frac{\Gamma (uq^{j-1-\ell}) \Gamma_+ (uq^{j-1-\ell})}{ c(u^2 q^{2j+1-2\ell}) c(u^2 q^{2j-1-2\ell}) }\ . 
	\end{equation}
    and $\Gamma(u)$ is the generating function~\eqref{gammaform} of central elements of ${\cal A}_q$.
	Recall that $\Gamma(u)$ is invertible, see 
	below~\cite[Lem.\,5.4]{LBG}, 
	therefore it follows that ${\bf g}^{(j)}(u)$ is invertible as well.
	The corresponding Y-system for ${\cal A}_q$ readily follows:
	\begin{equation} \label{Ysys}
	\by^{(j)}(u q^{-\h}) \by^{(j)}(u q^{\h}) = \bigl(\by^{(j+\h)}(u)+1\bigr)\bigl(\by^{(j-\h)}(u)+1\bigr)  
	\end{equation}
	through the formal substitution $\by^{(j)}(u)= {\bf g}^{(j)}(u)^{-1}\bt^{(j+\h)}(u) \bt^{(j-\h)}(u)$. Combining the facts that ${\cal A}_q$ is isomorphic to a central extension of the $q$-Onsager algebra \cite{BasBel,Ter21c}, and that the $q$-Onsager algebra is one of the simplest example of coideal subalgebra of $\Uqhat$ \cite{BB12}, it is thus natural to ask for an interpretation of the T-system (\ref{Tsys}) within the representation theory of the $q$-Onsager algebra. From the perspective of solvable lattice models, quantum field theory and related thermodynamic Bethe ansatz (see e.g.\ \cite{Zam,KNS0,Rav} for related Y-systems), it is also similarly expected that the Y-system (\ref{Ysys}) finds a natural interpretation for systems with boundaries.

 \subsection{Universal TQ-relations for $\cal A_q$}\label{subsec:univ-TQ-Aq}
     Another potential application of the formula (\ref{TT-rel})  or equally~\eqref{secondTT-rel}  concerns the construction of {\it universal}  TQ-relations. In \cite{YNZ06},  the  Q-operator for the spin-$\h$ open XXZ quantum spin chain is conjectured to be the  $j\rightarrow \infty$ limit of the spin-$j$ normalized transfer matrix. We may now ask about the existence of a universal Q-operator in $\cal A_q$ and identification of the corresponding universal TQ relation. Shifting the parameter $u\rightarrow uq^{ j -\h}$ in~\eqref{secondTT-rel}, respectively $u\rightarrow uq^{-j +\h}$ in~\eqref{TT-rel}, then $j\rightarrow j+\h$ and taking the limit $j\rightarrow \infty$ with the formal identification
     \beqa
     \bold{Q}_\pm(u) = \lim_{j\rightarrow \infty}  \bt^{(j)}\bigl(uq^{\pm(j+\h)}\bigr)
     \eeqa
     we obtain the \textit{universal} TQ-relations
     \beqa\label{TQj-univ}
     \bt^{(\h)}(u) \bold{Q}_\pm(u) =  \bold{Q}_\pm(uq^{\mp 1}) -\frac{\Gamma (uq^{(-1 \pm 1)/2}) \Gamma_+ (uq^{(-1\pm 1)/2})}{ c(u^2 q^{\pm 1}) c(u^2 q^{2\pm 1}) }\bold{Q}_\pm(uq^{\pm 1})   \ .
     \eeqa
Note that one can straightforwardly derive the TQ-relations associated with the $q$-Onsager algebra $O_q$ as well, by applying the algebra map $\gamma^{(\delta)}$ with~\eqref{map:AqOq} to~\eqref{TQj-univ} and using~\eqref{eq:durho}. 
Constructing corresponding universal Q-operators $\bold{Q}_\pm(u)$ satisfying these TQ-relations for $\cal A_q$ and $O_q$ is an open problem. In the special case $k_\pm=\bar k_\pm=0$ and properly normalized T- and Q-operators, we have verified that the TQ-relations for $O_q$ in the case of $\bold Q_-$ reduce to the TQ-relations given in~\cite[Eq.\,(5.9)]{T20}. 

For spin-chain representations of $\cal A_q$ studied in
Sections~\ref{sec:fusK-dresK} and~\ref{sec5}, recall that the spin-$j$ transfer matrix ${\boldsymbol  t}^{j,\bj}(u)$ in~\eqref{tjs}, where $\boldsymbol{\bj} := (j_1, \ldots, j_N)$ is any $N$-tuple of  spins at the quantum spaces, is obtained from the universal transfer matrix $\bt^{(j)}(u)$ defined by~\eqref{tg}, using~\eqref{evalTN},~\eqref{tgintN},~\eqref{KjN}.
    With the formal identification
     \beqa
     {Q}^{\bj}(u) = \lim_{j\rightarrow \infty}  {\boldsymbol  t}^{j,\bj}(uq^{-j-\h})\ ,
     \eeqa
    from the TT-relations for the transfer matrix~\eqref{normTTN} we obtain:
   \begin{align}\label{TQj}
   {\boldsymbol  t}^{\h,\bj}(u) {Q}^{\bj}(u) =   {Q}^{\bj}(uq) - \left [	
			\displaystyle{\prod_{n=1}^N} \beta^{(j_n)}(u v_n) \beta^{(j_n)}(u v_n^{-1})
			\right ]\frac{\Gamma_- (uq^{-1}) \Gamma_+ (uq^{-1})}{ c(u^2 q) c(u^2 q^{-1}) }{Q}^{\bj}(uq^{-1})\ 
	\end{align}
		with~\eqref{beta} and the quantum determinant~\eqref{gammaKM}. We can think about this TQ-relation as specialization of the universal TQ-relation~\eqref{TQj-univ} for properly normalized $\bold{Q}_-(u)$. Importantly, at the specialization (1) in Example~\ref{ex:FNR} we recover the TQ-relation~\cite[Eq.\,(3.2)]{FNR07}. 
   It would be desirable to better understand how this expression leads to the inhomogeneous TQ relations for a Baxter's Q-polynomial  considered in~\cite{CYSW13,YZYSW15}, or to construct explicit non-polynomial solutions~\cite{LP14} in terms of power series or rational functions.  

\subsection{Universal TT- and TQ-relations  for $\cal A^{aug}_q$}
\label{subsec:univ-TQ-Aq-aug}
It is known that degenerate versions of $\cal A_q$  such as the augmented $q$-Onsager \cite{IT10,BB12}, the triangular $q$-Onsager and $\Uq$-invariant $q$-Onsager algebras \cite{BB16,T19} control the corresponding degenerate boundary conditions (diagonal, triangular, and $\Uq$-invariant) on one side of the XXZ spin chain. It is therefore equally interesting to consider 
universal  TT- and TQ-relations for such degenerate versions of $\cal A_q$.
 For instance, in  view of applications to spin chains with  diagonal boundary conditions,  i.e.\ when the K-matrix on one side is diagonal, we consider a central extension of the {\it augmented} $q$-Onsager algebra $\cal A_q^{aug}$ defined in terms of generating functions~\cite{BB12}:
\begin{align}
			{\cK}_+(u)=\sum_{k\in {\mathbb N}}{\normalfont \tK}_{-k}U^{-k-1} \ , \quad {\cK}_-(u)=\sum_{k\in  {\mathbb N}}{\normalfont \tK}_{k+1}U^{-k-1} \ ,\label{caug1}\\
			\quad {\cZ}_+(u)=\sum_{k\in {\mathbb N}}{\normalfont \tZ}_{k+1}U^{-k-1} \ , \; \quad {\cZ}_-(u)=\sum_{k\in {\mathbb N}}{\normalfont \tilde{{\tZ}}_{k+1}}U^{-k-1} \ ,\label{caug2} \; 
\end{align}
and the corresponding fundamental K-operator
		\begin{equation}
			{\cal K}_{aug}^{(\frac{1}{2})}(u)=\begin{pmatrix} 
				uq \cK_+(u)-u^{-1}q^{-1}\cK_-(u) &\frac{1}{(q+q^{-1})}\cZ_+(u)\\ 
				\frac{1}{(q+q^{-1})}\cZ_-(u) & uq \cK_-(u) -u^{-1}q^{-1}\cK_+(u) 
			\end{pmatrix}  \label{K-Aqaug} \ 
		\end{equation}
        satisfying the reflection equation~\eqref{RE} with the replacement ${\cal K}^{(\frac{1}{2})}(u)\rightarrow {\cal K}_{aug}^{(\frac{1}{2})}(u)$.
	 Now, observe that the fundamental K-operator ${\cal K}_{aug}^{(\frac{1}{2})}(u)$ is the following limiting case $k_\pm\rightarrow 0$ 
     of the fundamental K-operator for $\cal A_q$  from~\eqref{K-Aq}, identifying~\cite[eqs.\,(3.9)\,\&\,(3.10)]{BB12}
\beqa\label{eq:aug-gen-ident}
&&\cK_\pm(u) := \cW_\pm(u)  \ , \qquad \cZ_\pm(u) := \frac{1}{k_\mp}\cG_\pm(u)    \ ,
\eeqa
while~\eqref{RE} becomes defining relations for $\cal A_q^{aug}$. 
Also, the limit $k_\pm\rightarrow 0$ of the quantum determinant $\Gamma(u)$ in~\eqref{gamma} can be expressed in terms of~\eqref{caug1} and~\eqref{caug2}:
\beqa
			\Gamma_{aug}(u)&=&\frac{1}{2}
			(u^2q^2-u^{-2}q^{-2}) \Big(-(q^2+q^{-2})\big(\cK_+(u)\cK_+(uq) + \cK_-(u)\cK_-(uq)\big) \label{gammaaug}\\ 
            && \qquad\qquad\qquad \qquad\qquad+(u^2q^2+u^{-2}q^{-2})\big(\cK_+(u)\cK_-(uq) + \cK_-(u)\cK_+(uq)\big) \nonumber \\
			&& \qquad\qquad\qquad\qquad \qquad - (q+q^{-1})^{-2} \big(\cZ_+(u)\cZ_-(uq) -\cZ_-(u)\cZ_+(uq)\big)\Big) \ .\nonumber
		\eeqa
Then, we derive the universal TT-relations associated with the diagonal case $k_\pm=0$ as follows. 
Define
	\begin{equation}
	\bt_{aug}^{(j)}(u) = \normalfont{\text{tr}}_{V^{(j)}}\bigl(K^{+{(j)}}(u){\cal K}_{aug}^{(j)}(u)\bigr) 
	\ ,\label{tg-aug} 
	\end{equation}
   with the dual K-matrix~\eqref{def:kp},  and the spin-$j$ K-operator ${\cal K}_{aug}^{(j)}(u)$ is the limiting case $k_\pm\rightarrow 0$ of ${\cal K}^{(j)}(u)$ using the identification in~\eqref{eq:aug-gen-ident}.
From~\eqref{TT-rel}, we then get the universal TT-relations for $\cal A_q^{aug}$:
		\begin{equation} \label{TT-relaug}
			\bt_{aug}^{(j)}(u) = \bt_{aug}^{(j-\h)}(u q^{-\h}) \bt_{aug}^{(\h)}(u q^{j-\h}) + \frac{\Gamma_{aug} (u q^{j-\tha}) \Gamma_+ (uq^{j-\tha})}{ c(u^2 q^{2j}) c(u^2 q^{2j-2}) } \bt_{aug}^{(j-1)}(u q^{-1})  \ 
		\end{equation}
        for all half-integer $j\geq1$ and
		with $\bt_{aug}^{(0)}(u) =  1$. Specialization of $\bt_{aug}^{(j)}(u)$ on spin-chain representations gives transfer matrices for spin-$j$ Hamiltonians with diagonal boundary conditions on one side of the spin chain, and general on the other side (in particular, diagonal too). We thus expect from~\eqref{TT-relaug} TT-relations for such transfer matrices with any diagonal boundary condition.

         We then turn to the corresponding TQ-relations. Taking the limit $j\rightarrow \infty$ in~\eqref{TT-relaug} with the formal identification 
     \beqa
     \bold{Q}^{aug}_-(u) = \lim_{j\rightarrow \infty}  \bt^{(j)}_{aug}\bigl(uq^{-j-\h}\bigr)
     \eeqa
     we obtain the \textit{universal} TQ-relations for $\cal A_q^{aug}$:
     \beqa\label{TQj-aug-univ}
     \bt^{(\h)}_{aug}(u) \bold{Q}^{aug}_-(u) =  \bold{Q}^{aug}_-(uq) -\frac{\Gamma_{aug} (uq^{-1}) \Gamma_+ (uq^{-1})}{ c(u^2 q) c(u^2 q^{-1}) }\bold{Q}^{aug}_-(uq^{- 1})   \ .
     \eeqa
     Now, recall the definition of $\Gamma_+(u)$ in~\eqref{gammaKP} with~\eqref{gammaKM}. At the specialization $\bar k_\pm=0$, we expect that the corresponding universal TQ-relation~\eqref{TQj-aug-univ} is a generalization of the one in~\cite[Thm.\,5.3]{VW20}, see also~\cite[Eq.\,(5.9)]{T20}, in the sense that our coefficients on the r.h.s.\ of~\eqref{TQj-aug-univ} are central elements and not just scalars.
	
\subsection{Bethe eigenstates and tridiagonal pairs}\label{sec7:Bethe}
For quantum integrable models associated with the image of (\ref{tg}) in some representation of ${\cal A}_q$, the diagonalization of the Hamiltonian and any higher order conserved quantities is usually based on the (modified or off-diagonal) Bethe ansatz approach \cite{CLSW02,FNR07,CYSW14,BeP15,YZYSW15} or separation of variables approach \cite{N12}.  
    The universal TT-relations~\eqref{TT-rel} -- from which the spin-chain TT-relations~\eqref{renormTTN} with~\eqref{trans-mat-jsIN} for the normalized transfer matrix  are derived -- change the perspective:  they relate the spectral problem for the  normalized transfer matrix of the spin chain for any $j$ given by~\eqref{tildetjs} to the spectral problem  for the $N$ mutually commuting operators $\psi^{(N)}_{\bj,\bar{v}}\bigl(\mathsf{I}_{2k+1}^{(N)}\bigr)$ given by~\eqref{IkOqN} with $0\leq k\leq N-1$.  
    
    For  the quotient of ${\cal A}_q$ known as the Askey-Wilson algebra $AW\cong \mathcal{A}_q^{(1)}$ discussed in Example~\ref{ex1},  the spectral problem for the normalized transfer matrix reduces to  diagonalization of the  operator $\psi^{(1)}_{j,v}\bigl(\mathsf{I}_{1}^{(1)}\bigr)$ on irreducible finite dimensional representations of $AW$~\cite{BP19}.   For the special choice $\overline{k}_\pm=0$ and $\overline{\varepsilon}_+=1, \overline{\varepsilon}_-=0$  (while the other side boundary parameters $k_\pm$ and ${\varepsilon}_\pm$ are generic)
    this operator takes a particularly simple form $\psi^{(1)}_{j,v}\bigl(\mathsf{I}_{1}^{(1)}\bigr) = \cal W_{0}^{(1)}$, and similarly for $\overline{\varepsilon}_-=1,\overline{\varepsilon}_+=0$ we get $\psi^{(1)}_{j,v}\bigl(\mathsf{I}_{1}^{(1)}\bigr) = \cal W_{1}^{(1)}$. The corresponding 
      Bethe eigenstates are eigenvectors of the Leonard pair $\cal W_{0}^{(1)},\cal W_{1}^{(1)}$. For generic boundary conditions $\overline{\varepsilon}_\pm, \overline{k}_\pm\neq 0$, the operator $\psi^{(1)}_{j,v}\bigl(\mathsf{I}_{1}^{(1)}\bigr)$ is a linear combination of $\cal W_{0}^{(1)}$, $\cal W_{1}^{(1)}$ and their $q$-commutators, its eigenvectors can  of course be obtained using modified algebraic Bethe ansatz, as done in~\cite{BP19}. However, such derivation relying solely on the theory of Leonard pairs or its appropriate generalization remains an open problem.
       
    Similarly for the more general quotients ${\cal A}_q^{(N)}$, with $N>1$, studied in Section~\ref{sec4}, the spectral problem for the normalized transfer matrix~\eqref{tildetjs} reduces to identifying the  common eigenvectors of $\psi^{(N)}_{\bj,\bar{v}}\bigl(\mathsf{I}_{2k+1}^{(N)}\bigr)$, with $0\leq k\leq N-1$.  While the modifed algebraic Bethe ansatz offers a way to compute these eigenvectors 
     in the form of Bethe states, their interpretation within the theory of tridiagonal pairs \cite{ITT99,INT10,NT17} (which applies to irreducible finite dimensional representations of $\cal A_q$) is an open problem. Even for the special choice  $\overline{\varepsilon}_+=1$ and $\overline{\varepsilon}_-=\overline{k}_\pm=0$ where $\psi^{(N)}_{\bj,\bar{v}}\bigl(\mathsf{I}_{2k+1}^{(N)}\bigr) = \cal W_{-k}^{(N)}$, common eigenvectors of $\cal W_{-k}^{(N)}$ for $0\leq k\leq N-1$ have not been studied yet.

\subsection{Relation to the blob algebra}\label{sec7:blob}
We notice that in the spin-$\h$ case, the  $q-$Onsager operators ${\cal W}_0^{(N)}$ and ${\cal W}_1^{(N)}$ commute not just with the Hamiltonians~\eqref{eq:Ham-sym-1}  and~\eqref{eq:Ham-sym-2}, as in Proposition~\ref{prop:HXXZ-comm-cond}, but equally with their Hamiltonian densities, both bulk and boundary interaction terms. For example, $\mathcal{H}_{XXZ}^{\textsf{spin $\h$} \, -}(h_\pm, h_z)$ can be written in terms of densities: 
\begin{equation}\label{eq:H-dens}
    \mathcal{H}_{XXZ}^{\textsf{spin $\h$} \, -}(h_\pm, h_z) 
= \mu b - 2 \sum_{i=1}^{N-1} e_i
\end{equation}
where $b$ is the boundary term at site $1$ and the bulk densities
	\begin{equation*}
	e_i = - ( \sigma_i^+ \sigma^-_{i+1} + \sigma^-_i \sigma^+_{i+1} + \frac{q+q^{-1}}{4}\sigma^z_i\sigma^z_{i+1}) +\frac{q-q^{-1}}{4} (\sigma^z_{i+1}-\sigma^z_i).
	\end{equation*} 
All together, $b$ and $e_i$'s form a representation of the blob algebra~\cite{blob} which is an extension of the famous Temperley-Lieb algebra with extra relations (for appropriate choice of the boundary coupling $\mu$  as a function of the boundary parameters $h_\pm$ and $h_z$)
	\begin{equation}
	b^2=b\ , \qquad e_1b e_1 = y e_1\ , \qquad e_j b = b e_j\ ,\qquad   \text{for} \quad j\geq 2,
	\end{equation}
and $y$ is also  a function of the boundary parameters $h_\pm$, $h_z$, see for details~\cite{dGN09}. 

It's well-known~\cite{PS90} that all the bulk densities $e_i$'s commute with the $\Uq$ action. Furthermore, $\calW_0^{(N)}$ belongs to the image of the spin-chain representation of $\Uq$. Indeed, \eqref{W0vN} for $v_n=1$, for all $1\leq n\leq N$, and $j=\h$ is the action of $k_+ q^{\h}E K^{\h} + k_-q^{-\h} FK^{\h} + \varepsilon_+ K$ with \eqref{eq:shUq} and the coproduct convention for $\Uq$ as in~\cite[Sec.\,2.4]{LBG}.  Therefore, $\calW_0^{(N)}$ commutes with $e_i$, for all $1\leq i \leq N-1$. And together with~\eqref{eq:commutrel} and~\eqref{eq:H-dens}, we conclude that   $\calW_0^{(N)}$ commutes with $b$ as well.  Actually, for generic values of $y$  the blob algebra is semi-simple~\cite{blob} and each eigenspace of $\calW_0^{(N)}$ is  irreducible with respect to the blob algebra action. We thus see that $\calW_0^{(N)}$ centralizes the blob algebra action, i.e.\ that every operator commuting with the densities $b$ and $e_i$'s is necessarily a polynomial in $\calW_0^{(N)}$.
It is in  this sense that
$\calW_0^{(N)}$ forms a non-trivial symmetry of the Hamiltonian $\mathcal{H}_{XXZ}^{\textsf{spin $\h$} \, -}(h_\pm, h_z)$.  The situation with $\calW_1^{(N)}$ is very similar, one should replace $\Uq$ by $U_{q^{-1}sl_2}$ in the analysis.
It is an interesting problem to extend this picture and interpretation of the property~\eqref{eq:commutrel} to all spins $j>\h$ using an appropriate (yet, non-existing) higher spin version of the blob algebra.

\smallskip
 Another interesting problem is to extend these symmetry results for the cases of $q$ a root of unity, possibly to non-abelian symmetries, as those in the closed case given by the divided powers of the  generators $S_\pm$ of $\Uq$ in~\cite{PS90,DFMC01} but instead using the divided polynomials~\cite{BGV16} in $\calW_{0/1}^{(N)}$.
 Indeed, the exchange relations in Proposition~\ref{corol1} are interesting analogues of the relations~\cite[Eqs.\,(2.62a)]{PS90}, and together with Remark~\ref{rem:direct-calc}, they equally hold for spin-$\h$ and $q$ any root of unity case.

\bigskip
	
	\noindent{\bf Acknowledgments:}  Authors thank A.\ Kuniba, G.\ Niccoli, R. Pimenta, T.\ Prosen, J.\ Sirker, J.\ Suzuki and W.-L.\ Yang for communications. P.B.\ and A.M.G.\  are supported by C.N.R.S. The work of A.M.G. was also partially supported by
the ANR grant NASQI3D ANR-24-CE40-7252. 
	\vspace{0.2cm}
	
	\begin{appendix}

		\section{The maps \texorpdfstring{${\mathcal{E}}^{(j)}$}{E(j)}, \texorpdfstring{${\mathcal{F}}^{(j)}$}{F(j)} and \texorpdfstring{$\bar{\mathcal{E}}^{(j-\h)}$}{bE(j-1/2)}, \texorpdfstring{$\bar{\mathcal{F}}^{(j-\h)}$}{bF(j-1/2)}}\label{Ap:EF} 
		Most of the material in this Appendix is taken from \cite[Sec.\,3]{LBG}. The reader is referred to this paper for details and proofs of the results below.  
		\smallskip
		
		Let	$E_{a,b}^{(j_1,j_2)}$ denote the matrix of dimension $(2j_1+2) \times (2j_2+2)$ with $1$ at position $(a,b)$ and $0$ otherwise.

        \subsection{The maps \texorpdfstring{$\mathcal{E}^{(j)}$}{E(j)} and \texorpdfstring{$\mathcal{F}^{(j)}$}{F(j)}}
The linear maps $\mathcal{E}^{(j)}$, $\mathcal{F}^{(j)}$  
are defined as
\begin{align}
\mathcal{E}^{(j)}\colon & \mathbb{C}_u^{2j+1} \rightarrow \mathbb{C}_{u_1}^{2}
\otimes \mathbb{C}_{u_2}^{2j} \ ,\qquad \mathcal{E}^{(j)}=\displaystyle{\sum_{a=1}^{4j}\sum_{b=1}^{2j+1} } \mathcal{E}_{a,b}^{(j)} E_{a,b}^{(2j-1,j-\h)}
\ ,\label{exprE}
\\
\mathcal{F}^{(j)} \colon  &\mathbb{C}^{2}_{u_1} \otimes \mathbb{C}_{u_2}^{2j} \rightarrow \mathbb{C}_u^{2j+1}\ ,\qquad \mathcal{F}^{(j)}=\displaystyle{\sum_{a=1}^{2j+1}\sum_{b=1}^{4j} } \mathcal{F}_{a,b}^{(j)} E_{a,b}^{(j-\h,2j-1)}\ , \label{exprF}
\end{align}
with   $u_1 = u q^{-j+\h}$, $u_2=u q^\h$, and the matrix entries $\mathcal{E}^{(j)}_{a,b}$ are given by
\begin{equation}
\mathcal{E}^{(j)}_{1,1}=1 \ , \qquad \mathcal{E}^{(j)}_{1+n,1+n}= \displaystyle{\prod_{p=0}^{n-1}} \frac{ B_{j-\h,j-p-\h}}{B_{j,j-p}} \ , \qquad \mathcal{E}^{(j)}_{2j+m,1+m}=[m]_q \frac{\mathcal{E}^{(j)}_{m,m}}{B_{j,j+1-m}}\ ,\label{Ep1}
\end{equation}
where  $n=1$, $2$, $\ldots$, $2j-1$, $m =1$, $2$, $\ldots$, $2j$  and $B_{j,j'}$ is given in~\eqref{Bdef}, and all the other entries are zero.
Here, we choose the basis of the source space of the map $\mathcal{E}^{(j)}$ to be $\{ \ket{j,m} \}$ with the order $m=j$, $ j-1$, $\ldots$, $-j+\h$, $-j$, while for its target  \allowbreak
$\{\ket{\uparrow} \otimes \ket{ j-\h,j-\h}$, $\ldots$, $\ket{\uparrow} \otimes \ket{j-\h,-j+\h}$, $\ket{\downarrow} \otimes \ket{j-\h,j-\h}$, $\ldots$, $\ket{\downarrow} \otimes \ket{j-\h,-j+\h} \}$.
 The conditions on the evaluation parameters are fixed such that the tensor product of evaluation representations of $\Loop$ gives a spin-$j$ sub-representation, 
and $\mathcal{E}^{(j)}$ becomes a $\Loop$-intertwiner.
For the map $\mathcal{F}^{(j)}$, the entries are given for any $j \in \h \mathbb{N}_+$ by:
\begin{align} \label{Fp1}
&\mathcal{F}_{1,1}^{(j)}=1 \ ,\qquad  \mathcal{F}_{n,n+2j-1}^{(j)}=\frac{ \mathcal{E}_{n+2j-1,n}^{(j)}}{(\mathcal{E}_{n,n}^{(j)})^2+(\mathcal{E}_{n+2j-1,n}^{(j)})^2} \ ,\\
\mathcal{F}_{2j+1,4j}^{(j)} &= (\mathcal{E}_{4j,2j+1}^{(j)})^{-1} \ , \qquad \mathcal{F}_{n,n}^{(j)}=\frac{1-\mathcal{F}_{n,n+2j-1}^{(j)}\mathcal{E}_{n+2j-1,n}^{(j)}}{\mathcal{E}^{(j)}_{n,n}}   \ ,  \label{Fp2}
\end{align}
where $n=2$, $3$, $\ldots$, $2j$, and all other entries are zero.
Note that it is a pseudo-inverse of $\mathcal{E}^{(j)}$
\begin{equation} \label{eq:FE1}
\mathcal{F}^{(j)}\mathcal{E}^{(j)}  = {\mathbb I}_{2j+1}\ .
\end{equation}

In the analysis, the following invertible diagonal matrix $\mathcal{H}^{(j)}$  will be useful.
It is given by:
\begin{equation}
\label{defH}
\mathcal{H}^{(j)}=\diag(\mathcal{H}_{1}^{(j)},\mathcal{H}_{2}^{(j)},\ldots, \mathcal{H}_{2j+1}^{(j)}) \ ,
\end{equation}
where the entries $\mathcal{H}_{n}^{(j)}$ are determined by the relation:
\begin{equation} \label{decompR}
R^{(\h,j-\h)}(q^{j})= \mathcal{E}^{(j)}\mathcal{H}^{(j)}\mathcal{F}^{(j)} \ ,
\end{equation}
with (\ref{R-Rqg}). They read: 
\vspace{-0.3cm}
\begin{equation}
\begin{aligned}\label{coefp2}
\mathcal{H}_{1}^{(j)}&=
\mathcal{H}_{2j+1}^{(j)}=
\prod_{k=2}^{2j} c(q^{k})\ ,\\
\mathcal{H}_{n}^{(j)} &=  
\frac{B_{j-\h,-j-\h+n}} {\mathcal{E}_{n,n}^{(j)} \mathcal{F}_{n,n+2j-1}^{(j)}}
\prod_{k=1}^{2j-1} c(q^{k})\ ,
\end{aligned}
\end{equation}
for $n=2$, $3$, $\ldots$, $2j$.
Besides, from the decomposition~\eqref{decompR} and~\eqref{eq:FE1}, it follows:
\begin{equation}
\begin{aligned}
\mathcal{E}^{(j)}\mathcal{H}^{(j)}&=R^{(\h,j-\h)}(q^{j})\mathcal{E}^{(j)} \ ,\\ \label{usefulEFH} \mathcal{H}^{(j)}\mathcal{F}^{(j)}&=\mathcal{F}^{(j)} R^{(\h,j-\h)}(q^{j}) \ ,\\
R^{(\h,j-\h)}(q^{j})&= \mathcal{E}^{(j)} \mathcal{F}^{(j)}R^{(\h,j-\h)}(q^{j}) \ .
\end{aligned}
\end{equation}
		\subsection{The maps \texorpdfstring{$\bar{\mathcal{E}}^{(j-\h)}$}{bE(j-1/2)} and \texorpdfstring{$\bar{\mathcal{F}}^{(j-\h)}$}{bF(j-1/2)}} \label{appA:def-Ebar}
		The linear maps $\bar{\mathcal{E}}^{(j-\h)}$ and $\bar{\mathcal{F}}^{(j-\h)}$ 
		are defined as
		\begin{align} 
			\bar{\mathcal{E}}^{(j-\h)}\colon & \mathbb{C}^{2j}_u \rightarrow \mathbb{C}_{u_1}^{2} 
			\otimes \mathbb{C}_{u_2}^{2j+1} \ , \qquad \bar{\mathcal{E}}^{(j-\h)}=\displaystyle{\sum_{a=1}^{4j+2}\sum_{b=1}^{2j}}  \bar{\mathcal{E}}_{a,b}^{(j-\h)} E^{(2j,j-1)}_{a,b}\ , \label{exprbarE} 
			\\ 
			\bar{\mathcal{F}}^{(j-\h)}\colon &\mathbb{C}_{u_1}^{2} \otimes \mathbb{C}_{u_2}^{2j+1} \rightarrow \mathbb{C}^{2j}_u \ , \qquad \bar{\mathcal{F}}^{(j-\h)}=\displaystyle{\sum_{a=1}^{2j}} \sum_{b=1}^{4j+2}  \bar{\mathcal{F}}_{a,b}^{(j-\h)} E^{(j-1,2j)}_{a,b} \ ,
		\end{align}
		where $\bar{\mathcal{E}}^{(j-\h)}_{a,b}$, $\bar{\mathcal{F}}^{(j-\h)}_{a,b}$ are certain scalars  and with $u_1/u_2= q^{j+\h}$,  $u_2=u q^\h$. These conditions are fixed such that we get a spin-$(j-\h)$ sub-representation from the tensor product of evaluation representations of $\Loop$.
		Also, the map $\bar{\mathcal{E}}^{(j-\h)}$ is a $\Loop$-intertwiner and its  entries are given for any $j \in \frac{1}{2} \mathbb{N}_+$ by
		\begin{equation} \label{barE-proj}
			\bar{\mathcal{E}}_{2,1}^{(j-\h)}=1 \ , \qquad
			\bar{\mathcal{E}}^{(j-\h)}_{2+n,1+n}= \displaystyle{\prod_{p=0}^{n-1}} \frac{ B_{j,j-p-1}}{B_{j-\h,j-\h-p}} \  \ ,\qquad \bar{\mathcal{E}}^{(j-\h)}_{2j+2+m,1+m}=\frac{[m-2j]_q}{B_{j,j-m}} 
			\bar{\mathcal{E}}_{2+m,1+m}^{(j-\h)} \ ,
		\end{equation}
		where  $n=1$, $2$, $\ldots$, $2j-1$, $m=0$, $1$, $\ldots$, $2j-1$ and $B_{j,j'}$ is given in~\eqref{Bdef}, and all the other entries are zero.
		For the map $\bar{\mathcal{F}}^{(j-\h)}$, the entries are given for any $j \in \h \mathbb{N}_+$ by:
		\begin{equation}\label{coefbar} 
			\bar{\mathcal{F}}_{n,n+2j+1}^{(j-\h)} = \frac {\bar{\mathcal{E}}_{n+2j+1,n}^{(j-\h)}}{(\bar{\mathcal{E}}_{n+1,n}^{(j-\h)})^2 + (\bar{\mathcal{E}}_{n+2j+1,n}^{(j-\h)})^2 } \ , \qquad
			\bar{\mathcal{F}}_{n,n+1}^{(j-\h)} = \frac { 1 - \bar{\mathcal{F}}_{n,n+2j+1}^{(j-\h)} \bar{\mathcal{E}}_{n+2j+1,n}^{(j-\h)} } { \bar{\mathcal{E}}_{n+1,n}^{(j-\h)}} \ ,
		\end{equation}
		where $n=1$, $2$, $\ldots$, $2j$, and all other entries are zero.
		Note that it is a pseudo-inverse of $\bar{\mathcal{E}}^{(j-\h)}$
		\begin{equation} \label{eq:bFE1}
		\bar{\mathcal{F}}^{(j-\h)} \bar{\mathcal{E}}^{(j-\h)} = {\mathbb I}_{2j}\ .
		\end{equation}

		The following invertible diagonal matrix $\bar{\mathcal{H}}^{(j-\h)}$  will be also useful.
		It is given by:
		\begin{equation}
			\label{defbarH}
			\bar{\mathcal{H}}^{(j-\h)}=\diag(\bar{\mathcal{H}}_{1}^{(j-\h)},\bar{\mathcal{H}}_{2}^{(j-\h)},\ldots, \bar{\mathcal{H}}_{2j}^{(j-\h)}) \ ,
		\end{equation}
		where the entries $\bar{\mathcal{H}}_{n}^{(j-\h)}$ are determined by the relation:
		\begin{equation} \label{decompbarR}
			{R}^{(\h,j)}(q^{-j-\h})= \bar{\mathcal{E}}^{(j-\h)} \bar{\mathcal{H}}^{(j-\h)} \bar{\mathcal{F}}^{(j-\h)} \ .
		\end{equation}
		They read:
		\begin{equation}
			\bar{\mathcal{H}}_{n}^{(j-\h)}= \left ( \displaystyle{\prod_{k=0}^{2j-2}} c(q^{-k-1}) \right) \frac {(q-q^{-1})  B_{j,-j+n} }{ \bar{\mathcal{E}}^{(j-\h)}_{n+2j+1,n} \bar{\mathcal{F}}_{n,n+1}^{(j-\h)}} \ ,
		\end{equation}
		for $n=1$, $2$, $\ldots$, $2j$.
		Besides, from the decomposition~\eqref{decompbarR} and~\eqref{eq:bFE1}, it follows:
		\begin{equation}
			\begin{aligned}
				\bar{\mathcal{E}}^{(j-\h)}\bar{\mathcal{H}}^{(j-\h)}&=R^{(\h,j)}(q^{-j-\h})\bar{\mathcal{E}}^{(j-\h)} \ ,\\ \label{usefulbarEFH} \bar{\mathcal{H}}^{(j-\h)}\bar{\mathcal{F}}^{(j-\h)}&=\bar{\mathcal{F}}^{(j-\h)}	R^{(\h,j)}(q^{-j-\h}) \ ,\\ 
				R^{(\h,j)}(q^{-j-\h})&=	\bar{\mathcal{E}}^{(j-\h)} \bar{\mathcal{F}}^{(j-\h)}R^{(\h,j)}(q^{-j-\h}) \ .
			\end{aligned}
		\end{equation}
		\subsection{Examples} {\small{
		\begin{equation} \label{eq:EHF1}
			\mathcal{E}^{(1)}=
			\begin{pmatrix}
				1&  0&  0\\ 
				0&  \frac{1}{\sqrt{ [2]_q}}&  0\\ 
				0 &  \frac{1}{\sqrt{ [2]_q}}&  0\\ 
				0&  0&  1
			\end{pmatrix} , \quad \mathcal{H}^{(1)}= 
			\begin{pmatrix}
				q^2-q^{-2}&  0&  0\\ 
				0&  2(q-q^{-1})&  0\\ 
				0 &  0&  q^2-q^{-2}
			\end{pmatrix} , \quad \mathcal{F}^{(1)}=
			\begin{pmatrix}
				1 & 0 & 0 & 0\\ 
				0&  \frac{\sqrt{ [2]_q}}{2} &  \frac{\sqrt{[2]_q}}{2}& 0\\ 
				0& 0 & 0 &1 
			\end{pmatrix}  ,
		\end{equation}
		\begin{equation*} 	
			\mathcal{E}^{(\tha)} = 
	\begin{pmatrix}
		1 &  0                                        & 0 &0 \\ 
		0&  \sqrt{\frac{[2]_q}{[3]_q}}       & 0 &0 \\ 
		0 &  0                                        & \frac{1}{ \sqrt{[3]_q}} &0 \\ 
		0 & \frac{1}{\sqrt{[3]_q}}                   & 0 &0 \\ 
		0 & 0                                         & \sqrt{\frac{[2]_q}{[3]_q}}&0 \\ 
		0 & 0                                         & 0 &1 
	\end{pmatrix}  ,\quad
			\mathcal{F}^{(\tha)}=  
	\begin{pmatrix}
		1 & 0 & 0 & 0 & 0 &0 \\ 
		0 & \frac{ \sqrt{[2]_q[3]_q}}{1+[2]_q} & 0 & \frac{\sqrt{[3]_q}}{1+[2]_q} & 0 &0 \\ 
		0&0  &\frac{\sqrt{[3]_q}}{1+[2]_q}  &0  &
		\frac{ \sqrt{[2]_q[3]_q}}{1+[2]_q} &0 \\ 
		0 & 0 &0  &0  & 0 & 1
	\end{pmatrix},
		\end{equation*}
		\begin{equation*}	
			\mathcal{H}^{(\tha)}= \begin{pmatrix}
				q^5-q-q^{-1}+q^{-5}&  0&  0&0\\ 
				0&  (q - q^{-1})^2 \lbrack 2 \rbrack_q (1+\lbrack 2 \rbrack_q) &  0&0\\ 
				0 &  0&   (q - q^{-1})^2 \lbrack 2 \rbrack_q (1+\lbrack 2 \rbrack_q)  &0 \\
				0 & 0 & 0 & q^5-q-q^{-1}+q^{-5}
			\end{pmatrix}  ,
		\end{equation*}
		\begin{equation*}
			\bar{\mathcal{E}}^{(0)} =
			\begin{pmatrix}
				0 \\
				1\\
				-1\\
				0
			\end{pmatrix} ,\qquad  \bar{\mathcal{H}}^{(0)} = 2(q^{-1}-q), \qquad \bar{\mathcal{F}}^{(0)}= 
			\begin{pmatrix}
				0 & \frac{1}{2} & -\frac{1}{2}  & 0  \\
			\end{pmatrix} \ .
		\end{equation*}
		\begin{equation*} 
			\bar{\mathcal{E}}^{(\h)}= \begin{pmatrix}
				0 & 0 \\ 
				1 & 0\\ 
				0 & \sqrt{[2]_q}\\ 
				-\sqrt{[2]_q} &0 \\ 
				0& -1 \\ 
				0 &0 
			\end{pmatrix} ,\quad  \bar{\mathcal{H}}^{(\h)} =(q-q^{-1})^2(1+\lbrack 2 \rbrack_q) \, {\mathbb I}_2, \quad \bar{\mathcal{F}}^{(\h)}= \begin{pmatrix}
				0& \frac{1}{1+[2]_q}  & 0 & - \frac{\sqrt{[2]_q}}{1+[2]_q}  &0  &0 \\ 
				0&0  &\frac{\sqrt{[2]_q}}{1+[2]_q}   &0  &-\frac{1}{1+[2]_q}  &0 
			\end{pmatrix} .
		\end{equation*}
		}}
		\normalsize

		\subsection{Other relations} \label{Ap:propEHF}
		For the proof of the TT-relations, few relations satisfied by the above maps are needed.  
		\begin{lem} \label{lem:rel}
			The following relations hold: 
			\begin{align} \label{rel1}
				\mathcal{H}_{1}^{(j)} \ 	\mathcal{E}^{(j)} = 	\big[ \mathcal{F}^{(j)} \big ]^t \mathcal{H}^{(j)}\ , \qquad  \bar{\mathcal{E}}^{(j-\h)} \bar{\mathcal{F}}^{(j-\h)} + \mathcal{E}^{(j+\h)} \mathcal{F}^{(j+\h)} = \mathbb{I}_{4j+2}\ , \\ \label{rel2}
				\big [\mathcal{F}^{(j)} \big ]^t  \mathcal{F}^{(j)} R^{(\h,j-\h)}(q^{j}) = \mathcal{H}_{1}^{(j)} \  \mathcal{E}^{(j)} \mathcal{F}^{(j)} \ , \qquad  \mathcal{H}_{1}^{(j)} \ \mathcal{E}^{(j)}  \big [ \mathcal{E}^{(j)} \big ]^t = R^{(\h,j-\h)}(q^j) \ ,
			\end{align}
            with the scalar $\mathcal{H}_{1}^{(j)} $ from~\eqref{coefp2}.
		\end{lem}
		\begin{proof}
			The relations in~\eqref{rel1} are straightforwardly checked using the expressions of $\mathcal{E}^{(j+\h)}$, $\mathcal{F}^{(j+\h)}$, $\mathcal{H}^{(j+\h)}$ and $\bar{\mathcal{E}}^{(j-\h)}$, $\bar{\mathcal{F}}^{(j-\h)}$, $\bar{\mathcal{H}}^{(j-\h)}$ given in previous subsections. Recall the decomposition of the R-matrix~\eqref{R-Rqg} at special points (\ref{decompR}), (\ref{decompbarR}).
			Since $R^{(\h,j)}(u)$ is symmetric~\eqref{symRj}, we have:
			\begin{equation} \label{tEHF}
				R^{(\h,j)}(q^{j+\h}) =  \big [ \mathcal{F}^{(j+\h)} \big ]^t \mathcal{H}^{(j+\h)} \big [ \mathcal{E}^{(j+\h)}  \big ]^t \ , \quad R^{(\h,j)}(q^{-j-\h}) = \big [ \bar{\mathcal{F}}^{(j-\h)} \big ]^t \bar{\mathcal{H}}^{(j-\h)}  \big [ \bar{\mathcal{E}}^{(j-\h)} \big ]^t \ .
			\end{equation}	 
			For the relations in~\eqref{rel2}, the first one is obtained as follows:
			\begin{equation}
				\big [\mathcal{F}^{(j)} \big ]^t  \mathcal{F}^{(j)} R^{(\h,j-\h)}(q^{j}) =  \big [\mathcal{F}^{(j)} \big ]^t  \mathcal{F}^{(j)} \mathcal{E}^{(j)} \mathcal{H}^{(j)} \mathcal{F}^{(j)}  = \big [\mathcal{F}^{(j)} \big ]^t  \mathcal{H}^{(j)}  \mathcal{F}^{(j)} = \mathcal{H}_{1}^{(j)} \ \mathcal{E}^{(j)}  \big [ \mathcal{E}^{(j)} \big ]^t \ ,
			\end{equation}
			where we used~\eqref{eq:FE1} and~\eqref{rel1}. The second relation in~\eqref{rel2} is obtained from~\eqref{tEHF}, using~\eqref{rel1}.
		\end{proof}

		Similarly, for $\bar{\mathcal{E}}^{(j)}$, $\bar{\mathcal{F}}^{(j)}$ and $\bar{\mathcal{H}}^{(j)}$, one has:
		\begin{lem} The following relations hold: 
			\begin{align} \label{eq:rel2p1}
				\bar{\mathcal{F}}_{1,2}^{(j)} \bar{\mathcal{H}}_1^{(j)} \bar{\mathcal{E}}^{(j)} = \lbrack \bar{\mathcal{F}}^{(j)} \rbrack^t \bar{\mathcal{H}}^{(j)} \ , \qquad  \bar{\mathcal{F}}^{(j-1)}_{1,2} \bar{\mathcal{H}}_{1}^{(j-1)} \bar{\mathcal{E}}^{(j-1)} \lbrack \bar{\mathcal{E}}^{(j-1)} \rbrack^t = R^{(\h,j-\h)}(q^{-j}) \ ,  \\
				\lbrack \bar{\mathcal{F}}^{(j-1)} \rbrack^t \bar{\mathcal{F}}^{(j-1)} R^{(\frac 12, j-\frac 12)}(q^{-j}) = \bar{\mathcal{F}}_{1,2}^{(j-1)} \bar{\mathcal{H}}_{1}^{(j-1)} \ \bar{\mathcal{E}}^{(j-1)} \bar{\mathcal{F}}^{(j-1)} \ ,
			\end{align}
            with the scalars $\bar{\mathcal{F}}_{1,2}^{(j)}$ and $\bar{\mathcal{H}}_1^{(j)}$ introduced in~\eqref{coefbar}  and~\eqref{coefp2}, respectively. We note that $\bar{\mathcal{F}}_{1,2}^{(j)} \bar{\mathcal{H}}_1^{(j)} = (q-q^{-1})^2\prod_{k=2}^{2j} c(q^{-k})$.
		\end{lem}
		\begin{proof}
			The first equation in~\eqref{eq:rel2p1} is proven directly. The other relations are obtained as in Lemma~\ref{lem:rel}.
		\end{proof}
		
		\section{Spin-$1$ K-operator}	\label{Ap:ex-K-op}
			The spin-$1$ fused K-operator for ${\cal A}_q$ is computed using (\ref{fused-K-op}) with (\ref{K-Aq}). The entries $({\cal K}^{(1)}(u))_{mn}$ are given by:
			\begin{align*}
				\mathcal{K}^{(1)}_{11}(u)&= \big ( c(q)^{-1} +\rho^{-1} \tG_+(u q^{-\h})  \big ) \big ( \rho + c(q)  \tG_-(u q^\h)  \big )    \\
				&+c(u^2q) \big ( u q^\h \tW_+(u q^{-\h}) - u^{-1} q^{-\h} \tW_-(u q^{-\h})   \big ) \big ( u q^{\frac 32} \tW_+(u q^\h) - u^{-1} q^{-\frac 32} \tW_-(u q^\h) \big ), \\
				\mathcal{K}^{(1)}_{12}(u)&= \frac{(q+q^{-1})^{-\frac 32}}{k_-} \Big( c(u^2) \big ( \rho c(q)^{-1} + \tG_+(u q^{-\h}) \big ) \big (u q^{\frac 32} \tW_+(uq^\h) - u^{-1} q^{-\frac 32} \tW_-(u q^\h) \big ) \\
				 &+ \big ( \rho + c(q) \tG_+(u q^{-\h}) \big ) \big ( u q^{\frac 32} \tW_-(u q^\h) - u^{-1} q^{-\frac 32} \tW_-(u q^\h) \big ) \\
				 & + c(u^2 q) \big ( u q^\h \tW_+(u q^{-\h}) - u^{-1} q^{-\h} \tW_-(u q^{-\h}) \big ) \big ( \rho c(q)^{-1} + \tG_+(u q^\h) \big ) \Big ), \\  
				\mathcal{K}^{(1)}_{13}(u)&=  \frac{c(u^2)}{k_-^2c(q^2)^2} \big ( \rho + c(q) \tG_+(uq^{-\h}) \big )  \big ( \rho + c(q) \tG_+(u q^\h)  \big ) ,    \\   \numberthis \label{expK-spin1} 
				\mathcal{K}^{(1)}_{21}(u)&= \frac{c(q)^{-1}}{2k_+\sqrt{q+q^{-1}}} \Big (  c(u^2q) \big ( \rho + c(q) \tG_-(u q^{-\h}) \big ) \big ( u q^{\frac 32} \tW_+(u q^\h ) - u^{-1} q^{-\frac 32} \tW_-(u q^\h)  \big )  \\
				&+\big ( q^{-\h} ( u^{-3} + u(-2+q^2)) \tW_-(u q^{-\h}) + q^{\h} (u^3 + u^{-1}(-2+q^{-2})) \tW_+(u q^{-\h}) \big ) \Big ) \\
				& \times \big ( \rho + c(q) \tG_-(u q^\h)    \big ), \\
				\mathcal{K}^{(1)}_{22}(u)&= \frac{c(u^2q)}{2c(q)^2 \rho} \Big (  \big (\rho + c(q) \tG_+(u q^{-\h}) \big ) \big (  \rho 
				 + c(q) \tG_-(u q^\h) \big) + \big ( \rho + c(q) \tG_-(u q^{-\h}) \big ) \big ( \rho  + c(q) \tG_+(u q^\h) \big) \Big ) \\
				 &+ \frac{1}{2} \big( q^{-\h} (u^{-3} +u (-2+q^2)) \tW_+(u q^{-\h}) + q^{\h} (u^{-1} (-2+q^2) +u^3) \tW_-(u q^{-\h})  \big ) \\
				 &\times \big ( u q^{\frac 32} \tW_+(u q^\h) - u^{-1} q^{-\frac 32} \tW_-(u q^\h) \big) \\
				 & + \frac{1}{2} \big( q^{-\h} (u^{-3} +u (-2+q^2)) \tW_-(u q^{-\h}) + q^{\h} (u^{-1} (-2+q^2) +u^3) \tW_+(u q^{-\h})  \big ) 
				 \\
				 &\times \big ( u q^{\frac 32} \tW_-(u q^\h) - u^{-1} q^{-\frac 32} \tW_+(u q^\h) \big) , \\
				 {\cal K}^{(1)}_{23}(u)&=\sigma({\cal K}^{(1)}_{21}(u))
				 \	, \qquad {\cal K}^{(1)}_{31}(u)=\sigma({\cal K}^{(1)}_{13}(u)) 
				 \	,  \\
				 {\cal  K}^{(1)}_{32}(u)&=
				 \sigma({\cal  K}^{(1)}_{12}(u))
				 \	, \qquad {\cal  K}^{(1)}_{33}(u)=\sigma({\cal  K}^{(1)}_{11}(u)) \ ,
			 \end{align*}
		 \noindent
	 where $\sigma$ is  defined by \cite{BS09}:
\begin{equation}
\sigma\colon \qquad \cW_\pm(u) \to \cW_\mp(u), \quad \cG_\pm(u) \to \cG_\mp(u), \quad k_\pm \to k_\mp \ .\label{sigma}
\end{equation}
\section{Proof of Theorem \ref{TTrel}}\label{Ap:proofT34}
The following Lemmas and some properties of the intertwining operator $\mathcal{E}^{(j)}$ and its pseudo-inverse $\mathcal{F}^{(j)}$, reported in Appendix \ref{Ap:EF}, will be used for the proof of the  universal  TT-relations.
Recall first the definition of the operators $\bar{\mathcal{E}}^{(j)}$ and $\bar{\mathcal{F}}^{(j)}$ in Section~\ref{appA:def-Ebar}.
\begin{lem}\label{lem:TT}
	Assume Conjecture~\ref{conj1}, then the following relations hold:
	\begin{align} \nonumber
		&\bar{\mathcal{F}}^{(j-\h)}_\fu \mathcal{K}_1^{(\h)}(u q^{j+1}) R^{(\h,j)}(u^2 q^{j+\tha}) \mathcal{K}_2^{(j)}(u q^{\h}) \bar{\mathcal{E}}_\fu^{(j-\h)} \\ \label{TT1}
		&\qquad= \left ( \displaystyle{\prod_{k=0}^{2j-2} } c(u^2 q^{2j-1-k}) c(u^2q^{2j+1-k}) \right ) \Gamma(u q^j) \mathcal{K}^{(j-\h)}(u) \ ,\\
		&\bar{\mathcal{F}}^{(j-\h)}_\fu K_2^{+(j)}(u q^{-\h}) R^{(\h,j)}(u^{-2} q^{-j-\tha}) K_1^{+(\h)}(u q^{j}) \bar{\mathcal{E}}_\fu^{(j-\h)}\nonumber
        \\ \label{TT3} 
        &\qquad = \frac{ f^{(j-\h)}(u q^{-1})}{f^{(j)}(u q^{-\h})}\left ( \displaystyle{\prod_{k=0}^{2j-2} } c(u^2 q^{2j-k}) c(u^2q^{2j-2-k}) \right ) \Gamma_{+}(u q^{j-1}) K^{+(j-\h)}(u q^{-1}) \ ,
	\end{align}
	where $f^{(j)}(u)$ is given in~\eqref{eq:fj}, $\Gamma(u)$ and $\Gamma_+(u)$ are defined in~\eqref{gammaform} and~\eqref{gammaKP}, respectively.
\end{lem}
\begin{proof} The equation~\eqref{TT1} is proven in~\cite[Prop.\,6.11 \& Eq.\,(6.28)]{LBG}  under the assumption that Conjecture~\ref{conj1} holds. The equation~\eqref{TT3} is proven as follows. By~\cite[Eq.\,(6.29)]{LBG}, we also have:
	\begin{align} \nonumber
		&\bar{\mathcal{F}}^{(j-\h)}_\fu \mathcal{K}_2^{(j)}(u q^{-\h}) R^{(\h,j)}(u^2 q^{-j-\tha}) \mathcal{K}_1^{(\h)}(u q^{-j-1}) \bar{\mathcal{E}}_\fu^{(j-\h)} \\ \label{TT2}
		& = \left ( \displaystyle{\prod_{k=0}^{2j-2} } c(u^2 q^{-2j+2+k}) c(u^2q^{-2j+k}) \right ) \Gamma(u q^{-j-1}) \mathcal{K}^{(j-\h)}(u) \ .
	\end{align}
	Recall $\epsilon$ from Proposition~\ref{prop:eps}.
	Applying $(\epsilon \otimes \id)$ to the above equation using~\eqref{eq:Kaqj-eps} and~\eqref{eq:eps-gam}, one gets
		\begin{align}  \nonumber
			&\bar{\mathcal{F}}^{(j-\h)}_\fu K_2^{(j)}(u q^{-\h}) R^{(\h,j)}(u^2 q^{-j-\tha}) K_1^{(\h)}(u q^{-j-1}) \bar{\mathcal{E}}_\fu^{(j-\h)} \\  \label{eq:reduc-gam}
			& = \left ( \displaystyle{\prod_{k=0}^{2j-2} } c(u^2 q^{-2j+2+k}) c(u^2q^{-2j+k}) \right ) \Gamma_-(u q^{-j-1}) K^{(j-\h)}(u) \ ,
		\end{align}
		where we used that $\varsigma^{(j)}(u q^{-\h}) = \varsigma^{(j-\h)}(u) \varsigma^{(\h)}(u q^{-j})$ due to~\eqref{eq:varsigj}.
	Then, taking the transpose of~\eqref{eq:reduc-gam} and replacing $u \rightarrow -u^{-1}$, $\varepsilon_\pm \rightarrow - \ov{\varepsilon}_\mp$, $k_\pm \rightarrow - \ov{k}_\mp $, from~\eqref{def:kp} and~\eqref{gammaKM} one has
	\begin{align} \nonumber
		&\lbrack\bar{\mathcal{E}}^{(j-\h)}_\fu\rbrack^t 
		K_1^{+(\h)}(u q^{j}) R^{(\h,j)}(u^{-2} q^{-j-\tha}) K_2^{+(j)}(u q^{-\h})  \lbrack\bar{\mathcal{F}}_\fu^{(j-\h)} \rbrack^t \\ \label{p1TT3}
		& = \frac{ f^{(j-\h)}(u q^{-1})}{f^{(j)}(u q^{-\h})} \left ( \displaystyle{\prod_{k=0}^{2j-2} } c(u^2 q^{2j-k}) c(u^2q^{2j-2-k}) \right ) \Gamma_{+}(u q^{j-1}) K^{+(j-\h)}(u q^{-1}) \ .
	\end{align}
	Finally, using~\eqref{eq:rel2p1}, the relations in~\eqref{usefulbarEFH}
	and the dual reflection equation, the l.h.s.\ of~\eqref{p1TT3} becomes:
	\begin{align*}
		&\underline{\bar{\mathcal{H}}^{(j-\h)} \bar{\mathcal{F}}^{(j-\h)} }	K_1^{+(\h)}(u q^{j}) R^{(\h,j)}(u^{-2} q^{-j-\tha}) K_2^{+(j)}(u q^{-\h})  \bar{\mathcal{E}}^{(j-\h)} \lbrack \bar{\mathcal{H}}^{(j-\h)} \rbrack^{-1} \\
		&=\bar{\mathcal{F}}^{(j-\h)} \underline{R^{(\h,j)}(q^{-j-\h})	K_1^{+(\h)}(u q^{j}) R^{(\h,j)}(u^{-2} q^{-j-\tha}) K_2^{+(j)}(u q^{-\h}) } \bar{\mathcal{E}}^{(j-\h)} \lbrack \bar{\mathcal{H}}^{(j-\h)} \rbrack^{-1}\\
		&=\bar{\mathcal{F}}^{(j-\h)} K_2^{+(j)}(u q^{-\h})	 R^{(\h,j)}(u^{-2} q^{-j-\tha}) K_1^{+(\h)}(u q^{j})  \underline{R^{(\h,j)}(q^{-j-\h})  \bar{\mathcal{E}}^{(j-\h)} } \lbrack \bar{\mathcal{H}}^{(j-\h)} \rbrack^{-1} \ ,
	\end{align*}
	and using~\eqref{usefulbarEFH}, the equation~\eqref{TT3} follows.
\end{proof}
\begin{lem} \label{lem:Kp}
	The following relation holds:
	\begin{align} \nonumber
		&K_2^{+(j-\h)} (u q^{-\h}) R^{(\h,j-\h)}(u^{-2} q^{-j-1}) K_1^{+(\h)}(u q^{j-\h})  \bar{\mathcal{E}}^{(j-1)}_\fu \\ \label{eq:bEFKP}
		&= \bar{\mathcal{E}}_\fu^{(j-1)} \bar{\mathcal{F}}_\fu^{(j-1)} K_2^{+(j-\h)} (u q^{-\h}) R^{(\h,j-\h)}(u^{-2} q^{-j-1}) K_1^{+(\h)}(u q^{j-\h})  \bar{\mathcal{E}}^{(j-1)}_\fu  \ .
	\end{align}
\end{lem}
\begin{proof}
	Inserting $ \bar{\mathcal{H}}^{(j-1)} \lbrack \bar{\mathcal{H}}^{(j-1)} \rbrack^{-1} = \mathbb{I}_{2j-1}$ in the l.h.s.\ of~\eqref{eq:bEFKP} and using~\eqref{usefulbarEFH}, one gets:
	\begin{align} \nonumber
		K_2^{+(j-\h)} (u q^{-\h})& R^{(\h,j-\h)}(u^{-2} q^{-j-1}) K_1^{+(\h)}(u q^{j-\h})  \bar{\mathcal{E}}^{(j-1)}_\fu   \bar{\mathcal{H}}^{(j-1)}  \lbrack \bar{\mathcal{H}}^{(j-1)} \rbrack^{-1}  \\ \label{eq:steplemKP}
		&=K_2^{+(j-\h)} (u q^{-\h}) R^{(\h,j-\h)}(u^{-2} q^{-j-1}) K_1^{+(\h)}(u q^{j-\h}) R^{(\h,j-\h)}(q^{-j}) \bar{\mathcal{E}}^{(j-1)}_\fu \lbrack \bar{\mathcal{H}}^{(j-1)} \rbrack^{-1} \ .
		\intertext{Then, using the dual reflection equation~\eqref{REKdual} and~\eqref{usefulbarEFH}, one finds that~\eqref{eq:steplemKP} equals}
		\bar{\mathcal{E}}^{(j-1)}_\fu& \bar{\mathcal{F}}^{(j-1)}_\fu R^{(\h,j-\h)}(q^{-j}) K_1^{+(\h)}(u q^{j-\h})  R^{(\h,j-\h)}(u^{-2} q^{-j-1})  K_2^{+(j-\h)} (u q^{-\h}) \bar{\mathcal{E}}^{(j-1)}_\fu \lbrack \bar{\mathcal{H}}^{(j-1)} \rbrack^{-1} . \nonumber
	\end{align}
	Finally, using again~\eqref{REKdual} and~\eqref{usefulbarEFH}, the equation~\eqref{eq:bEFKP} follows.
\end{proof}
\begin{proof}[Proof of Theorem \ref{TTrel}]
	Assuming Conjecture~\ref{conj1}, another expression for the fused K-operators that agrees with~\eqref{fused-K-op} is given by~\cite[Props.\,5.9\,\&\,6.13]{LBG}
	\begin{equation}\label{v2fused-K-op}
		\mathcal{K}^{(j)}(u) =  \mathcal{F}^{(j)}_\fu  \mathcal{K}_2^{(j-\h)}(u q^{-\h})R^{(\h,j-\h)}(u^2 q^{j-1}) \mathcal{K}_1^{(\h)}(u q^{j-\h})  \mathcal{E}^{(j)}_\fu \ .
	\end{equation}
	We underline the step of calculations that use~\eqref{rel1},~\eqref{rel2}, the reflection equation or the dual reflection equation.
	Inserting~\eqref{fused-KP} and~\eqref{v2fused-K-op} in~\eqref{tg}, $\bt^{(j)}(u)$ reads, up to an overall scalar factor, as follows: 
	\begin{align} \nonumber
		& \text{\normalfont{tr}} \Big ( \big [ \underline{\mathcal{E}^{(j)}_\fu } \big ]^t   K_2^{+(j-\h)} (u q^{-\h}) R^{(\h,j-\h)}(u^{-2} q^{-j-1}) K_1^{+(\h)}(u q^{j-\h}) \big [ \mathcal{F}^{(j)}_\fu \big ]^t  \\  \nonumber
		&  \times  \mathcal{F}^{(j)}_\fu  \mathcal{K}_2^{(j-\h)}(u q^{-\h})R^{(\h,j-\h)}(u^2 q^{j-1}) \mathcal{K}_1^{(\h)}(u q^{j-\h})  \underline{ \mathcal{E}^{(j)}_\fu } \Big ) \\ \nonumber
		&\overset{\eqref{rel2}}{=}  \frac{1}{\mathcal{H}_{1}^{(j)}} \text{\normalfont{tr}} \Big (   K_2^{+(j-\h)} (u q^{-\h}) R^{(\h,j-\h)}(u^{-2} q^{-j-1}) K_1^{+(\h)}(u q^{j-\h}) \big [ \mathcal{F}^{(j)}_\fu \big ]^t  \\  \nonumber
		&  \times  \mathcal{F}^{(j)}_\fu \underline{  \mathcal{K}_2^{(j-\h)}(u q^{-\h})R^{(\h,j-\h)}(u^2 q^{j-1})\mathcal{K}_1^{(\h)}(u q^{j-\h}) R^{(\h,j-\h)}(q^j) } \Big ) \\ \nonumber
		&\overset{\eqref{REKop}}{=} \frac{1}{\mathcal{H}_{1}^{(j)}} \text{\normalfont{tr}} \Big (   K_2^{+(j-\h)} (u q^{-\h}) R^{(\h,j-\h)}(u^{-2} q^{-j-1}) K_1^{+(\h)}(u q^{j-\h}) \underline{ \big [ \mathcal{F}^{(j)}_\fu \big ]^t } \\  \nonumber
		&  \times   \underline{\mathcal{F}^{(j)}_\fu R^{(\h,j-\h)}(q^j) } \mathcal{K}_1^{(\h)}(u q^{j-\h})  R^{(\h,j-\h)}(u^2 q^{j-1}) \mathcal{K}_2^{(j-\h)}(u q^{-\h}) \Big ) \\ \nonumber
		&\overset{\eqref{rel2}}{=}  \text{\normalfont{tr}} \Big (   K_2^{+(j-\h)} (u q^{-\h}) R^{(\h,j-\h)}(u^{-2} q^{-j-1}) K_1^{+(\h)}(u q^{j-\h})  \\  \nonumber
		&  \times   \underline{ \mathcal{E}_\fu^{(j)} \mathcal{F}_\fu^{(j)} } \mathcal{K}_1^{(\h)}(u q^{j-\h})  R^{(\h,j-\h)}(u^2 q^{j-1}) \mathcal{K}_2^{(j-\h)}(u q^{-\h}) \Big ) \ . 
	\end{align}
	 Then, from the second relation of (\ref{rel1}) one has $\mathcal{E}_\fu^{(j)} \mathcal{F}_\fu^{(j)} = \mathbb{I}_{4j} - \bar{\mathcal{E}}_\fu^{(j-1)} \bar{\mathcal{F}}_\fu^{(j-1)}$, and taking into account the overall factor,
		one gets:
		\begin{equation} \label{t-ab}
			\bt^{(j)}(u) = (a) +(b) \ ,
		\end{equation}
		where
		\begin{align*}
			(a)&=\frac{f^{(j-\h)}(u q^{-\h})}{f^{(j)}(u)} \text{\normalfont{tr}} \Big (   K_2^{+(j-\h)} (u q^{-\h}) R^{(\h,j-\h)}(u^{-2} q^{-j-1})\\   &\times K_1^{+(\h)}(u q^{j-\h})  
			\mathcal{K}_1^{(\h)}(u q^{j-\h})  R^{(\h,j-\h)}(u^2 q^{j-1}) \mathcal{K}_2^{(j-\h)}(u q^{-\h}) \Big ) \ , \\
			(b)&=- \frac{f^{(j-\h)}(u q^{-\h})}{f^{(j)}(u)} \text{\normalfont{tr}} \Big (  K_2^{+(j-\h)} (u q^{-\h}) R^{(\h,j-\h)}(u^{-2} q^{-j-1}) K_1^{+(\h)}(u q^{j-\h})   \\  
			& \times \bar{\mathcal{E}}^{(j-1)}_\fu \bar{\mathcal{F}}^{(j-1)}_\fu \mathcal{K}_1^{(\h)}(u q^{j-\h})  R^{(\h,j-\h)}(u^2 q^{j-1}) \mathcal{K}_2^{(j-\h)}(u q^{-\h}) \Big ) \ .
		\end{align*}
		Now, we rewrite the first term as follows: 
	\begin{align} \nonumber
		(a)&= \frac{f^{(j-\h)}(u q^{-\h})}{f^{(j)}(u)} \text{\normalfont{tr}} \Big (  \left [  K_2^{+(j-\h)} (u q^{-\h}) \big [R^{(\h,j-\h)}(u^{-2} q^{-j-1}) K_1^{+(\h)}(u q^{j-\h}) \big ]^{t_1} \right ]^{t_1} \\  \nonumber
		& \times \left [ \big[ \mathcal{K}_1^{(\h)}(u q^{j-\h})  R^{(\h,j-\h)}(u^2 q^{j-1}) \big ]^{t_1} \mathcal{K}_2^{(j-\h)}(u q^{-\h})  \right ]^{t_1}\Big ) \\  \nonumber
		&= \frac{f^{(j-\h)}(u q^{-\h})}{f^{(j)}(u)} \text{\normalfont{tr}} \Big (  \left [  K_2^{+(j-\h)} (u q^{-\h})  \big [ K_1^{+(\h)}(u q^{j-\h})  \big ]^{t_1} \big  [R^{(\h,j-\h)}(u^{-2} q^{-j-1}) \big ]^{t_1}  \right ]^{t_1} \\  \nonumber
		& \times \left [\big [  R^{(\h,j-\h)}(u^2 q^{j-1}) \big ]^{t_1}   \big[ \mathcal{K}_1^{(\h)}(u q^{j-\h}) \big ]^{t_1} \mathcal{K}_2^{(j-\h)}(u q^{-\h})  \right ]^{t_1}\Big ) \ .
		\intertext{ Then, apply $t_1$ inside the trace and use the cyclicity of the trace to get:} \nonumber
		(a)&= \frac{f^{(j-\h)}(u q^{-\h})}{f^{(j)}(u)} \text{\normalfont{tr}} \Big (  \big  [R^{(\h,j-\h)}(u^{-2} q^{-j-1}) \big ]^{t_1} \big [  R^{(\h,j-\h)}(u^2 q^{j-1}) \big ]^{t_1} \big[ \mathcal{K}_1^{(\h)}(u q^{j-\h}) \big ]^{t_1}  \\  \nonumber	
		& \times     \mathcal{K}_2^{(j-\h)}(u q^{-\h}) K_2^{+(j-\h)} (u q^{-\h})  \big [ K_1^{+(\h)}(u q^{j-\h})  \big ]^{t_1}     \Big )  \ ,
		\intertext{and using the crossing symmetry~\eqref{crossingJ} we obtain} \nonumber
		&(a)=\frac{f^{(j-\h)}(u q^{-\h})}{f^{(j)}(u)} \xi^{(j-\h)}(u^{-2}q^{-j-1}) \bt^{(j-\h)}(u q^{-\h}) \bt^{(\h)}(u q^{j-\h}) \\ \label{part-a}
		&\hphantom{(a)}= \bt^{(j-\h)}(u q^{-\h}) \bt^{(\h)}(u q^{j-\h}) \ .
	\end{align}
	The latter equality is obtained using the expressions of $\xi^{(j)}(u)$, $f^{(j)}(u)$ and $c(u)$ given in~\eqref{beta},\eqref{eq:fj},\eqref{eq:cu}, respectively.
 We now compute the second term in~\eqref{t-ab}. Using~\eqref{eq:bEFKP} and the cyclicity of the trace, it reads 
	\begin{align} \nonumber
		(b)&= - \frac{f^{(j-\h)}(u q^{-\h})}{f^{(j)}(u)}   \text{\normalfont{tr}} \Big ( \bar{\mathcal{F}}^{(j-1)}_\fu K_2^{+(j-\h)} (u q^{-\h}) R^{(\h,j-\h)}(u^{-2} q^{-j-1}) K_1^{+(\h)}(u q^{j-\h})   \\ \nonumber
		& \times \bar{\mathcal{E}}^{(j-1)}_\fu \bar{\mathcal{F}}^{(j-1)}_\fu \mathcal{K}_1^{(\h)}(u q^{j-\h})  R^{(\h,j-\h)}(u^2 q^{j-1}) \mathcal{K}_2^{(j-\h)}(u q^{-\h})\bar{\mathcal{E}}^{(j-1)}_\fu  \Big ) \ .
	\end{align}
	Finally, using the relations~\eqref{TT1} and~\eqref{TT3} we get
	\begin{align} \nonumber
		(b)&= - \frac{f^{(j-1)}(u q^{-1})}{f^{(j)}(u)}   \left( \displaystyle{ \prod_{k=0}^{2j-3}}  \displaystyle{ \prod_{\ell=1}^{4}} c(u^2 q^{2j-k-\ell})  \right) \Gamma (u q^{j-\tha}) \Gamma_+(uq^{j-\tha}) \bt^{(j-1)}(u q^{-1}) \\ \label{part-b}
		&=  \frac{\Gamma (u q^{j-\tha}) \Gamma_+ (uq^{j-\tha})}{ c(u^2 q^{2j}) c(u^2 q^{2j-2}) } \bt^{(j-1)}(u q^{-1}) \ .
	\end{align}
	Therefore, inserting~\eqref{part-a} and~\eqref{part-b} in~\eqref{t-ab}, we obtain \eqref{TT-rel}.
\end{proof}
	\section{The representation map $\psi^{(N)}_{\bj,\bar{v}}$}	\label{apD}
Here, we describe 	the spin-chain representation map $\psi^{(N)}_{\bj,\bar{v}}
 \colon \mathcal{A}^{(N)}_q \rightarrow \End(\mathbb{C}^{2j_N+1} \otimes \ldots \otimes \mathbb{C}^{2j_1+1}) $
	defined by \eqref{eq:dressKH}.
 Recall the matrices $S_\pm$, $S_3$ as defined in~\eqref{Bdef}. 
	
According to the ordering of tensor factors in the   vector space $\mathbb{C}^{2j_N+1}\otimes \ldots \otimes \mathbb{C}^{2j_1+1}$, we have for $1\leq k \leq N-1$ \cite{BK07}:
	\begin{align}
		\calW_{-k}^{(N)}&=\frac{(w_0^{(j_N)}-(q+q^{-1})q^{S_3})}{(q+q^{-1})}\otimes
		\calW_{k}^{(N-1)}
		-\frac{v_N^2+v_N^{-2}}{(q+q^{-1})} 1 \otimes \calW_{-k+1}^{(N-1)} +\ \frac{(v_N^2+v_N^{-2})w_0^{(j_N)}}{(q+q^{-1})^2}\calW_{-k+1}^{(N)}
		\nonumber\\
		& + \ \frac{(q-q^{-1})}{k_+k_-(q+q^{-1})^2}
		\left(k_+ v_N q^{1/2}S_+q^{S_3/2}\otimes
		\calG_{k}^{(N-1)}+k_- v_N^{-1} q^{-1/2}S_-q^{S_3/2}\otimes {\tilde \calG}_{k}^{(N-1)}\right)\nonumber\\
		&  +\ q^{S_3}\otimes \calW_{-k}^{(N-1)}
		\ ,\nonumber\\
		\calW_{k+1}^{(N)}&=\frac{(w_0^{(j_N)}-(q+q^{-1})q^{-S_3})}{(q+q^{-1})}\otimes
		\calW_{-k+1}^{(N-1)}
		-\frac{v_N^2+v_N^{-2}}{(q+q^{-1})}1\otimes \calW_{k}^{(N-1)} +\ \frac{(v_N^2+v_N^{-2})w_0^{(j_N)}}{(q+q^{-1})^2}\calW_{k}^{(N)}
		\nonumber\\
		& +\ \frac{(q-q^{-1})}{k_+k_-(q+q^{-1})^2}
		\left(k_+ v_N^{-1} q^{-1/2}S_+q^{-S_3/2}\otimes
		\calG_{k}^{(N-1)}+k_- v_N q^{1/2}S_-q^{-S_3/2}\otimes {\tilde \calG}_{k}^{(N-1)}\right)\nonumber\\
		&  +\ q^{-S_3}\otimes \calW_{k+1}^{(N-1)}
		\ ,\nonumber\\
		\calG_{k+1}^{(N)}&=
		\frac{k_-(q-q^{-1})^2}{k_+(q+q^{-1})}
		S_-^2\otimes {\tilde \calG}_{k}^{(N-1)}
		-\frac{1}{(q+q^{-1})}(v_N^2 q^{S_3}+ v_N^{-2} q^{-S_3})\otimes \calG_{k}^{(N-1)} 
		+1 \otimes \calG_{k+1}^{(N-1)}\nonumber\\
		&+ (q^2-q^{-2})\left(
		k_- v_N q^{-1/2}S_-q^{S_3/2}\otimes \big(\calW_{-k}^{(N-1)}\!\!-\!\calW_{k}^{(N-1)}\big)
		+k_- v_N^{-1} q^{1/2}S_-q^{-S_3/2}\otimes \big(\calW_{k+1}^{(N-1)}\!\!-\!\calW_{-k+1}^{(N-1)}\big)\!
		\right)\nonumber\\
		&+\frac{(v_N^2+v_N^{-2})w_0^{(j_N)}}{(q+q^{-1})^2}\calG_{k}^{(N)}\ ,\nonumber\\
		{\tilde\calG}_{k+1}^{(N)}&=
		\frac{k_+(q-q^{-1})^2}{k_-(q+q^{-1})}
		S_+^2\otimes {\calG}_{k}^{(N-1)}
		-\frac{1}{(q+q^{-1})}(v_N^2 q^{-S_3}+v_N^{-2}q^{S_3})\otimes {\tilde\calG}_{k}^{(N-1)} 
		+ 1 \otimes {\tilde\calG}_{k+1}^{(N-1)}\nonumber\\
		& + (q^2-q^{-2})\left(
		k_+ v_N^{-1} q^{1/2}S_+q^{S_3/2}\otimes \big(\calW_{-k}^{(N-1)}\!\!-\!\calW_{k}^{(N-1)}\big)
		+k_+ v_N q^{-1/2}S_+q^{-S_3/2}\otimes \big(\calW_{k+1}^{(N-1)}\!\!-\!\calW_{-k+1}^{(N-1)}\big)\!
		\right)\nonumber\\
		&+\frac{(v_N^2+v_N^{-2})w_0^{(j_N)}}{(q+q^{-1})^2}{\tilde\calG}_{k}^{(N)}\ , \nonumber
	\end{align}
	where
	\begin{equation*}
		w_0^{(j_n)}=q^{2j_n+1}+q^{-2j_n-1} \ .
	\end{equation*}
	For the special case $k=0$, we set\footnote{Although the notation is ambiguous, the reader must keep in mind that $\calW_k^{(N)}
|_{k=0} \neq\calW_{-k}^{(N)}|_{k=0}$ and $\calW_{-k+1}^{(N)}
|_{k=0} \neq\calW_{k+1}^{(N)}|_{k=0}$
for any $N$.}
	\begin{equation*}
		\calW_k^{(N)}\rvert_{k=0} \equiv 0\ , \qquad \calW_{-k+1}^{(N)}\rvert_{k=0} \equiv 0 \ , \qquad \calG_k^{(N)}\rvert_{k=0} = \tilde{\calG}_k^{(N)}\rvert_{k=0} \equiv  \frac{k_+ k_- (q+q^{-1})^2}{q-q^{-1}} {\mathbb I}^{(N)} \ ,
	\end{equation*}
	and also  the initial c-number conditions
	\begin{equation}\label{W0-init-cond}
		\calW_0^{(0)} \equiv \varepsilon_+^{(0)}, \qquad \calW_1^{(0)} \equiv \varepsilon_-^{(0)}, \qquad \calG_1^{(0)} = \tilde{\calG}_1^{(0)} \equiv \varepsilon_+^{(0)} \varepsilon_-^{(0)}(q-q^{-1}) \ .
	\end{equation}

    \section{Hamiltonians of the open XXZ spin-$\h$ chain}\label{ApE}
    In this appendix, we compute explicitly the two first local conserved quantities~\eqref{H1fromT}, ~\eqref{H2fromT} for the open XXZ spin-$\h$ chain.

\underline{The open XXZ Hamiltonian:} We begin with
\beqa
{\cal H}^{(1)}= \tilde {\boldsymbol  t}^{\h,\bar{\h}}(1)^{-1} \frac{d}{du}\tilde  {\boldsymbol  t}^{\h,\bar{\h}}(u)|_{u=1}\ \label{H1}
\eeqa
where the expression for $\tilde{{\boldsymbol  t}}^{\h,\bar{\h}}(1)$ is obtained from~\eqref{tildetjs}. 
By eq.~\eqref{eq:R-P}, we have $\tilde R^{(\h,\h)}_{0n}(1)=\cal P^{(\h,\h)}_{0n}$. Also, recall the  expression of the spin-$\h$ normalized K-matrix~\eqref{renormK} with~\eqref{KM-spin1/2}, and the dual normalized K-matrix~\eqref{tildeKplus}. With this, we obtain
\beqa
\tilde {\boldsymbol  t}^{\h,\bar{\h}}(1) &=&  \normalfont{\text{tr}}_{{\mathbb C}^{2}}(\tilde K^{+(\h)}(1))\tilde K^{(\h)}(1)\ \label{tm11} \\
&=& (q+q^{-1})(\varepsilon_+ + \varepsilon_-)(\bar\varepsilon_+ + \bar\varepsilon_-)\ .\nonumber
\eeqa
Furthermore, we calculate the derivation at $u=1$, see also~\cite{CLSW02}:
\begin{align} \label{eq:tprime}
\frac{d}{du}\tilde {\boldsymbol  t}^{\h,\bar{\h}}(u)|_{u=1}&=\left (\tilde K^{(\h)}_1(1) \tr(\tilde K^{+(\h)'}(1)) +   \tilde K^{(\h)'}_1(1) \tr(\tilde K^{+(\h)}(1)) \right ) \\ \nonumber
& + c(q)^{-1} \tr(\tilde K^{+(\h)}(1))  \tilde K^{(\h)}_1(1) \Big (H_{21} + 2 \displaystyle{\sum_{k=2}^{N-1}} H_{k+1\,k}   \Big)  \\ \nonumber
& +c(q)^{-1} \Big (  \tr(\tilde K^{+(\h)}(1)) H_{21} \tilde K^{(\h)}_1(1) +2 \tilde K^{(\h)}_1(1) \tr_a(\tilde K_a^{+(\h)}(1) H_{aN}) \Big ) \
\end{align}
where we use the short-hand notation   $\tilde K^{(\h)'}(1) = \frac{d\tilde K^{(\h)}(u)}{du}|_{u=1}$ and we denote
 \begin{align}\label{eq:localham}
c(q)^{-1}H_{k+1\ k} = \mathcal{P}_{k+1\ k} {\tilde R^{(\h,\h)'}_{k+1\ k}}(1)   
= c(q)^{-1}\left (  \sigma_{k+1}^x\sigma_{k}^x+\sigma_{k+1}^y\sigma_{k}^y + \frac{q+q^{-1}}{2} (\sigma_{k+1}^z\sigma_{k}^z - \mathbb{I}) \right ) \ .
\end{align}
Combining~\eqref{tm11} and~\eqref{eq:tprime} according to~\eqref{H1} yields to the standard expression for the Hamiltonian of the open XXZ spin-$\h$ chain given by~\eqref{H1fromYB}.

Then, using Proposition~\ref{prop:tIN} we find that
\beqa
\frac{d}{du}\tilde {\boldsymbol  t}^{\h,\bar{\h}}(u)|_{u=1}&=&4c(q^2)c(q)^{-2N}\left( \psi^{(N)}_{\bar{\h},\bar{1}}(\mathsf{I}^{(N)}(1)) + h_0^{(N)}(1) \mathsf{I}_0 \right ) \label{dertOq}\\  
&& +  2c(q)^{-2N+1}\left(\overline{\varepsilon}_+  \varepsilon_+^{(N)} +  \overline{\varepsilon}_-\varepsilon_-^{(N)} \right) -2N\frac{(q+q^{-1})^2}{c(q)}\left(\overline{\varepsilon}_+  +  \overline{\varepsilon}_-\right)\left(  \varepsilon_++ \varepsilon_- \right)\ ,
\nonumber
\eeqa
with $\mathsf{I}_0$ defined in~\eqref{I0}.
Combining~\eqref{tm11} and~\eqref{dertOq} we get the expression of ${\cal H}^{(1)}$ in~\eqref{H1XXZ}.\vspace{1mm}

\underline{Higher Hamiltonians:} We now consider  the quantity in~\eqref{H2fromT}:   
\begin{equation}
    \tilde {\boldsymbol  t}^{\h,\bar{\h}}(1)^{-1} \frac{d^2}{du^2}\tilde {\boldsymbol  t}^{\h,\bar{\h}}(u)|_{u=1} - \tilde {\boldsymbol  t}^{\h,\bar{\h}}(1)^{-2}\left(\frac{d}{du}\tilde {\boldsymbol  t}^{\h,\bar{\h}}(u)|_{u=1}\right)^2 
    .\label{H2}
\end{equation}
Inserting~\eqref{tm11}, \eqref{dertOq} and
\beqa
\frac{d^2}{du^2}\tilde {\boldsymbol  t}^{\h,\bar{\h}}(u)|_{u=1} = -4c(q^2) c(q)^{-2N}\psi^{(N)}_{\bar{\h},\bar{1}}(\mathsf{I}^{(N)}(1)) + 8\frac{d}{du}\left( c(u^2q^2)c(uq)^{-2N}\psi^{(N)}_{\bar{\h},\bar{1}}\bigl(\mathsf{I}^{(N)}(u)\bigr) \right)\Big|_{u=1}
\eeqa
in~\eqref{H2}, we get the higher Hamiltonian ${\cal H}^{(2)}$ in~\eqref{H2XXZ}.
\end{appendix}

	\end{document}